\documentclass[10pt,reqno]{amsart}
\usepackage{amssymb,amsmath,amsthm,andy,fullpage}

\newtheorem{thm}{Theorem}[section]
\newtheorem{prop}[thm]{Proposition}
\newtheorem{lem}[thm]{Lemma}
\newtheorem{cor}[thm]{Corollary}

\theoremstyle{definition}
\newtheorem{defn}[thm]{Definition}

\theoremstyle{remark}
\newtheorem{remark}[thm]{Remark}

\newtheorem{rem}[thm]{Remark}

\newcommand{\tnabla}{\widetilde\nabla}
\renewcommand{\I}{\vartheta}

\begin{document} 

\title {Heat Equations and the Weighted $\dbar$-Problem}
\author{Andrew Raich}%

\thanks{The author is partially supported by NSF grant DMS-0855822}

\address{Department of Mathematical Sciences, SCEN 327, 1 University of Arkansas, Fayetteville, AR 72701}
\email{araich@uark.edu}



\begin{abstract}
The purpose of this article is to establish regularity and pointwise upper bounds for the (relative) fundamental solution 
of the heat equation
associated to the weighted $\dbar$-operator in $L^2(\C^n)$ for a certain class of weights.
The weights depend on a parameter, and we find pointwise
bounds for heat kernel, as well as its derivatives in time, space, and the parameter.
We also prove cancellation conditions for the heat semigroup.
We reduce the $n$-dimensional case to the one-dimensional case, and 
the estimates in one-dimensional case are achieved by 
Duhamel's principle and commutator properties of the operators. As an application, we recover estimates 
of heat kernels on polynomial models in $\C^2$.
\end{abstract}

\maketitle

%%%%%%%%%%%%%%%%%%%%%%%%%%%%%%%5
%
%        SECTION: INTRODUCTION
%
%%%%%%%%%%%%%%%%%%%%%%%%%%%%%%55
\section{Introduction}\label{sec:intro}

The purpose of this article is to establish regularity and pointwise upper bounds for the (relative) fundamental solution 
of the heat equation
associated to the weighted $\dbar$-operator in $L^2(\C^n)$ for a certain class of weights.
The weights depend on a parameter, and we find bounds for heat kernel, as well as its derivatives in time, space, and the parameter.
We also prove cancellation conditions for the heat semigroup.
We reduce the $n$-dimensional case to the one-dimensional case, and 
the estimates in one-dimensional case will be achieved by a novel use of
Duhamel's principle.

As an application of our estimates, we can recover and improve the 
estimates by Nagel and Stein in \cite{NaSt01h} for heat kernels on polynomial models in $\C^2$.

An additional point of interest is that the infinitesimal generator of the semigroup
is a magnetic Schr\"odinger operator. Also, when the parameter is negative, 
this Schr\"odinger operator has a nonpositive and possibly unbounded electric potential,
yet the large-time behavior of the semigroup is well-controlled (the eigenvalues of the
generator are always nonnegative).

%%%%%%%%%%%%%%%%%%%%%%%%%%%%%%%%%%%%%%
%
%  subsection: the n-dimensional case
\subsection{The set-up in one dimension -- heat equations in $(0,\infty)\times\C$}\label{subsection:heat equation in C}
Let $p:\C\to\R$ be a subharmonic, nonharmonic polynomial and $\tau\in\R$ a parameter. Define
\[
\Zbstpz = \frac{\p}{\p\z} + \tau \frac{\p p}{\p\z},
\]
a one-parameter family of differential operators acting on functions defined on $\C$. 
To solve the Cauchy-Riemann equations
\[
\dbar u=f
\]
for a function $f\in 
L^2(\C,e^{-2\tau p}) = \{ \vp : \int_{\C} |\vp|^2 e^{-2\tau p}\, dV <\infty\}$, 
it is equivalent to solve the weighted $\dbar$-problem
\[
\Zbstpz \alpha=\beta.
\]
Our interest is studying the $\Zbstpz$-problem through its associated heat equation (defined below). We wish to express
the solution as an integral operator and finding the regularity and smoothness of the (relative) fundamental solution.

To study the $\Zbstp$-equation, we let 
\[
\Zstpz = -\Zbstpz^* = e^{\tau p}\frac{\p}{\p\z} e^{- \tau p} = \frac{\p}{\p z} - \tau\frac{\p p}{\p z}
\]
introduce the $\Zbstp$-Laplacian
\[
\Boxtpz =-\Zbstpz\Zstpz.
\]
As observed in \cite{Rai07}, if $\tilde p(x_1,x_2) = p(-x_1,-x_2)$ and $\Boxwtpz = -\Zstpz\Zbstpz$, 
then 
\[
\Boxw_{\tau\tilde p,z} = \Box_{(-\tau)p,z}.
\] 
Thus, we will (mostly) restrict ourselves to the case $\tau>0$ 
and study the heat equations
\begin{equation}\label{eq:he i1}
\begin{cases} {\displaystyle \frac{\p u}{\p s}} + \Boxtpz u=0\vspace*{.1in}\\
u(0,z) = f(z)\end{cases}
\end{equation}
and
\begin{equation}\label{eq:he2}
\begin{cases} {\displaystyle \frac{\p \tilde u}{\p s}} + \Boxwtpz \tilde u=0\vspace*{.1in}\\
\tilde u(0,z) = \tilde f(z).\end{cases}
\end{equation}	
We write $\Boxtp$ in lieu of $\Boxtpz$ when the application is clear. 
From \cite{Rai06h, Rai07}, $\Boxtp$ and $\Boxwtp$ are self-adjoint,
so from the spectral theorem and the Riesz Representation Theorem,
we can express our solutions 
\begin{align}
u(s,z) &= e^{-s\Boxtp}[f](z) = \int_{\C}\Htp(s,z,w)f(w)\, dA(w) \label{eq:he I}\\
\intertext{and}
\tilde u(s,z) &= e^{-s\Boxwtp}[f](z) = \int_{\C}\Hwtp(s,z,w)f(w)\, dA(w), \label{eq:hew I}
\end{align}
where $\Htp(s,z,w)$ and $\Hwtp(s,z,w)$ are $C^\infty$ away from $\{(s,z,w): s=0\text{ and }z=w\}$
(assuming $s\geq 0$, of course) and $dA$ is Lebesgue measure on $\C$.

%%%%%%%%%%%%%%%%%%%%%%%%%%%%%%%%
%
%
%   	SECTION: BACKGROUND
%
%%%%%%%%%%%%%%%%%%%%%%%%%%%%%%%%%%%5
\section{Discussion of the $n=1$ case}\label{sec:background}

We have several motivations for studying the heat equations (\ref{eq:he i1}) and (\ref{eq:he2}).

\subsection{$e^{-s\Box_b}$ on polynomial models in $\C^2$} \label{subsec:heat kernels on poly models in C2}
A polynomial model $M_p$ in $\C^2$ is a CR-manifold of the form
\[
M_p = \{ (z,w)\in\C^2 : \Imm w = p(z) \}
\]
where $p$ is subharmonic, nonharmonic polynomial. $M_p$ is the boundary of an unbounded pseudoconvex domain.
When $p(z) = |z|^2$, $M_p$ is the Heisenberg group $\mathbb{H}^1$, thus we can consider the Heisenberg group as the simplest example of a polynomial model. 

The analysis of polynomial models in $\C^2$ is directly related to the $\Zbstp$-problem. $M_p \cong \C\times \R$, and if $w = t+ip(z)$, then
$\dbar_b$ on $M$ can be identified with the vector field
$\LL = \frac{\p}{\p\z} - i\frac{\p p}{\p\z}\frac{\p}{\p t}$. $\LL$ is translation invariant in $t$, so if we take the partial Fourier transform in $t$ where $\tau$ is the transform variable
of $t$, then  $\LL \mapsto^{\hspace{-.075in}\widehat{}} \hspace{.05in}\Zbstp$. Consequently, $\Boxtp$ and $\Boxwtp$ are the partial Fourier transforms of 
$\Box_b$ on (0,1)-forms and functions, respectively.

In $\C^2$, the $\Box_b$-heat equation on the 
polynomial model $M_p$ was solved by Nagel and Stein \cite{NaSt01h}. 
They prove that the heat kernel of $e^{-s\Box_b}$  satisfies rapid decay.
In Section \ref{sec:heat kernel estimates on poly models}, we improve their heat kernel estimates to Gaussian estimates in $|z-w|$.
Street \cite{Stre08} has shown Gaussian decay with respect to the control metric for the $\Box_b$-heat kernel. His method does not seem to
generalize to the $n\geq 2$ case, however, while ours ought to, as noted in Section \ref{sec:heat equation in Cn}.

In light of the Nagel-Stein result and the connection of $\Htp(s,z,w)$ and $\Hwtp(s,z,w)$ 
with the $\Box_b$-heat kernel on polynomial models on $\C^2$, it follows that
$\Htp(s,z,w)$ and $\Htp(s,z,w)$ are actually $C^\infty$ off the diagonal
$\{(s,z,w,\tau): s=0,\ z=w\}$. Thus, the content of this article is to prove the
decay estimates and cancellation conditions.

\subsection{$\Boxtp$ and $\Boxwtp$ as magnetic Schr\"odinger operators}
If $a = \tau(-\frac{\p p}{\p x_2}, \frac{\p p}{\p x_1})$ and $V = \frac \tau 2\triangle p$,
then
\[
2\Boxtp = \frac 12 (i\nabla -a)^2 + V, \qquad
2\Boxwtp = \frac 12 (i\nabla -a)^2 - V
\]
magnetic Schr\"odinger operators with magnetic potential $a$ and electric
potential $\pm V$.

The operators $\Boxtp$ and $\Boxwtp$ behave quite differently. As discussed in
\cite{Rai07}, $\Boxwtp$ has a 
nonpositive and unbounded potential (unbounded if $\deg\triangle p\geq 1$). A further complication is that for $\tau>0$,
$\Null(\Boxwtp)\neq \{0\}$ (and in fact may be infinite dimensional, see \cite{Christ91}) 
while $\Null(\Boxtp) =\{0\}$. Fortunately,  $\Boxwtp$ has nonnegative eigenvalues
and is self-adjoint, so it follows from the spectral theorem that
$\lim_{s\to\infty}e^{-s\Boxwtp} = \Ss$ where $\Ss$ is the Szeg\"o 
projection, i.e., the projection of $L^2(\C)$ onto $\Null\Zbstp$.
The Szeg\"o projection is given by
\[
\Stp[f](z) = \int_\C \Stp(z,w) f(w)\, dA(w)
\]
and  $\Stp(z,w)\in C^\infty(\C)$ \cite{Christ91}.

A consequence of the nonzero limit is that the kernel of $e^{-s\Boxwtp}$ cannot vanish
as $s\to\infty$. Thus, $\int_{0}^\infty e^{-s\Boxwtp}\,ds$ diverges and cannot be the
relative fundamental solution of $\Boxwtp$. $e^{-s\Boxwtp}(I-\Ss)$ functions as the
natural replacement for $e^{-s\Boxwtp}$ since $\int_{0}^\infty e^{-s\Boxwtp}(I-\Ss)\, ds$
converges and equals the relative fundamental solution of $\Boxwtp$.

\subsection{$\Zbstp$, $\Boxtp$, and the weighted $\dbar$-operator in $\C$}
In \cite{Christ91}, Christ studies the $\dbar$-problem in $L^2(\C,e^{-2p})$ by
the $\Zbstp$-problem when $\tau=1$. 
Christ solves the $\Box_p$-equation using methods different from ours.
He shows that $\Box_p$ is invertible and that the solution can be written
as integration against a fractional integral operator. He 
finds pointwise upper bounds on the integral kernel and related objects. 
For more background on the weighted $\dbar$-problem in $\C$, see
\cite{Rai07, Christ91, Ber96, FoSi91}

\subsection{$\Boxtp$ and Hartogs Domains in $\C^2$}
Mathematicians have analyzed operators on Hartogs domains in $\C^n$ by understanding
weighted operators on their base spaces. The original operators are then reconstructed using 
Fourier series \cite{Li89, FoSi91, Ber94}.
Recently, on a class of Hartogs domains $\Omega\subset\C^2$,
Fu and Straube \cite{FuSt02,FuSt04} establish an equivalence between the compactness of
the $\dbar$-Neumann problem and the blowup of the smallest eigenvalue of
$\Boxtp$ as $\tau\to\infty$. Christ and Fu \cite{ChFu05} build on 
the work of Fu and Straube to show that the following
are equivalent: compactness of the $\dbar$-Neumann operator, compactness of
the complex Green operator, and $b\Omega$ satisfying property $(P)$.

%%%%%%%%%%%%%%%%%%%%%%%%%%%%%%%%%%%%%%
%
%  Section: the n-dimensional case
%
%%%%%%%%%%%%%%%%%%%%%%%%%%%%%%%%%%%%%%%
\section{The $n\geq 2$ case and its reduction to $n=1$}\label{sec:heat equation in Cn}
% subsection: the Box_D heat equation
\subsection{The $\Box_{\D}$ and $\Box_{\D}$-heat equations}
In $\C^n$, $n\geq 2$, the Cauchy-Riemann equations in weighted spaces take the form
\[
\dbar u=f
\]
where $f$ is a  $(0,q+1)$-form in $L^2(\C^n,e^{-2\lam}) = \{ \vp : \int_{\C^n} |\vp|^2 e^{-2\lam}\, dV <\infty\}$. We can solve the weighted $\dbar$-equation by solving the 
equivalent unweighted problem  
\[
\D \alpha=\beta
\]
where $\D = e^{-\lam}\dbar e^{\lam}$. Solving the $\dbar$-equation by solving a related weighted problem is a classical
technique that goes back to H\"ormander \cite{Hor65}. Our interest is not simply in solving the $\D$-problem, but expressing
the solution as an integral operator and finding the regularity and smoothness of the (relative) fundamental solution.
We work with the class of weights $\lam= \tau P(z_1,\dots,z_n) = \tau \sum_{j=1}^n p_j(z_j)$ where $\tau\in\R$ is a parameter
and $p_j$ are subharmonic, nonharmonic polynomials. We call such polynomials $P$ decoupled. For the remainder of the section,
we assume that $\D$ is of the form
\[
\D_{\tau P} = \D = e^{-\tau P}\dbar e^{\tau P}
\]
where $P$ is a decoupled polynomial and each $p_k$ is subharmonic and nonharmonic.

To study the $\D$-equation, we let $\Ds$ be the $L^2$-adjoint of $\D$ and
introduce the $\D$-Laplacian
\[
\Box_{\D}^q = \Box_{\D} = \Ds\D+\D\Ds
\]
on $(0,q)$-forms.
We will study the $\Box_{\D}$-heat equation
\begin{equation} \label{eqn:BoxD heat equation}
\begin{cases} \frac{\p u}{\p s} + \Box_{\D} u =0 \\
u(0,z) = f(z) \end{cases}
\end{equation}
because (as discussed below) we can recover both the solution to $\Box_{\D}$-equation, $\Box_{\D} \alpha = \beta$, and the projection
onto the null-space of $\Box_{\D}$. We call the projection the Szeg\"o projection and denote it $S_{\Box_{\D}}$.
Our goal is to find the (relative) fundamental solution to the $\Box_{\D}$-heat equation and express the solution 
\[
u(s,z) = \int_{\C^n} H_{\tau P}(s,z,w) f(w) \, dV(w)
\]
where $H_{\tau P}(s,z,w)$ is the relative fundamental solution, hereafter called the $\Box_{\D}$-heat kernel. We wish 
to find the regularity and pointwise upper bounds of the $\Box_{\D}$-heat kernel and its derivatives in 
time ($s$), space ($z$ and $w$), and the parameter $(\tau)$.

%%%%%%%%%%%
%
%	subsection: connection with Box_b on decoupled polynomial models
%
%%%%%%%%%%%%%%%%%%%
\subsection{$\Box_b$ on decoupled polynomial models in $\C^n$}
In \cite{NaSt06}, Nagel and Stein find optimal estimates for solutions to the Kohn Laplacian $\Box_b$ on a class of models
in $\C^n$.
They study hypersurfaces of the form
$M_P = \{(z_1,\dots,z_{n+1})\in\C^{n+1}: \Imm z_{n+1} = P(z_1,\cdots, z_{n})\}$
where $P(z_1,\dots,z_{n}) =p_1(z_1)+\cdots + p_{n-1}(z_{n-1})$ and
$p_j$ are subharmonic, nonharmonic polynomials. 
The surface $M_P$  is the boundary of an unbounded pseudoconvex domain and is called a decoupled polynomial model. 
For example, if $p(z) = |z_1|^2+ \cdots+|z_n|^2$, then $M_p$ is the Heisenberg group $\mathbb{H}^n$ and is the 
boundary of the Siegel upper half space, and in \cite{BoRa09}, we use Hermite functions to explicitly compute the Fourier transform of the $\Box_b$-heat kernel as well
as the $\Box_{\D}$-heat kernel. In \cite{BoRa10}, Boggess and Raich 
present a calculation of the Fourier transform of the fundamental solution of the $\Box_b$-heat equation
on quadric submanifolds $M\subset \C^n\times\C^m$. In particular, we compute the analog of the $\Box_{\D}$-heat kernel.

In analog to polynomial models in $\C^2$, the $M_P \cong \C^n\times \R$ since points in
$M_P$ are of the form $(z_1,\dots,z_n, t +iP(z))$. 
The realization of $\dbar_b$ (defined on $M_p$) on $\C^n\times\R$ is a translation invariant operator in $t$.
The partial Fourier transform in $t$ of the Kohn Laplacian $\Box_b$ (identified with its image on $\C^n\times\R$)
is the $\D$-Laplacian. Thus, the $\Box_{\D}$-heat kernel is the partial Fourier transform of the $\Box_b$-kernel.

We would like to study the heat semigroup $e^{-s\Box_b}$ on $M_P$. A motivation
for studying the heat kernel is that one of the most important aspects of \cite{NaSt06} 
is that their qualitatively sharp estimates for $\Box_b$ are written in
terms of both the control metric and the Szeg\"o pseudometric
(see \cite{NaStWa85, BrNaSt88, McN92} for background on these metrics). 
We would like to understand the appearance of both metrics in the estimates.

%%%%%%%%
%
%	subsection: reduction to the n=1 case
\subsection{Reduction from the case $n\geq 2$ to $n=1$}
The reason that the $n\geq 2$ and $n=1$ cases can be studied together is that the decoupling of the polynomial $P$ allows the 
$\Box_{\D}$-heat kernel to be expressed as
a product of the $\Boxtp$- and $\Boxwtp$-heat kernels. Consequently, the $n$-dimensional problem reduces to a
one-dimensional problem. This method of expressing
an $n$-dimensional heat kernels as a product of one-dimensional heat kernels was used by Boggess and Raich to present a
simplified calculation of the heat kernel on the Heisenberg group \cite{BoRa09}. 

To see the factoring, we need to compute
$\Box_{\D}$.
Let $\I_q$ be the set of increasing $q$-tuples $I=(i_1,\dots,i_q)$ with $1\leq i_1<\cdots<i_q\leq n$.
For $J\in\I_q$, let
\[
\Box_k^{J(k)} = \begin{cases}\Box_{\tau p_k}  & k\in J \\ \Boxw_{\tau p_k}  & k\not\in J\end{cases},
\]
and
\[
\Box_J = \sum_{k=1}^n \Box_k^{J(k)}.
\]
If $f = \sum_{J\in\I_q} f_J\, d\z_J$, then since $Z_{\tau p_j}$ commutes with $\ZZ_{\tau p_k}$ (here $p_j = p_j(z_j)$ and
$p_k = p_k(z_k)$) if $j\neq k$, a standard computation shows
\[
\Box_{\D} f = \sum_{J\in\I_q} \Box_J f_J\, d\z_J.
\]
Since $\Box_{\D}$ is self-adjoint, we can solve the $\Box_{\D}$-heat equation via the spectral theorem, i.e.,
$u(s,z) = e^{-s\Box_{\D}}[f](z)$. Since $\Box_{\D}$ acts diagonally, it is enough to study $\Box_J$ acting on functions.
$\Box_J$ is a sum of commuting operators $\Box_k^{J(k)}$, so
\[
e^{-s\Box_J} = e^{-s\sum_{k=1}^n \Box_k^{J(k)}} = \prod_{k=1}^n e^{-s\Box_k^{J(k)}}.
\]
Each $e^{-s\Box_k^{J(k)}}$ acts only in the $z_k$ variable, and consequence of this fact is that
the integral kernel of the product has a simpler form than (the analog of) a convolution. 
If $H_J(s,z,w)$ is the heat-kernel to the $\Box_J$-heat equation, then 
\[
H_J(s,z,w) = \prod_{k=1}^n H_k^{J(k)}(s,z_k,w_k)
\]
where 
\[
H_k^{J(k)}(s,z_k,w_k) = \begin{cases} H_{\tau p_k}(s,z_k,w_k)  &\text{if }\Box_k^{J(k)} = \Box_{\tau p_k} \\
\tilde H_{\tau p_k}(s,z_k,w_k)  &\text{if }\Box_k^{J(k)} = \Boxw_{\tau p_k} \end{cases}.
\]

We would like to thank Peter Kuchment, Emil Straube, and Alex Nagel
for their support and encouragement. We would also like to thank Al  Boggess 
for his helpful comments regarding the writing of this article.

%%%%%%%%%%%%%%%%%%%%%%%%%%%
%
%
%	SECTION: RESULTS
%
%
%%%%%%%%%%%%%%%%%%%%%%%%%%%5
\section{Results}\label{sec:results}
Since the $n$-dimensional heat kernel can be written in terms of one-dimensional heat kernels, for simplicity we will
write all of our results in terms of the one-dimensional kernels.

In \cite{Rai06h, Rai07}, we establish pointwise upper bounds for $\Htp(s,z,w)$
and $\Hwtp(s,z,w)$ and their space and time derivatives. 
The spectral theorem techniques in our earlier work
are poorly suited for differentiating in the parameter, so we develop
a new integral formula based on Duhamel's principle
to handle derivatives in the parameter. The results
proven in this article greatly extend our earlier results.

In order to write the estimates for $\Htp(s,z,w)$ and $\Hwtp(s,z,w)$, we need the appropriate
differential operators. Since $\Boxtp$ is a self-adjoint operator in $L^2(\C)$, it follows
that $\Htp(s,z,w) = \overline{\Htp(s,w,z)}$ \cite{Rai06h}. 
Thus, in $w$, the appropriate differential operators are:
\begin{align*}
\Wbstpw &= \overline{(\Zstpw)}=\frac{\p}{\p \w} 
-  \tau \frac{\p p}{\p \w}= e^{\tau p}\frac{\p p}{\p \w}e^{-\tau p},
& \Wstpw &= \overline{(\Zbstpw)} = 
\frac{\p}{\p w} +  \tau \frac{\p p}{\p w}= e^{-\tau p}\frac{\p p}{\p w}e^{\tau p}.
\end{align*}
To motivate the correct differential operator in $\tau$, it is essential to have the ``twist" term
\[
T(w,z) = -2\Imm\Big(\sum_{j\geq 1} \frac{1}{j!} \frac{\p^j p(z)}{\p z^j}(w-z)^j\Big)
\]
from the control metric on $M_p$. It turns out that the distance on $M_p$ in the
$t$-component is written in terms of
$t+T(w,z)$, and the partial Fourier transform in $t$ of $t+T(w,z)$ is the twisted derivative
\[
\Mtpzw = \T = \frac{\p}{\p \tau} - i T(w,z).
\]
Also associated to the control metric is the pseudo-distance 
\[
\mu_p(z,\delta) = \inf_{j,k\geq 1} 
\bigg|\frac{\delta}{\frac 1{j!k!}\frac{\p^{j+k} p(z)}{\p z^j \p\z^k}}\bigg|^{1/(j+k)}
\]
and its approximate inverse
\[
\Lambda(z,\delta) = \sjk \bigg|\frac 1{j!k!} \frac{\p^{j+k} p(z)}{\p z^j \p\z^k}\bigg| |\delta|^{j+k}.
\]
Roughly speaking, the volume of a ball in $M_p$ of radius $\delta$ is approximately $\delta^2\Lambda(z,\delta)$ and the distance from a point $(z,t)$ to $(w,s)$ is
$|z-w| + \mu_p(z,t-s+T(w,z))$.
The functions $\mu$ and $\Lambda$ satisfy
\[
\mu_p\big(z, \Lambda(z,\delta)\big) \sim \Lambda\big(z,\mu_p(z,\delta)\big) \sim \delta.
\]

Finally, for $(s,z)\in (0,\infty)\times\C$, define
\[
\Delta = \min\{ \mu_p(z,1/\tau), s^{1/2}\}.
\]

It will be important to distinguish the number of space derivatives from the number of
$\tau$-derivatives, and we do this by introducing the $(n,\ell)$-differentiation classes for functions of
$(\tau,z,w)\in\R\times\C\times\C$.
\begin{defn} \label{defn:(n,ell)-deriv} 
We say that $Y^J$ is an \emph{$(n,\ell)$-derivative} and write $Y^J\in(n,\ell)$ if
$Y^J = Y_{|J|} Y_{|J|-1}\cdots Y_1$
is a product of $|J|$ operators of the form 
$Y_j =\Zbstpz, \Zstpz, \Wbstpw, \Wstpw$, or $\Mtpzw$
where $|J|=n+\ell$, $n = \#\{Y_j: Y_j = \Mtpzw\}$ and 
$\ell = \#\{Y_j : Y_j = \Zbstpz, \Zstpz,\Wbstpw,\Wstpw\}$.
Also, we write $Y^J \leq (n,\ell)$ if $Y^J\in (k,j)$ where $k\leq n$ and $j\leq \ell$.
\end{defn}
For $Y^J\in (n,\ell)$, if
$0 \leq \alpha \leq |J|$, let $Y^{J-\alpha} =  Y_{|J|-\alpha} Y_{|J|-1-\alpha}\cdots
Y_1$. While
this is an abuse of notation, we will only use it
in situations where the length
of the derivative is important. Also, we commonly use the notation $X^J$ for operators
$X^J\in (0,\ell)$ and $Y^J$ when $Y^J\in (n,\ell)$, $n\geq 0$.

The main results for the heat equations of $\Boxtp$ and $\Boxwtp$ are the following.
%
%  THEOREM: BOXTP HEAT ESTIMATES
%
\begin{thm}\label{thm:Boxtp heat estimates}
Let $p$ be a subharmonic, nonharmonic polynomial and
$\tau>0$ a parameter. If $Y^J\in(n,\ell)$ and $k\geq 0$, then there exist constants 
$C_{k,|J|}, c >0$ so that
\[
\left| \frac{\p^k}{\p s^k} Y^J \Htp(s,z,w)\right| \leq C_{k,|J|}
\frac{\Lambda(z,\Delta)^n} {s^{1 + k+\frac 12\ell} }
e^{-c \frac{|z-w|^2}s} e^{-c \frac{s}{\mu_p(z, 1/\tau)^2}} 
e^{-c \frac{s}{\mu_p(w, 1/\tau)^2}}.
\]
\end{thm}

Since $\Lambda(z,\mu_p(z,1/\tau)) \sim 1/\tau$,  we have the immediate corollary.
\begin{cor}\label{cor:Boxtp heat estimates}
Let $\tau>0$. If $Y^J\in(n,\ell)$ and $k\geq 0$, then there exist constants 
$C_{k,|J|}, c >0$ so that
\[
\left| \frac{\p^k}{\p s^k} Y^J \Htp(s,z,w)\right| \leq C_{k,|J|}
\frac{1} {\tau^n s^{1 + k+\frac 12\ell} }
e^{-c \frac{|z-w|^2}s} e^{-c \frac{s}{\mu_p(z, 1/\tau)^2}} 
e^{-c \frac{s}{\mu_p(w, 1/\tau)^2}}.
\]
\end{cor}

%%%%%%%%%%%%%%%%%%%5
% THEOREM: BOUND ON H WIGGLE AND DERIVATIVES
%%%%%%%%%%%%%%%%%%%55555
\begin{thm}\label{thm:H wig and deriv}  Let $p$ be a subharmonic, nonharmonic polynomial and
$\tau>0$ a parameter.  If $Y^\alpha\in(n,\ell)$ and $k\geq 0$, then there exist 
positive constants
$c, C_{|\alpha|}, C_{k,|\alpha|}$, so that
\[
\left|Y^\alpha \Hwtp(s,z,w)\right|
\leq C_{|\alpha|} \frac{\Lambda(z,\Delta)^n}{ \Delta^{2+\ell}}e^{-c \frac{|z-w|^2}s}e^{-c \frac{|z-w|}{\mu_p(z,1/\tau)}}e^{-c \frac{|z-w|}{\mu_p(w,1/\tau)}} 
%\max \left\{
%\frac{ 1}%e^{-c \frac{s}{\mu_p(w,1/\tau)^2}}e^{-c \frac{s}{\mu_p(z,1/\tau)^2}}} 
%{s^{1+\frac 12\ell}},
%\frac{1}{\mu_p(w,1/\tau)^{2+\ell}} \right\}.
\]
Also, if the derivatives 
annihilate the Szeg\"o kernel, i.e., $\frac{\p^k}{\p s^k}Y^\alpha \Ss(z,w)=0$,
then the estimate simplifies to
\[
\left|\frac{\p^k}{\p s^k}Y^\alpha \Hwtp(s,z,w)\right|
\leq \frac{C_{k,|\alpha|} \Lambda(z,\Delta)^n}{ s^{1+k+\frac 12\ell}} e^{-c \frac{|z-w|^2}s}
e^{-c \frac{s}{\mu_p(w,1/\tau)^2}}e^{-c \frac{s}{\mu_p(z,1/\tau)^2}}.
\]
\end{thm}

Given the importance of Szeg\"o projection to the relative fundamental
solution of $\Boxwtp$, we will also study
the integral kernel of $e^{-s\Boxwtp}(I-\Ss)$ and its derivatives. 
Specifically, we  write
\[
e^{-s\Boxwtp}(I-\Ss)[f](z) = \int_{\C}\Gwtp(s,z,w)f(w)\, dA(w)
\]
and will analyze $\Gwtp(s,z,w)$ and its derivatives. 

%%%%%%%%%%%%%%%%%%%%%%%%5
% THEOREM: BOUND ON G WIGGLE AND DERIVATIVES
%%%%%%%%%%%%%%%%%%%%%%%%%%5
\begin{thm}\label{thm:G wig and deriv} Let $p$ be a subharmonic, nonharmonic polynomial and
$\tau>0$ a parameter.  If $Y^\alpha\in(n,\ell)$ and $k\geq 0$, then there exist 
positive constants
$c, C_{|\alpha|}, C_{k,|\alpha|}$ so that
\[
\left| Y^\alpha \Gwtp(s,z,w)\right|
\leq \frac{C_{|\alpha|}}{\tau^n} 
e^{-c \frac{s}{\mu_p(w,1/\tau)^2}}e^{-c \frac{s}{\mu_p(z,1/\tau)^2}}e^{-c \frac{|z-w|}{\mu_p(z,1/\tau)}}
e^{-c \frac{|z-w|}{\mu_p(w,1/\tau)}} 
\max \left\{ \frac{ e^{-c \frac{|z-w|^2}s}} {s^{1+\frac 12\ell}},
\frac{1}{\mu_p(w,1/\tau)^{2+\ell}} \right\}.
\]
Also, when the derivatives annihilate the Szeg\"o kernel, i.e., 
$\frac{\p^k}{\p s^k}Y^\alpha \Ss(z,w)=0$,  $\frac{\p^k}{\p s^k}Y^\alpha \Hwtp(s,z,w)
= \frac{\p^k}{\p s^k}Y^\alpha \Gwtp(s,z,w)$ and
the estimate is
\[
\left|\frac{\p^k}{\p s^k}Y^\alpha \Gwtp(s,z,w)\right|
\leq C_{k,|\alpha|}\frac{\Lambda(z,\Delta)^n}{s^{1+k+\frac 12\ell}} e^{-c \frac{|z-w|^2}s}
e^{-c \frac{s}{\mu_p(w,1/\tau)^2}}e^{-c \frac{s}{\mu_p(z,1/\tau)^2}}.
\]
\end{thm}
\begin{rem}\label{rem:(0,ell) cases already proved}
The $(0,\ell)$-case of Theorem \ref{thm:Boxtp heat estimates}, Theorem \ref{thm:H wig and deriv},
and Theorem \ref{thm:G wig and deriv} is proved in \cite{Rai06h,Rai07}. The fact that
$\min_{s\geq 0} \frac {|z-w|^2}s + \frac{s}{\mu_p(w,1/\tau)^2} = \frac{|z-w|}{\mu_p(w,1/\tau)}$ which allows for the
$\exp(-c\frac{|z-w|}{\mu_p(w,1/\tau)})\exp(-c\frac{|z-w|}{\mu_p(z,1/\tau)})$ term to be factored outside of the max.
\end{rem}
The estimates in Theorem \ref{thm:H wig and deriv} and Theorem \ref{thm:G wig and deriv} generalize
the main results in \cite{Rai07} in which we prove the estimates for the $(0,\ell)$-case.

As essential tool in the proof of Theorem \ref{thm:G wig and deriv} 
is a special case of Corollary 4.2 from 
\cite{Rai07}. The corollary in \cite{Rai07} is based  on the identity
$e^{-s\Boxtp}\Zbstp = \Zbstp e^{-s\Boxwtp}$. If
$\Rtp$ is the relative inverse to $\Zbstp$  (i.e., $\Rtp\Zbstp = I-\Stp$)
and has integral kernel $\Rtp(z,w)$, then $\Rtp(z,w)$ is the relative fundamental solution to
$\Zbstp$, and it is shown that 
applying $\Rtp$ to the identity yields the following result.
%Corollary: integral relating \Gwtp and \Htp 
\begin{prop}\label{prop:integral relating Gwtp and Htp}
Let $\tau>0$. Then
\[
\Zbstpz\Hwtp(s,z,w) 
=\Zbstpz\Gwtp(s,z,w) 
=\Wbstpw\Htp(s,z,w), 
\]
and
\[
 \Wstpw\Hwtp(s,z,w) 
= \Wstpw\Gwtp(s,z,w) 
= \Zstpz\Htp(s,z,w). 
\]
Also, 
\[
\Gwtp(s,z,w) 
= -\int_\C \Wbstpw\Htp(s,v,w) \Rtp(z,v)\, dA(v).
\]
\end{prop}
From Theorem \ref{thm:Boxtp heat estimates} and the 
first two equalities in Corollary \ref{prop:integral relating Gwtp and Htp}, 
the latter statements
in Theorem \ref{thm:H wig and deriv} and Theorem \ref{thm:G wig and deriv} follow immediately.

The estimates in Theorem \ref{thm:Boxtp heat estimates} 
and Theorem \ref{thm:G wig and deriv} allow us to recover pointwise estimates
on the Szeg\"o kernel and the relative fundamental solution to $\Zbstp$. We need to integrate
out $s$ in the estimate of $\Zstpz\Htp(s,z,w)$ 
to recover estimates for the relative fundamental solution
of $\Zbstp$, and we need to take the limit as $s\to 0$ of the estimate for $\Gwtp$ to recover
the estimate of the Szeg\"o kernel.
We have the corollary
% Corollary: \Rtp estimates
\begin{cor}\label{cor:Rtp estimates}
Let $\tau>0$. If $Y^J$ is an $(n,\ell)$-derivative, then there exist constants 
$C_{|J|}, c >0$ so that
\[
|Y^J \Rtp(z,w)| \leq C_{|J|} \begin{cases} \tau^{-n}|z-w|^{-\ell} &|z-w|\leq \mu_p(z,1/\tau) \\
      \frac{1}{\tau^n \mu_p(z,1/\tau)^{1+|J|}} e^{-c \frac{|z-w|}{\mu_p(z,1/\tau)}}
 e^{-c \frac{|z-w|}{\mu_p(w,1/\tau)}} &|z-w|\geq \mu_p(z, 1/\tau).
                             \end{cases}
\]
Also,
\[
|Y^J \Stp(z,w)| \leq C_{|J|} \frac{1}{\tau^n \mu_p(z,1/\tau)^{2+\ell}}
e^{-c \frac{|z-w|}{\mu_p(z,1/\tau)}}  e^{-c \frac{|z-w|}{\mu_p(w,1/\tau)}}.
\]
\end{cor}
The proof of Corollary \ref{cor:Rtp estimates} is identical to the proof of 
Corollary 2 in \cite{Rai06h}. 

\begin{remark}
The estimates in Theorem \ref{thm:H wig and deriv} and Theorem \ref{thm:G wig and deriv}
are natural given the estimate in Theorem \ref{thm:Boxtp heat estimates}. 
Since $\lim_{s\to\infty} e^{-s\Boxwtp} = \Stp$, the large time estimate in
Theorem \ref{thm:H wig and deriv} should agree with the estimate for the Szeg\"o kernel.
From Corollary \ref{cor:Rtp estimates}, we see that the estimates agree as $s\to\infty$.
Similarly, since $\lim_{s\to 0}e^{-s\Boxwtp}(I-\Stp) = I-\Stp$,
the estimates for $\Gwtp(s,z,w)$ must become the estimates for $\Stp(z,w)$ as $s\to 0$.
Thus, expressing the estimates in Theorem \ref{thm:H wig and deriv} and
Theorem \ref{thm:G wig and deriv} in terms of maximums 
is natural -- the estimates agree with the estimates for  $\Boxtp$-heat kernel and the
Szeg\"o kernel on the appropriate regions.
\end{remark}

We prove a cancellation condition for $e^{-s\Boxtp}$.
It compliments the pointwise estimates of Theorem \ref{thm:Boxtp heat estimates} and
is of interest in its own right. Following the notation of \cite{Rai06h, Rai07, NaSt01h},
we let
\[
\Htp^s[\vp](z) = e^{-s\Boxtp}[\vp](z)
\]
and similarly for $\Hwtp^s[\vp](z)$ and $\Gwtp^s[\vp](z)$.
%
%   THEOREM: CANCELLATION CONDITION
%
\begin{thm}\label{thm:Boxtp cancellation conditions}
Let $\tau>0$. If $Y^J\in(n,\ell)$,  $\delta< \max\{\mu_p(z,1/\tau), s^{\frac 12}\}$,
and $\vp\in C^\infty_c\big(D(z,\delta)\big)$, then
there exists a constant $C_{|J|}$ so that for
$\ell$ even,
\[
|Y^J \Htp^s [\vp](z)| \leq C_{|J|}\frac{\Lambda(z,\Delta)^n}{\delta} 
\big(\|\Boxtp^{\frac \ell2}\vp\|_{L^2(\C)} + \delta^2 \|\Boxtp^{\frac \ell2+1} \vp\|_{L^2(\C)}\big),
\]
and for $\ell$ odd,
\[
|Y^J \Htp^s [\vp](z)| \leq C_{|J|}\frac{\Lambda(z,\Delta)^n}{ \delta} 
\big(\delta\|\Boxtp^{\frac {\ell+1}2}\vp\|_{L^2(\C)} 
+ \delta^3 \|\Boxtp^{\frac {\ell+3}2} \vp\|_{L^2(\C)}\big),
\]
\end{thm}
Theorem \ref{thm:Boxtp cancellation conditions} allows us to recover a cancellation condition
for $G_{\tau p}=\Boxtp^{-1}$ and the relative fundamental solution of $\Zbstp$, denoted $\Rtp$. 
% Corollary: R_\tau p cancellation
\begin{cor}\label{cor:R cancel}
Let $\tau>0$.
Let $\delta>0$ and $\vp\in \cic{D(z,\delta)}$. Let
$Y^\alpha\in(n,\ell)$. 
There exists a constant $C_{n,|\alpha|}$ so that if $|\alpha|=2k>0$ is even
or $|\alpha|=0$ and $\delta \geq \mu_p(z,1/\tau)$,
then
\[
|Y^\alpha G_{\tau p}[\vp](z)| 
\leq \frac{C_{|\alpha|}}{\tau^n} 
\delta \big(\|\Boxtp^k\vp\|_{L^2(\C)} +\delta^2 \|\Boxtp^{k+1}\vp\|_{L^2(\C)}\big)
\]
and if $|\alpha|=2k+1>0$ is odd, then
\[
|Y^\alpha G_{\tau p}[\vp](z)| 
\leq \frac{C_{|\alpha|}}{\tau^n}  \delta
\big( \delta \|\Boxtp^{k+1}\vp\|_{L^2(\C)} + \delta^3 \|\Boxtp^{k+2}\vp\|_{L^2(\C)}\big).
\]
If $|\alpha|=0$ and $\delta < \mu_p(z, 1/\tau)$, then
\[
|G_{\tau p}[\vp](z)| \leq \frac{C_0}{\tau^n} \delta \Big( \log(\tfrac{2\mu_p(z,1/\tau)}{\delta}) \|\vp\|_{L^2(\C)}
+ \delta^2 \|\Boxtp\vp\|_{L^2(\C)}\Big).
\]
\end{cor}

The proof of Corollary \ref{cor:R cancel} can be followed line by line from the proof of Lemma 3.6 in
\cite{Rai07}.

%%%%%%%%%%%%%%%%%%%%%
%
%	SECTION: RECOVERY OF  RAPID DECAY IN C^2
%
%%%%%%%%%%%%%%%%%%%%%
\section{Estimates for heat kernels for polynomial models in $\C^2$.}
\label{sec:heat kernel estimates on poly models}
As discussed in \S\ref{subsec:heat kernels on poly models in C2}, we have the polynomial model $M_p$. 
Let $H(s,p,q)$ be the integral kernel of $e^{-s\Boxb}$.
Because $M$ is a polynomial model, if $p=(z,t_1)$, $q=(w,t_2)$ and $t=t_1-t_2$, then we can consider
$H(s,p,q) = H(s,z,w,t)$ and $G(s,p,q) = G(s,z,w,t)$. In this notation, if
$d_M(z,w,t) = |z-w| + \mu_p(z, t+T(w,z))$ and $X^\alpha$ and $X^\beta$ are products of the vectors fields $L$ and $\LL$, then 
Nagel and Stein \cite{NaSt01h} prove that for any nonnegative
integer $N$, there exists a constant $C_{N,\alpha,\beta,j}$ so that
\[
\bigg| \frac{\p^j}{\p s^j} X^\alpha_p X^\beta_q H(s,z,w,t_1-t_2) \bigg|
\leq C_{N,\alpha,\beta,j} \frac{d_M(z,w,t)^{-2j-|\alpha|-|\beta|}}{d_M(z,w,t)^2\Lambda(z,d_M(z,w,t))}\bigg[ \frac{s^N}{s^N + d_M(z,w,t)^{2N}}\bigg].
\]
%and there exists a constant $C_{\alpha,\beta,j}$ so that
%\[
%\bigg| \frac{\p^j}{\p s^j} X^\alpha_p X^\beta_q G(s,z,w,t_1-t_2) \bigg|
%\leq C_{\alpha,\beta,j} 
%\begin{cases} d_M(z,w,t)^{-2j-|\alpha|-|\beta|-2} \Lambda(z,d_M(z,w,t))^{-1} & s\leq d_M(z,w,t)^2 \\
%s^{-j-\frac 12(|\alpha|+|\beta|)-1}\Lambda(z,\sqrt s)^{-1} & s\geq d_M(z,w,t)^2
%\end{cases}
%\]
As a consequence of Theorem \ref{thm:Boxtp heat estimates} and Theorem \ref{thm:H wig and deriv}, we improve the previous estimates to the following.
% Theorem: better heat kernel decay estimates
\begin{thm} \label{thm:heat kernel decay estimates on polynomial models}
Let $p:\C\to \R$ be a subharmonic, nonharmonic polynomial and $M_p = \{ (z,w)\in\C^2 : \Imm w = p(z)\}$. Under the standard identification
of $M_p\cong\C\times\R$, if $H(s,z,w,t_1-t_2)$ is the integral kernel of $e^{-s\Boxb}$ and $X^\alpha$ and $X^\beta$ are compositions
of the vector fields $L$ and $\LL$, then for any integer $N$, there exists a constants $c = c_{N,\alpha,\beta,j}$ and $C_{N,\alpha,\beta,j}$ so that
\[
\bigg| \frac{\p^j}{\p s^j} X^\alpha_p X^\beta_q H(s,z,w,t_1-t_2) \bigg|
\leq C_{N,\alpha,\beta,j}\frac{e^{-c\frac{|z-w|^2}s}}{d_M(z,w,t)^{2+2j+|\alpha|+|\beta|} \Lambda(z,d_M(z,w,t))} \frac{ s^N }{\mu_p(z,t+T(w,z))^{2N}}
\]
where $t = t_1-t_2$.
\end{thm}

\begin{proof} We sketch the proof in the 
$\alpha=\beta=j=0$ case. The other cases follow similarly. By a partial Fourier transform,
\[
H(s,z,w,t) = \frac{1}{\sqrt{2\pi}} \int_\R e^{-it\tau} \Htp(s,z,w)\, d\tau.
\]
Next,
\[
\bigg|  \int_\R e^{-it\tau} \Htp(s,z,w)\, d\tau \bigg|
= \frac{1}{|t+T(w,z)|^n} \bigg| \int_\R e^{-it\tau} \Mtp^n \Htp(s,z,w)\, d\tau \bigg|.
\]
Note that $\mu_p(z,1/\tau) \geq s^{1/2}$ is equivalent to $\tau \leq \Lambda(z,s^{1/2})^{-1}$.
By Theorem \ref{thm:Boxtp heat estimates} and Theorem \ref{thm:H wig and deriv}, we estimate
\begin{align*}
\frac{1}{|t+T(w,z)|^n}& \bigg| \int_{|\tau|\leq \Lambda(z,s^{1/2})^{-1}}  e^{-it\tau} \Mtp^n \Htp(s,z,w)\, d\tau \bigg| \\
&\leq \frac{1}{|t+T(w,z)|^n} \int_{|\tau|\leq \Lambda(z,s^{1/2})^{-1}} \frac{\Lambda(z,s^{1/2})^n}s e^{-c\frac{|z-w|^2}2} \, d\tau 
\leq \frac{\Lambda(z,s^{1/2})^n}{|t+T(z,w)|^n} \frac{e^{-c\frac{|z-w|^2}s}}{s\Lambda(z,s^{1/2})}.
\end{align*}
Similarly, with $n$ large enough so that the integral converges, by Theorem \ref{thm:H wig and deriv}
\begin{align*}
\frac{1}{|t+T(w,z)|^n}& \bigg| \int_{|\tau|\geq \Lambda(z,s^{1/2})^{-1}}  e^{-it\tau} \Mtp^n \Htp(s,z,w)\, d\tau \bigg| \\
&\frac{1}{|t+T(w,z)|^n} \int_{|\tau|\geq \Lambda(z,s^{1/2})^{-1}} \frac{e^{-c\frac{|z-w|^2}s}}{|\tau|^n \mu_p(z,1/\tau)^2}\, d\tau
\leq C \frac{\Lambda(z,s^{1/2})^n}{|t+T(z,w)|^n} \frac{e^{-c\frac{|z-w|^2}s}}{s\Lambda(z,s^{1/2})}.
\end{align*}
There are two key ideas to finish the proof. The first is that the bound $\frac{\Lambda(z,s^{1/2})^n}{|t+T(z,w)|^n}$ for all $n$ is equivalent
to having the bound $\frac{ s^N }{\mu_p(z,t+T(w,z))^{2N}}$ for all $N$ (it is a matter is expanding the $\mu_p(z,t+T(w,z))$ terms and reshuffling
the $\frac{\p^{j+k} p(z)}{\p z^j \p\z^k}$ terms. The second key fact is that the statement $n$ cannot equal zero is 
a manifestation of the lack of decay in $s$ for $H(s,z,w,t)$. It turns out that to  achieve the term $\frac{ s^N }{\mu_p(z,t+T(w,z))^{2N}}$
in the estimate for $H(s,z,w,t)$ 
causes the replacement of  $s$ with $d(z,w,t)^2$ throughout the denominator. 
\end{proof}

If $S$ is the Szeg\"o projection of $L^2(M)$ on $\ker(\dbarb)$,
we can also recover estimates for the kernel of $e^{-s\Boxb}(I-S)$ and the Szeg\"o projection.
The difference in the argument
is that we only integrate by parts only when $\tau$ is away from zero. The estimates themselves are less subtle as they do not involve
rapid or exponential decay. The estimates on the Szeg\"o projection are computed in both \cite{NaRoStWa89} and \cite{McN89}.

%%%%%%%%%%%%%%%%%
%
%   SECTION: NONHOMOGENEOUS IVP
%
%%%%%%%%%%%%%%%%%%555
\section{The nonhomogenous ivp, uniqueness, and mixed derivatives of the heat kernel}
\label{sec:nonhomog IVP}

For the remainder of the article, we assume that $\tau>0$.

The key to expressing derivatives of $\Htp(s,z,w)$ in terms of quantities
we can estimate is to solve the  nonhomogeneous heat equation via Duhamel's principle.

% subsection: nonhomogeneous IVP and uniqueness
\subsection{Uniqueness of solutions of the nonhomogeneous IVP}\label{subsec:IVP}

% Proposition: nonhomogeneous IVP
\begin{prop}\label{prop:IVP} Let $\tau\in\R$.
Let $g:(0,\infty)\times\C\to\C$ and $f:\C\to\C$ be $H^2(\C)$ 
for each $s$ and vanish as
$|z|\to\infty$.
The solution to the nonhomogeneous heat equation
\begin{equation}\label{eqn:nonhomog IVP}
\begin{cases} \displaystyle \frac{\p u}{\p s} + \Boxtp u = g \text{  in } (0,\infty)\times\C
\vspace{.075in}\\
\displaystyle \lim_{s\to 0} u(s,z) = f(z) \end{cases}
\end{equation}
is given by
\begin{equation}\label{eqn:nonhomog soln}
u(s,z) = \int_\C \Htp(s,z,\xi) f(\xi)\, dA(\xi) + \int_0^s \int_\C \Htp(s-r,z,\xi) g(r,\xi)\, dA(\xi) dr.
\end{equation}
\end{prop}

\begin{proof} By \cite{Rai06h}, it suffices to show that 
$u(s,z) =  \int_0^s \int_\C \Htp(s-r,z,\xi) g(r,\xi)\, dA(\xi) dr$ solves \eqref{eqn:nonhomog IVP} when
$f\equiv0$. Let
\[
u_\ep(s,z) = \int_0^s \int_\C \Htp(s-r+\ep,z,\xi)g(\xi,r)\, dA(\xi) dr.
\]
Then
\[
\frac{\p u_\ep}{\p s} = \int_0^s\int_\C \frac{\p\Htp}{\p s}(s-r+\ep,z,\xi)g(r,\xi)\, dA(\xi) dr
+ \int_\C \Htp(\ep,z,\xi)g(s,\xi)\, dA(\xi)
\]
and
\[
\Boxtpz u_\ep(s,z) = \int_0^s \hspace{-3.5pt}\int_\C \Boxtpz \Htp(s-r+\ep,z,\xi) g(r,\xi)\, dA(\xi) dr
 = \int_0^s \hspace{-3.5pt}\int_\C -\frac{\p\Htp}{\p s}(s-r+\ep,z,\xi)g(r,\xi)\, dA(\xi) dr.
\]
Adding the previous two equations together, we have (let $g_s(z) = g(s,z)$)
\[
\left(\frac{\p u_\ep}{\p s} + \Boxtpz u_\ep\right)(s,z) = e^{-\ep\Boxtp}[g_s](z)
\stackrel{\ep\to 0}{\longrightarrow} g(s,z)
\]
in $L^2(\C)$ for each fixed $s$. If $g_s\in C^2_c(\C)$, then the convergence is uniform (as a consequence
of \cite{Rai06h}). We need to show that  $\lim_{\ep\to 0} u_\ep(s,z) = u(s,z)$. But this follows
from writing
\[
u_\ep(s,z)-u(s,z) = \ep \int_0^s \int_\C g(r,\xi) 
\frac{\Htp(s-r+\ep,z,\xi)-\Htp(s-r,z,\xi)}{\ep}\, dA(\xi) dr
\]
and using the size and cancellation conditions for $\frac{\p \Htp}{\p s}$ 
in Theorem \ref{thm:Boxtp heat estimates}
Theorem \ref{thm:Boxtp cancellation conditions} (the $(0,\ell)$-case is
proved in \cite{Rai07}).
The final fact we must check is that $\lim_{s\to 0}u(s,z)=0$ in $L^2(\C)$.
This, however, follows from the fact that $e^{-s\Boxtp}$ is a contraction in $L^2$ for all $s\geq 0$.
\end{proof}

We next prove a uniqueness result for solutions of \eqref{eqn:nonhomog IVP}.
% Proposition: uniqueness for the IVP
\begin{prop} \label{prop:IVP uniqueness}
Let $\tau>0$ and $u_1$ and $u_2$ satisfy (\ref{eqn:nonhomog IVP}). If
$u_1, u_2 \in C^1\big((0,\infty)\times\C\big)$ and are in $L^2(\C)$ for each $s$, then
$u_1 = u_2$.
\end{prop}

\begin{proof} We will use the fact from \cite{Christ91} that 
for $\tau>0$, there exists $C=C(\tau,p)>0$ so
that $\|f\|_{L^2(\C)} \leq C\|\Zstp f\|_{L^2(\C)}$.
Since $u_1,u_2$ satisfy (\ref{eqn:nonhomog IVP}), it follows that
$h(s,z) = u_1(s,z) - u_2(s,z)$ satisfies $(\frac{\p}{\p s} + \Boxtp)h =0$ and $h(0,z)=0$.
Let $g(s) = \int_\C |h(s,z)|^2\, dA(z)$. Note that $g(s)\geq0$ and $\frac{\p h}{\p s} = \Zbstp\Zstp h$.
Consequently,
\begin{align*}
g'(s) &= 2\Rre \Big( \int_\C \frac{\p h}{\p s}(s,z) \overline{h(s,z)}\, dA(z) \Big)\\
&= 2\Rre \Big( \int_\C \Zbstp\Zstp h(s,z) \overline{h(s,z)}\, dA(z) \Big) 
= - 2 \int_\C |\Zstp h(s,z)|^2\, dA(z) \leq 0.
\end{align*}
Since $g$ is nonnegative, $g(0)=0$ and $g'(s) \leq 0$ for all $s$, 
it follows that $g(s)=0$ for all $s$.
Thus, $h(s,z)=0$ for all $s, z\in \C$.
\end{proof}

% subsection: bootstrapping
\subsection{Mixed Derivatives of $\Htp(s,z,w)$}

We now derive a formula to express $Y^J \Htp(s,z,w)$ using
Duhamel's principle.
% Proposition: commutators
\begin{prop}\label{prop:Y^J H_tp formula}
Let $\tau\in\R$ and $Y^J \in(n,\ell)$. If
\begin{equation}\label{eqn:H^J expansion}
\Htp^J(s,\xi,w) = \hspace{-5.3pt}
\sum_{k=0}^{|J|-2}\Big(\prod_{\imath=0}^k Y_{|J|-\imath}\Big)\big[\Boxtpxi, 
Y_{|J|-k-1}\big] Y^{J-k-2}\Htp(s,\xi,w) + \big[\Boxtpxi,Y_{|J|}\big] Y^{J-1}\Htp(s,\xi,w)
\end{equation}
where $\prod_{\imath=0}^k Y_{|J|-\imath} = Y_{|J|}Y_{|J|-1}\cdots Y_{|J|-k}$, then
\begin{equation}\label{eqn:H^J int}
Y^J\Htp(s,z,w) = \lim_{\ep\to 0} \int_\C \Htp(s,z,\xi) Y^J\Htp(\ep,\xi,w)\, dA(\xi) + 
\int_0^s \int_\C \Htp(s-r,z,\xi) \Htp^J(r,\xi,w)\, dA(\xi) dr.
\end{equation}
\end{prop}

\begin{proof}
Observe that $\frac{\p}{\p s}$ commutes with $Y_j$ for all $j$, so
\begin{align*}
&\Big(\frac{\p}{\p s} + \Boxtpxi\Big)Y^J \Htp(s,\xi,w)
= Y_{|J|}\Big(\frac{\p}{\p s} + \Boxtpxi\Big)Y^{J-1}\Htp(s,\xi,w) + 
\big[\Boxtpxi,Y_{|J|}\big] Y^{J-1}\Htp(s,\xi,w) \\
&= Y_{|J|}Y_{|J|-1}\Big(\frac{\p}{\p s} + \Boxtpxi\Big)Y^{J-2}\Htp(s,\xi,w)
+ Y_{|J|}\big[\Boxtpxi, Y_{|J|-1}\big] Y^{J-2}\Htp(s,\xi,w) \\
&\ +\big[\Boxtpxi,Y_{|J|}\big] Y^{J-1}\Htp(s,\xi,w) \\
&= \cdots = Y^J \Big(\frac{\p}{\p s} + \Boxtpxi\Big)\Htp(s,\xi,w) + 
\sum_{k=0}^{|J|-2}\Big(\prod_{\imath=0}^k Y_{|J|-\imath}\Big)\big[\Boxtpxi, Y_{|J|-k-1}\big]
Y^{J-k-2}\Htp(s,\xi,w)\\
&\  + \big[\Boxtpxi,Y_{|J|}\big] Y^{J-1}\Htp(s,\xi,w).
\end{align*}
But $\Htp(s,\xi,w)$ is annihilated by the heat operator, so
\[
\Big(\frac{\p}{\p s} + \Boxtpxi\Big)Y^J \Htp(s,\xi,w) = \Htp^J(s,\xi,w).
\]

$Y^J\Htp(s+\ep,z,w)$ satisfies (\ref{eqn:nonhomog IVP}) with
initial value $Y^J\Htp(\ep,z,w)$ and inhomogeneous term $\Htp^J(s,z,w)$. 
By Proposition \ref{prop:IVP}, sending
$\ep\to 0$ in the second integral finishes the proof of the result.
\end{proof}

In order to use Proposition \ref{prop:Y^J H_tp formula}, we must understand
$\big[\Boxtpz, Y_j\big]$. Certainly, if $Y_j = \Wbstpw$ or $\Wstpw$, 
$\big[\Boxtpz, Y_j\big]=0$. For the other cases, it is helpful to recall the
expansions of $\Boxtp$ and $\Boxwtp$.
Recall that 
\begin{align}
\Boxtp &= -\frac{\p^2 }{\p z\p\z} + \tau \frac{\p^2 p}{\p z\p\z}
+ \tau^2 \frac{\p p}{\p z}\frac{\p p}{\p\z}
+\tau\left( \frac{\p p}{\p z} \frac{\p }{\p\z} - \frac{\p p}{\p\z}\frac{\p }{\p z}\right) \label{eqn:Boxz}\\
&= -\frac14 \triangle + \frac14 \tau \triangle p+ \frac {\tau^2}4|\nabla p|^2  + \frac i2\tau 
\left(\frac{\p p}{\p x_1}\frac{\p}{\p x_2} - \frac{\p p}{\p x_2}\frac{\p}{\p x_1}\right)\nn \label{eqn:Boxx}
\end{align}
and
\begin{align}
\Boxwtp &=\frac{\p^2 }{\p z\p\z} - \tau \frac{\p^2 p}{\p z\p\z}
+ \tau^2 \frac{\p p}{\p z}\frac{\p p}{\p\z}
+\tau\left( \frac{\p p}{\p z} \frac{\p }{\p\z} - \frac{\p p}{\p\z}\frac{\p }{\p z}\right)\\ %\label{eqn:Boxwz}\\
&= -\frac14 \triangle - \frac14 \tau \triangle p+ \frac {\tau^2}4|\nabla p|^2  + \frac i2\tau 
\left(\frac{\p p}{\p x_1}\frac{\p}{\p x_2} - \frac{\p p}{\p x_2}\frac{\p}{\p x_1}\right).  
\nn %\label{eqn:Boxwx}
\end{align}
As noted above, $\Boxtp$ and $\Boxwtp$ are self-adjoint operators in $L^2(\C)$, and since they
are not real, their adjoints in the sense of distributions do not agree with their 
$L^2$-adjoints
(on any functions in both domains, e.g., functions in $\cic\C$). If $X$ is a differential operator, 
denote its distributional adjoint by $X\sh$. Note that $\Zbstp\sh = -\Wbstp$ and 
$\Zstp\sh = -\Wstp$.

%Proposition: Box, deriv commutator
\begin{prop}\label{prop:box, deriv commutator} Let
\[
e(w,\xi) = \sum_{j\geq 1} \frac{1}{j!} \frac{\p^{j+1}p(\xi)}{\p\xi^j\p\bar\xi}(w-\xi)^j.
\]
Since $\Mtp f(\xi,w) = e^{i\tau T(w,\xi)}\frac{\p}{\p\tau}e^{-i\tau T(w,\xi)}f(\xi,w)$,
we have the following: 
\begin{enumerate}\renewcommand{\theenumi}{\alph{enumi}}
\item\vspace*{.07in} $\displaystyle \big[ \Boxtpxi, \Zbstpxi\big]
= -2\tau \frac{\p^3 p}{\p\xi \p\bar\xi^2}  - 2\tau\frac{\p^2 p}{\p\xi\p\bar\xi}\Zbstpxi$.

\item\vspace*{.07in} $\displaystyle \big[ \Boxtpxi, \Zstpxi\big] =  2\tau\frac{\p^2 p}{\p\xi\p\bar\xi}\Zstpxi$.

\item\vspace*{.07in} $\displaystyle \big[\Boxtpxi, \Mtpxiw\big]  
=  -\frac{\p^2 p}{\p\xi\p\bar\xi}  - e(w,\xi)\Zstpxi 
+ \overline{e(w,\xi)}\Zbstpxi$.

\item\vspace*{.07in} $\displaystyle [\Mtpxiw,\Zbstpxi] = -e(w,\xi)$.

\item\vspace*{.07in} $\displaystyle [\Mtpxiw, \Zstpxi] = \overline{e(w,\xi)}$.
\end{enumerate}

\end{prop}

\begin{proof} The proof is a computation. Parts (d) and (e) are proven directly from the definitions. For (a), 
\[
\big[ \Boxtpxi, \Zbstpxi\big] = \Zstpxi (\Boxwtpxi - \Boxtpxi)
= \Zbstpxi \Big(-2\tau \frac{\p^2 p}{\p\xi\p\bar\xi}\Big)
=-2\tau \frac{\p^3 p}{\p\xi \p\bar\xi^2} - 2\tau\frac{\p^2 p}{\p\xi\p\bar\xi}\Zbstpxi.
\]
(b) is a similar computation.
% \[
% \big[ \Boxtpxi, \Zstpxi\big] = (\Boxtpxi - \Boxwtpxi)\Zstpxi = 2\tau \frac{\p^2 p}{\p\xi\p\bar\xi}\Zstp.
% \]
To prove (c), observe that
\begin{align*}
 [\Boxtpxi,\Mtpxiw]f = [\Mtpxiw,\Zbstpxi]\Zstpxi f - \Zbstpxi[\Zstpxi,\Mtpxiw]f
= -e(w,\xi)\Zstp f + \overline{e(w,\xi)}\Zbstpxi f + \frac{\p(\overline{e(w,\xi)})} {\p\bar\xi} f,
\end{align*}
and $\frac{\p(\overline{e(w,\xi)})} {\p\bar\xi} = -\frac{\p^2 p(\xi)}{\p\bar\xi\p \xi}$, and the result follows from (d) and (e).
\end{proof}
Proposition \ref{prop:box, deriv commutator} underscores the importance of conjugating the
$\tau$-derivative by an oscillating factor. $\frac{\p p}{\p\xi}$ and $\frac{\p p}{\p\bar\xi}$ are
not controlled by $\Lambda_{p}$ and $\mu_{p}$ 
while $|e(w,\xi)(w-\xi)| \leq \Lambda(\xi,|w-\xi|)$.

%%%%%%%%% STRATEGY DIGRESSION
\subsection{Digression -- Discussion of Strategy}
We will prove Theorem
\ref{thm:Boxtp heat estimates} and Theorem \ref{thm:Boxtp cancellation conditions} 
in part by a double
induction. In order to understand the double induction, 
it is helpful to
investigate (\ref{eqn:H^J expansion}) and (\ref{eqn:H^J int})
further. Let $Y^J\in(n,\ell)$. From \cite{Rai06h}, we have the desired bounds for $n=0$ 
and $\ell$ arbitrary.
Thus, it is natural to induct on $n$.

Assume that $Y^J = X^\alpha (\Mtpzw)^n$. From (\ref{eqn:H^J int}), it is clear that 
we will have to understand
\[
\big(\prod_{\imath=0}^k Y_{|J|-\imath}\big) [\Boxtpxi, Y_{|J|-k-1}] Y^{J-k-2} \Htp(s,\xi,w).
\]
If $Y_{|J|-k-1}=\Wbstpw$ or $\Wstpw$, then the commutator is 0 because the terms commute.
If $Y_{|J|-k-1} = \Mtpzw$, then
\begin{multline*}
\big(\prod_{\imath=0}^k Y_{|J|-\imath}\big) [\Boxtpxi, \Mtpxiw] Y^{J-k-2} \Htp(s,\xi,w)
= \big(\prod_{\imath=0}^k Y_{|J|-\imath}\big) \Big( -\frac{\p^2 p(\xi)}{\p\xi \p\bar\xi}\Big) 
Y^{J-k-2} \Htp(s,\xi,w) \\
+ \big(\prod_{\imath=0}^k Y_{|J|-\imath}\big) \Big(e(w,\xi) \Zstpxi 
+ \overline{e(w,\xi)}\Zbstpxi\Big) Y^{J-k-2} \Htp(s,\xi,w).
\end{multline*}
All three of the terms in the right-hand side involve $(n-1)$ appearances of $\Mtpxiw$, 
so presumably they can be
controlled by the induction hypothesis on $n$. 

If, however, $Y_{|J|-k-1}=\Zstp$, then $[\Boxtp, \Zstp]= 
2\tau \frac{\p^2 p(\xi)}{\p\xi\p\bar\xi} \Zstpxi$, 
then we must estimate
\[
\int_0^s \int_\C \Htp(s-r,z,\xi) \Big(\prod_{\imath=0}^k Y_{|J|-\imath} \Big) 
\tau\frac{\p^2 p(\xi)}{\p\xi\p\bar\xi} \Zstpxi
Y^{|J|-k-2}\Htp(r,\xi,w)\, dA(\xi) dr.
\]
As written, this integral is not good because the second 
term in the integral is not covered by the 
induction hypothesis -- $\Mtpxiw$ appears $n$ times. However, 
there exists a $\xi$-derivative term before any 
$\Mtpxiw$-term , so we 
can integrate by parts. Specifically, 
$(\prod_{\imath=0}^k Y_{|J|-\imath}) \tau\frac{\p^2 p(\xi)}{\p\xi\p\bar\xi}\Zstpxi Y^{|J|-k-2}$
can be written in one of the following two forms:
$X^{\alpha_1}_w X_\xi X^{\alpha_2}\big( \tau \frac{\p^2 p}{\p\xi\p\bar\xi}\big)\Zstpxi Y^{|J|-k-2}$ or
$X^{\alpha_1}_w \tau \frac{\p^2 p}{\p\xi\p\bar\xi}\Zstpxi Y^{|J|-k-2}$. 
The form depends on whether or not
$\Zstpxi$ is the first term that involves taking a $\xi$ derivative. The integral then becomes
\begin{align*}
\int_0^s &\int_\C \Htp(s-r,z,\xi) \Big(\prod_{\imath=0}^k Y_{|J|-\imath} \Big) 
\tau\frac{\p^2 p(\xi)}{\p\xi\p\bar\xi}\Zstpxi 
Y^{|J|-k-2}\Htp(r,\xi,w)\, dA(\xi) dr\\
&= \int_0^s \int_\C \Htp(s-r,z,\xi)X^{\alpha_1}_w X_\xi 
X^{\alpha_2}\big( \tau \frac{\p^2 p(\xi)}{\p\xi\p\bar\xi}\big)\Zstpxi 
Y^{|J|-k-2}\Htp(r,\xi,w)\, dA(\xi) dr \\
&= \int_0^s \int_\C X\sh_{\xi} 
\Htp(s-r,z,\xi) X^{\alpha_1}_w 
X^{\alpha_2}\tau \frac{\p^2 p(\xi)}{\p\xi\p\bar\xi}\Zstpxi 
Y^{|J|-k-2}\Htp(r,\xi,w)\, dA(\xi) dr
\end{align*}
or
\begin{align*}
\int_0^s &\int_\C \Htp(s-r,z,\xi) \Big(\prod_{\imath=0}^k Y_{|J|-\imath} \Big) 
\tau\frac{\p^2 p(\xi)}{\p\xi\p\bar\xi} \Zstpxi 
Y^{|J|-k-2}\Htp(r,\xi,w)\, dA(\xi) dr\\
&= \int_0^s \int_\C \Htp(s-r,z,\xi)X^{\alpha_1}_w  
\big( \tau \frac{\p^2 p(\xi)}{\p\xi\p\bar\xi}\big)\Zstpxi Y^{|J|-k-2}\Htp(r,\xi,w)\, dA(\xi) dr \\
&= -\int_0^s \int_\C  \Wstpxi\Big[ \tau \frac{\p^2 p(\xi)}{\p\xi\p\bar\xi}\Htp(s-r,z,\xi)\Big]
 X^{\alpha_1}_w Y^{|J|-k-2}\Htp(r,\xi,w)\, dA(\xi) dr
\end{align*}
In both integrals, the number of $\Mtpxiw$ terms remains $n$, but
the number of derivatives in $w$ and $\xi$
is $(\ell-1)$. 
This suggests that we
ought to induct in  $\ell$, so the $(n,\ell-1)$-case is covered by the induction 
hypothesis.

Our goal is pointwise estimates of $Y^J\Htp(s,z,w)$. 
The complicating factor is that the induction hypothesis
needs to include both a size estimate \emph{and} a cancellation condition.

%%%%%%%%%%55 PRELIMINARY CALCULATIONS
\subsection{Preliminary Computations}
It is convenient to use the shorthand $\A{jk}z = \frac{1}{j!k!}\frac{\p^{j+k} p(z)}{\p z^j\p\z^k}$.
With this notation,
\[
p(w) = \sum_{j,k\geq 0} \A{jk}z(w-z)^j\overline{(w-z)}{}^k.
\]
The following estimates will be useful.
% Lemma:  replacing \Lambda by 1/\tau
\begin{lem}\label{lem: Lambda(z,|xi-z|) replaced by 1/tau}
Let $c >0$ and $\ep>0$. 
With a decrease in $c$, we have the bounds 
\begin{enumerate}\renewcommand{\theenumi}{\alph{enumi}}
\item\vspace*{.07in} \label{eqn:a}$\displaystyle
e^{-c \frac{|\xi-w|^2}s} e^{-c \big(\frac{s}{\mu_p(\xi,1/\tau)^2}\big)^\ep} 
|\nabla^\ell_{\xi,w} e(w,\xi)|
\les  e^{-c \frac{|\xi-w|^2}s} e^{-c \big(\frac{s}{\mu_p(\xi,1/\tau)^2}\big)^\ep} \tau^{-1}
\min\{s^{-\frac \ell2-\frac 12}, \mu_p(\xi,\tfrac 1\tau)^{-\ell-1}\}$.

\item\vspace*{.07in}\label{eqn:b}
$\displaystyle
e^{-c \big(\frac{s}{\mu_p(\xi,1/\tau)^2}\big)^\ep} | \nabla^\ell_{\xi} \triangle p(\xi)|
\les  e^{-c \big(\frac{s}{\mu_p(\xi,1/\tau)^2}\big)^\ep} \tau^{-1}
\min\{s^{-\frac \ell2-1}, \mu_p(\xi,\tfrac 1\tau)^{-\ell-2}\}$.

\item \label{eqn:c}
\begin{multline*}
e^{-c \frac{|\xi-w|^2}s} e^{-c \big(\frac{s}{\mu_p(\xi,1/\tau)^2}\big)^\ep}
e^{-c \big(\frac{s}{\mu_p(w,1/\tau)^2}\big)^\ep} | \nabla^\ell_{\xi} \triangle p(\xi)|
\\ \les  e^{-c \frac{|\xi-w|^2}s}e^{-c \big(\frac{s}{\mu_p(w,1/\tau)^2}\big)^\ep}
e^{-c \big(\frac{s}{\mu_p(\xi,1/\tau)^2}\big)^\ep} \tau^{-1}
\min\{s^{-\frac \ell2-1}, \mu_p(w,\tfrac 1\tau)^{-\ell-2}\}
\end{multline*}

\item \label{eqn:d}
\begin{multline*}
e^{-c \frac{|\xi-w|^2}s} e^{-c \big(\frac{s}{\mu_p(w,1/\tau)^2}\big)^\ep} 
e^{-c \big(\frac{s}{\mu_p(\xi,1/\tau)^2}\big)^\ep}
| \nabla^\ell_{\xi,w} e(w,\xi)| \\
\les  e^{-c \frac{|\xi-w|^2}s} e^{-c \big(\frac{s}{\mu_p(\xi,1/\tau)^2}\big)^\ep} 
e^{-c \big(\frac{s}{\mu_p(w,1/\tau)^2}\big)^\ep}
\tau^{-1}\min\{s^{-\frac \ell2-\frac 12}, \mu_p(w,\tfrac 1\tau)^{-\ell-1}\}.
\end{multline*}
\end{enumerate}
\end{lem}

\begin{proof} First, 
\begin{equation}\label{eqn:e est}
|\nabla^\ell_{\xi,w} e(w,\xi)| 
\leq \sup_{\atopp{j+k-1-\ell\geq 0}{j,k\geq 1}}|\A{jk}{\xi}||\xi-w|^{j+k-1-\ell}.
\end{equation}
and 
\[
|\nabla^\ell_{\xi} \triangle p(\xi)| \sim |\A{jk}{\xi}|
\] for some $j,k$ satisfying
$\ell + 2 = j+k$ and $j,k\geq 1$.

Next, with a decrease in $c$,
\begin{equation}\label{eqn:using exp decay}
e^{-c \frac{|\xi-w|^2}s} e^{-c \big(\frac{s}{\mu_p(\xi,1/\tau)^2}\big)^\ep}
\leq C_{\alpha,\beta} \frac {s^\alpha}{|\xi-w|^{2\alpha}} \frac{\mu_p(\xi,1/\tau)^{2\beta}}{s^\beta}
e^{-c \frac{|\xi-w|^2}s} e^{-c \big(\frac{s}{\mu_p(\xi,1/\tau)^2}\big)^\ep}.
\end{equation}
Since $\Lambda\big(\xi, \mu_p(\xi, 1/\tau)\big) \sim 1/\tau$,
combining (\ref{eqn:e est}) and (\ref{eqn:using exp decay})
finishes the proof of (\ref{eqn:a}). 
The proof of (\ref{eqn:b}) is simpler.
$|\tau \A{jk}{\xi}| \leq \mu_p(\xi,1/\tau)^{-j-k}$. For the other inequality, use the previous
inequality and (\ref{eqn:using exp decay}) with $\alpha=0$ and $\beta = \frac 12(j+k)$. 

The proofs of  \eqref{eqn:c} and \eqref{eqn:d} are similar. We can write
$p(\xi) = \sum_{\alpha,\beta\geq 0} \A{\alpha\beta}{w}(\xi-w)^\alpha \overline{(\xi-w)}{}^{\beta}$, so
\begin{equation}\label{eqn:a expansion}
|\A{jk}{\xi}| = \big|c_{j,k} \sum_{\atopp{\alpha\geq j}{\beta\geq k}} c_{\alpha,\beta,j,k} 
\A{\alpha\beta}{w}(\xi-w)^{\alpha-j} \overline{(\xi-w)}{}^{\beta-k}\big|
\les \sup_{\atopp{\alpha \geq j}{\beta \geq k}} |\A{\alpha\beta}{w} |\xi-w|^{\alpha+\beta-j-k}. 
\end{equation}
Since $|\nabla^\ell\triangle p(\xi)| \sim |\A{jk}{\xi}|$ for some $j,k$ satisfying $j\geq 1$, $k\geq 1$, 
$j+k=\ell+2$,
it follows that (with a decrease in $c$)
\begin{align*}
 &\big|e^{-c \frac{|\xi-w|^2}s} e^{-c \big(\frac{s}{\mu_p(\xi,1/\tau)^2}\big)^\ep}
 e^{-c \big(\frac{s}{\mu_p(w,1/\tau)^2}\big)^\ep}\A{jk}{\xi}\big| \\
 &\leq e^{-c \frac{|\xi-w|^2}s} e^{-c \big(\frac{s}{\mu_p(\xi,1/\tau)^2}\big)^\ep}
 e^{-c \big(\frac{s}{\mu_p(w,1/\tau)^2}\big)^\ep}
\sup_{\atopp{\alpha\geq j}{\beta\geq k}} |\A{\alpha\beta}{w}||\xi-w|^{\alpha+\beta-j-k}
\frac{s^{\frac 12(\alpha+\beta-j-k)}}{|\xi-w|^{\alpha+\beta-j-k}} 
\frac{\mu_p(w,1/\tau)^{\alpha+\beta-j-k}}{s^{\frac 12(\alpha+\beta-j-k)}} \\
&\leq e^{-c \frac{|\xi-w|^2}s} e^{-c \big(\frac{s}{\mu_p(\xi,1/\tau)^2}\big)^\ep}
 e^{-c \big(\frac{s}{\mu_p(w,1/\tau)^2}\big)^\ep}
\frac{1}{\mu_p(w,1/\tau)^{j+k}} \Lambda(w,\mu_p(w,\tfrac 1\tau)) \\
&\sim e^{-c \frac{|\xi-w|^2}s} e^{-c \big(\frac{s}{\mu_p(\xi,1/\tau)^2}\big)^\ep}
 e^{-c \big(\frac{s}{\mu_p(w,1/\tau)^2}\big)^\ep} 
\frac{1}{\tau\mu_p(w,1/\tau)^{\ell+2}}.
\end{align*}
The other cases are handled similarly.
\end{proof}

% Proposition: almost skew symmetry of e(w,z)
\begin{prop}\label{prop:e(w,z)=e(z,w)}
\[
e(w,\xi) = -e(\xi,w) - \sjk \frac{1}{j!k!} \frac{\p^{j+k+1}p(w)}{\p w^j\p\w^{k+1}}
(\xi-w)^j \overline{(\xi-w)}{}^k 
= -\sum_{\atopp{j\geq 1}{k\geq 0}} \frac{1}{j!k!} \frac{\p^{j+k+1}p(w)}{\p w^j\p\w^{k+1}}
(\xi-w)^j \overline{(\xi-w)}{}^k.
\]
\end{prop}

\begin{proof}Since $p(\xi) = \sum_{m,n\geq 0} \frac{1}{m!n!}
\frac{\p^{m+n}p(w)}{\p w^m\p\w^n}(\xi-w)^m\overline{(\xi-w)}{}^n$, it follows that
\[
\frac{\p^{j+1}p(\xi)}{\p\xi^j\p\bar\xi} 
= \sum_{\atopp{m\geq j}{n\geq 1}} \frac{1}{(m-j)!(n-1)!}
\frac{\p^{m+n}p(w)}{\p w^m\p\w^n} (\xi-w)^{m-j}\overline{(\xi-w)}{}^{n-1}.
\]
Thus,
\[
e(w,\xi) = \sum_{m,n\geq 1} \Big(\sum_{j=1}^m \binom mj (-1)^j\Big)
\frac{1}{m!(n-1)!} \frac{\p^{m+n}p(w)}{\p w^m\p\w^n}
(\xi-w)^{m}\overline{(\xi-w)}{}^{n-1},
\]
the desired equality since $\sum_{j=1}^m \binom mj (-1)^j = -1$ if $m\geq 1$.
\end{proof}

Another useful equality  is   
\begin{equation}\label{eqn:ajkz vs. ajkw}
\A{jk}{z} = \frac{1}{j!k!} \sum_{\atopp{m\geq j}{n\geq k}} \A{mn}{w}
(z-w)^{m-j}\overline{(z-w)}{}^{n-k}.
\end{equation}
An immediate consequence of
(\ref{eqn:ajkz vs. ajkw}) is that $\Lambda(z,|w-z|) \leq e^2 \Lambda(w,|z-w|)$ and
\begin{equation}\label{eqn:Lam(w) approx Lam(z)}
\Lambda(z,|w-z|) \sim \Lambda(w,|z-w|).
\end{equation}
The $e^2$ appears because $\sjk\frac 1{j!k!}\leq e^2$.

% Corollary: bounding derivatives of e
\begin{cor}\label{cor:deriv of e bound} Let $s>0$ 
and $z,w\in\C$ so that
$|z-w|\leq s^{1/2}$. There exists a constant $C$ depending on $\deg p$ so that
if $\xi \in\C$, the following holds with a decrease in $c$:
\begin{enumerate}\renewcommand{\theenumi}{\roman{enumi}}
\item
$\displaystyle \big|e^{-c\frac{|z-\xi|^2}{s}} \nabla^m e(w,\xi) \big| \leq
C e^{-c\frac{|z-\xi|^2}{s}} s^{-(m+1)/2}\Lambda(z,s^{1/2})$.

\item
$\displaystyle \big|e^{-c\frac{|z-\xi|^2}{s}} \Lambda(\xi, s^{1/2}) \big| 
\leq
C e^{-c\frac{|z-\xi|^2}{s}} \Lambda(z,s^{1/2})$.
\end{enumerate}
\end{cor}

\begin{proof} Proof of (i).
If $|\xi-z|>|\xi-w|$, then by (\ref{eqn:ajkz vs. ajkw})
\begin{align*}
\big|e^{-c\frac{|z-\xi|^2}{s}}\nabla^m e(w,\xi) \big|
&\leq e^{-c\frac{|z-\xi|^2}{s}}\sum_{j\geq m+1} |\A{j1}{\xi}||\xi-z|^{j-m}\\
&\leq e^{-c\frac{|z-\xi|^2}{s}+1} \sum_{j+k\geq m+1}|\A{jk}{z}| |\xi-z|^{j+k-m-1} 
\les e^{-c\frac{|z-\xi|^2}{s}} s^{-\frac 12(m+1)}\Lambda(z,s^{\frac 12}).
\end{align*}
with a decrease in the constant $c$.
If $|\xi-w|> 2|z-w|$, then $|\xi-w|\sim |\xi-z|$, and we can use the same argument
just given. If $ |\xi-z| < |\xi-w| < 2|z-w|$, then
$\frac 12|z-w| < |\xi-w| < 2|z-w|$. Therefore, by Proposition
\ref{prop:e(w,z)=e(z,w)} and (\ref{eqn:ajkz vs. ajkw}),
\[
\big|\nabla^m e(w,\xi) \big|
\les \sum_{j+k\geq m+1} |\A{jk}{w}||z-w|^{j+k-m-1}
\sim \sum_{j+k\geq m+1} |\A{jk}{z}||z-w|^{j+k-m-1} \leq s^{-\frac 12(m+1)}
\Lambda(z,s^{1/2}).
\]
(ii) is proved using similar methods, namely with (\ref{eqn:ajkz vs. ajkw}) and
the exponential decay.
\end{proof}

%%%%%%%%%%%%%%%%%%%%%%%%%%%%%%%%%%%%%%%
%
%	SECTION: CANCELLATION CONDITIONS FOR Y^J \Htp^s[\vp]
%
%
%%%%%%%%%%%%%%%%%%%%%%%%%%%%%%%%%%%%%%%%%5
\section{Cancellation Conditions and Size Estimates for $Y^J e^{-s\Boxtp}$}
\label{sec:size and cancel for Y^J Htp^s}

We start by defining objects which will allow us to control the numerology in the induction.
% DEFINITION: (n,\ell) size conditions
\begin{defn}\label{defn:(n,ell) size condition}
Let $n$ and $\ell$ be nonnegative integers. We say that $\Htp(s,z,w)$ satisfies the 
$(n,\ell)$\emph{-size conditions} if there exist positive constants $C_{|J|}$, $c$, and 
$\ep = \ep(n,\ell)$ 
so that for any $Y^J\in(n,\ell)$, 
\begin{equation}\label{eqn:(n,ell) size def}
|Y^J \Htp(s,z,w)| \leq C_{|J|}\frac{\Lambda(z,\Delta)^n} {s^{1+\frac 12\ell}} 
e^{-c \frac{|z-w|^2}s}
e^{-c \big(\frac{s}{\mu_p(z,1/\tau)^2}\big)^\ep} 
e^{-c \big(\frac{s}{\mu_p(w,1/\tau)^2}\big)^\ep}.
\end{equation}
We say that $\Htp(s,z,w)$ satisfies the $(n,\infty)$\emph{-size conditions} if
$\Htp(s,z,w)$ satisfies the $(n,\ell)$-size conditions for all integers $\ell\geq 0$.
\end{defn}
Notice that the $(n,\ell)$-size condition is not as good as the decay 
in Theorem \ref{thm:Boxtp heat estimates}.
It is, however, what the induction argument yields.

%Definition: (n,\ell) cancellation conditions
\begin{defn}\label{defn:(n,ell) cancellation conditions}
Let $n$ and $\ell$ be nonnegative integers
and $Y^J\in (n,\ell)$.
We say that $\Htp^s$ satisfies the 
\emph{$(n,\ell)$-cancellation conditions} if there exists a constant $C_{\ell,n}$ so that
for any
$(s,z)\in (0,\infty)\times\C$ with $\delta\leq\max\{\mu_p(z,1/\tau), s^{\frac 12}\}$
and   $\vp\in \cic{D(z,\delta)}$, then
\begin{equation}\label{eqn:(n,ell) cancel def even}
|Y^J \Htp^s[\vp](z)| 
= \bigg|\int_{\C} Y^J\Htp(s,z,w)\vp(w)\, dA(w)\bigg|
\leq C_{|J|}\frac{\Lambda(z,\Delta)^n}{\delta} 
\big(\big\| \Boxtp^{\frac \ell2}\vp \big\|_{L^2(\C)} 
+ \delta^2 \big\| \Boxtp^{\frac \ell2+1} \vp\big\|_{L^2(\C)} \big)
\end{equation}
if $\ell$ is even and
\begin{equation}\label{eqn:(n,ell) cancel def odd}
|Y^J \Htp^s[\vp](z)| 
= \bigg|\int_{\C} Y^J\Htp(s,z,w)\vp(w)\, dA(w)\bigg|
\leq C_{|J|}\frac{\Lambda(z,\Delta)^n}\delta \big(\delta \big\| \Boxtp^{\frac {\ell+1}2}\vp \big\|_{L^2(\C)} 
+ \delta^3 \big\| \Boxtp^{\frac {\ell+1}2+1} \vp\big\|_{L^2(\C)} \big).
\end{equation}
if $\ell$ is odd.
We say that $\Htp^s$ satisfies the $(n,\infty)$\emph{-cancellation conditions}  if
$\Htp^s$ satisfies the $(n,\ell)$-cancellation conditions for all $\ell\geq 0$.
\end{defn}
If $\vp$ satisfies the hypotheses of the test function in Definition 
\ref{defn:(n,ell) cancellation conditions}, we will call $\vp$ a cancellation test function.

%%%%%%%%%%%55 REDUCTION OF THE PROBLEM
\subsection{Reduction of the General Case.}
As discussed above, we will prove the size and cancellation conditions
by inducting on both $n$ and $\ell$. In this spirit, we reduce the problem from 
analyzing $Y^J\Htp(s,z,w)$ and $Y^J \Htp^s[\vp]$ for a general
$Y^J\in(n,\ell)$ to $Y^J$ of the form 
\begin{equation}\label{eqn:acceptable Y^J}
Y^J = X^{J_1} (\Mtpzw)^n.
\end{equation}

Since Theorem \ref{thm:Boxtp heat estimates} and Theorem \ref{thm:Boxtp cancellation conditions}
are already proved for $n=0$ \cite{Rai06h,Rai07}, 
we can assume that
$n\geq 1$. Moreover, $Y^J = (\Mtpzw)^n$ is already in the desired form, so 
we may assume that $\ell\geq 1$ and $Y^J$ is not in the form $X^{J_1} (\Mtpzw)^n$.

For $k<\ell$, assume that $\Htp^s$ and $\Htp(s,z,w)$ satisfy the $(n-1,k)$-cancellation
and size conditions, respectively.
Since $n,\ell\geq 1$, we can write $Y^J = Y^{I_1}\Mtpzw X Y^{I_2}$. The technique for handling
$X = \Zstpz$ or $\Zbstpz$ is the same, so we will assume that
$X = \Zstpz$. In this case, by Proposition \ref{prop:box, deriv commutator},
\[
Y^{I_1}\Mtpzw \Zstpz Y^{I_2} =Y^{I_1}\Zstpz\Mtpzw Y^{I_2} 
+ Y^{I_1}\overline{e(w,z)}Y^{I_2}.
\]
The second term can be written as a sum of
$(n-1,k)$-derivatives where $k\leq \ell-1$. 
Using Lemma \ref{lem: Lambda(z,|xi-z|) replaced by 1/tau}, it is straightforward
exercise to use the $(n-1,k)$-size and cancellation conditions to check that
$Y^{I_1}\overline{e(w,z)}Y^{I_2}\Htp^s$ and
$Y^{I_1}\overline{e(w,z)}Y^{I_2}\Htp(s,z,w)$ also satisfy the 
$(n,\ell)$-cancellation and size conditions.
We iterate the commutation process for $Y^{I_1}\Zstpz\Mtpzw Y^{I_2}$, dealing
with all of the error terms from the commutators as before. Thus,
we can reduce $Y^J$ to the form in (\ref{eqn:acceptable Y^J}).

%%%%%%%%5 COMPUTATION OF THE INTEGRAL OVER \C
\subsection{Computation of $\int_\C\Htp(s,z,\xi)Y^J\Htp(\ep,\xi,w)\, dA(\xi)$}
If $Y^J = X^{J_1}(\Mtpzw)^n$, then $X^{J_1} = U^\alpha_wX^\beta_z$
where $U = -X\sh$ and $X^{J_1}\in(0,|J_1|)$. The reason
that we reduce $Y^J$ to $U^\alpha_w X^\beta_z  (\Mtpzw)^n$ is the following theorem 
and  corollary. 
%Theorem: the bad integral is 0
\begin{thm}\label{thm:bad int is 0} 
Let  $\tau\in\R$, $n\geq 1$ and $Y^J = U^\alpha_w X^\beta_z(\Mtpzw)^n\in (n,\ell)$. Then
\[
\lim_{\ep\to 0} \int_\C \Htp(s,z,\xi) U^\alpha_w X^\beta_\xi(\Mtpxiw)^n \Htp(\ep,\xi,w)\, dA(\xi)
= (-1)^{|\beta|} U^\alpha_w\big[
r(w,\xi,z)^n U^\beta_\xi\Htp(s,z,\xi)\big]\Big|_{\xi=w}
\]
where $r(w,\xi,z) = 2\Imm\big( \sjk \A{jk}\xi (w-\xi)^j \overline{(z-\xi)}{}^k\big)$.
\end{thm}

Using the $(0,\ell)$-bounds
from Theorem \ref{thm:Boxtp heat estimates} and Theorem \ref{thm:Boxtp cancellation conditions}
and Lemma \ref{lem: Lambda(z,|xi-z|) replaced by 1/tau}, the following corollary
shows that 
$\lim_{\ep\to 0} \int_\C \Htp(s,z,\xi) U^\alpha_w X^\beta_x(\Mtpxiw)^n \Htp(\ep,\xi,w)\, dA(\xi)$
satisfies the bounds for the $(n,\ell)$-size and
cancellation conditions. 
\begin{cor} \label{cor:bad int is 0}
Let $U^\alpha_w$, $X^\beta_z$, and $n$ be as in Theorem
\ref{thm:bad int is 0}. If
\[
f_{\tau p}(s,z,w) = (-1)^{|\beta|} U^\alpha_w\big[
r(w,\xi,z)^n U^\beta_\xi\Htp(s,z,\xi)\big]\Big|_{\xi=w},
\]
then:
\begin{enumerate}
\item $|f_{\tau p}(s,z,w)|$ is bounded by the right hand side of (\ref{eqn:(n,ell) size def}).
\item If $\vp$ is a cancellation test function, then
$\left|\int_\C f_{\tau p}(s,z,w)\vp(w)\, dA(w)\right|$ is bounded by
(\ref{eqn:(n,ell) cancel def even}) or (\ref{eqn:(n,ell) cancel def odd}), depending
on whether $|\alpha|+|\beta|$ is even or odd. 
\end{enumerate}
\end{cor}

The proof of Theorem \ref{thm:bad int is 0} is essentially combinatorial, and
we have to establish some facts first.

We have the following:
% Proposition: T(w,z) = T(w,v) + T(v,z) + ok
\begin{prop}\label{prop:good T lemma}
\[
T(w,z) = T(w,\xi) + T(\xi,z) - r(w,\xi,z).
\]
\end{prop}
To prove Proposition \ref{prop:good T lemma}, we need the following combinatorial fact.
% Lemma: combinatorial fact
\begin{lem}\label{lem:combinatorics} Fix integers $k$ and $n$ so that $0\leq k\leq n$. Then
\[
\sum_{j=k}^n (-1)^j \binom{n}{j}\binom{j}{k} = 
\begin{cases} (-1)^n &k=n \\ 0 &k<n \end{cases}
\]
\end{lem}
\begin{proof} Let $s(n,k) = \sum_{j=k}^n (-1)^j \binom{n}{j}\binom{j}{k}$. 
The $k=0$ case is standard. 
Indeed for $n>0$, the $k=0$ case follows from the expansion of $(x+y)^n$ with $x=1$ and $y=-1$.
Recall that $\binom{n+1}j = \binom{n}{j} + \binom{n}{j-1}$. If
$k\geq 1$, then
\begin{align*}
s(n+1,k) &= \sum_{j=k}^{n+1} (-1)^j \binom{j}{k} \left[ \binom{n}{j} + \binom{n}{j-1}\right]
= \sum_{j=k}^{n} (-1)^j\binom{n}{j} \binom{j}{k}  + \sum_{j=k}^{n+1} (-1)^j
\binom{n}{j-1} \binom{j}{k}\\
&= s(n,k) + \sum_{j=k-1}^n (-1)^{j-1} \binom{n}{j} \left(\binom jk + \binom j{k-1}\right)
= s(n,k) - \sum_{j=k-1}^n (-1)^j \binom nj \binom jk - s(n,k-1) \\
&= s(n,k) - s(n,k) - s(n,k-1).
\end{align*}
Thus,
$s(n,k) = - s(n,k-1) = \cdots = (-1)^k s(n-k,0)$, and the results follows from the $k=0$ case.
\end{proof}
With our combinatorial lemma in hand, we prove Proposition \ref{prop:good T lemma}.
% Proposition: T(w,z) = T(w,v) + T(v,z) +ok
\begin{proof}[Proof. (Proposition \ref{prop:good T lemma})]
Recall that $T(w,z) = -2\Imm\big( \sum_{j\geq 1} \frac{1}{j!} 
\frac{\p^j p(z)}{\p z^j} (w-z)^j\big)$. 
The strategy is to expand $\frac{\p^j p(z)}{\p z^j}$ about $\xi$.
$p(z) = \sum_{n,\ell\geq 0} \frac{1}{n!\ell!} \frac{\p^{n+\ell}p(\xi)}
{\p \xi^n \p\bar \xi^\ell}(z-\xi)^n
\overline{(z-\xi)}{}^\ell$. Then
\[
\frac{\p^j p(z)}{\p z^j} = \sum_{\atopp{n\geq j}{\ell\geq 0}} \frac{1}{(n-j)! \ell!}
\frac{\p^{n+\ell}p(\xi)}{\p \xi^n \p\bar \xi^\ell}(z-\xi)^{n-j} \overline{(z-\xi)}{}^\ell,
\]
so by Lemma \ref{lem:combinatorics},
\begin{align*}
&\sum_{j\geq 0} \frac{1}{j!} \frac{\p^j p(z)}{\p z^j} (w-z)^j
= \sum_{j\geq 0} \sum_{\atopp{n\geq j}{\ell\geq 0}} \frac{1}{(n-j)! \ell!}
\frac{\p^{n+\ell}p(\xi)}{\p \xi^n \p\bar \xi^\ell}(z-\xi)^{n-j}
\overline{(z-\xi)}{}^\ell\big((w-\xi)+(\xi-z)\big)^j\\
&= \sum_{j\geq 0}\sum_{\atopp{n\geq j}{\ell\geq 0}} \sum_{k=0}^j \binom nj \binom jk
\frac{1}{n!\ell!}\frac{\p^{n+\ell}p(\xi)}{\p \xi^n \p\bar \xi^\ell} (z-\xi)^{n-j}
\overline{(z-\xi)}{}^\ell (w-\xi)^k (\xi-z)^{j-k} \\
&= \sum_{n,\ell \geq0} \sum_{k=0}^n \left( \sum_{j=k}^n (-1)^j\binom nj \binom jk\right) 
\frac{(-1)^k}{n!\ell!} \frac{\p^{n+\ell}p(\xi)}{\p \xi^n \p\bar \xi^\ell} (z-\xi)^{n-k}
\overline{(z-\xi)}{}^\ell (w-\xi)^k \\
&=\sum_{n,\ell \geq0} \frac{1}{n!\ell!} \frac{\p^{n+\ell}p(\xi)}{\p \xi^n \p\bar \xi^\ell} 
\overline{(z-\xi)}{}^\ell (w-\xi)^n\\
&= 
\sum_{n\geq0} \frac{1}{n!} \frac{\p^{n}p(\xi)}{\p \xi^n} (w-\xi)^n 
+ \sum_{\ell \geq0} \frac{1}{\ell!} \frac{\p^{\ell}p(\xi)}{\p\bar \xi^\ell} 
\overline{(z-\xi)}{}^\ell 
+ \sum_{n,\ell \geq1} \frac{1}{n!\ell!} \frac{\p^{n+\ell}p(\xi)}{\p \xi^n \p\bar \xi^\ell} 
\overline{(z-\xi)}{}^\ell (w-\xi)^n   
\end{align*}
From \cite{Rai06f}, $T(w,\xi) = -2\Imm\big( \sum_{j\geq 0} \frac{1}{j!} 
\frac{\p^j p(\xi)}{\p \xi^j} (w-\xi)^j\big)$
and $T(z,\xi) = -T(\xi,z)$, and we have
\begin{align*}
T(w,z) &= T(w,\xi) - T(z,\xi) - 2 \Imm \Big( \sjk \A{jk}{\xi} 
(w-\xi)^j \overline{(z-\xi)}{}^k\Big)\\
&= T(w,\xi)+ T(\xi,z) -r(w,\xi,z).
\end{align*}
\end{proof}

The following combinatorial results will help us  with the bookkeeping in the proof of Theorem \ref{thm:bad int is 0}
% Combinatorial fact
\begin{prop}\label{prop:combinatorics with the a's}
For $n\geq 0$, let $0\leq k \leq n$ and $0\leq j \leq n-k$. 
Let $\{\gamma_j^{n,k}\}$ be a  set of numbers so that
$\gamma_0^{0,0}=1$ and $\gamma_{-1}^{n, k} = \gamma_j^{n,-1}=0$ for all $j,k,n$.
If $\gamma_j^{n,k}$ satisfy the recursion relation 
\[
\gamma_j^{n,k} = \gamma_j^{n-1,k} - \gamma_{j-1}^{n-1,k} - \gamma_j^{n-1,k-1},
\]
then
\[
\gamma_j^{n,k} = (-1)^{j+k} \binom n k \binom{n-k}j.
\]
\end{prop}

\begin{proof}We induct on $n\geq 1$. $\gamma_0^{1,0} =1$, $\gamma_1^{1,0}=-1$, and
$\gamma_0^{1,1}=-1$, as predicted. Assume the result holds at level $(n-1)$.
Then
\begin{align*}
\gamma_j^{n,k} 
&= (-1)^{j+k} \binom{n-1}{k}\binom{n-k-1}{j}
- (-1)^{j+k+1} \binom{n-1}{k}\binom{n-k-1}{j-1}
- (-1)^{j+k-1} \binom{n-1}{k-1}\binom{n-k}{j} \\
&=(-1)^{j+k} \frac{(n-1)!}{(k-1)!(n-j-k-1)!(j-1)!}\left(
\frac{1}{jk} + \frac{1}{k(n-j-k)} + \frac{1}{j(n-k-j)}\right)\\
&= (-1)^{j+k}\binom n k \binom{n-k}j. 
\end{align*}
\end{proof}

We are now ready to prove Theorem \ref{thm:bad int is 0}.
\begin{proof}[Proof. (Theorem \ref{thm:bad int is 0})]
The plan is to strip away $\Mtp$ terms from $\Htp(\ep,\xi,w)$. We cannot, however, integrate
by parts since there is no $\tau$-integral. We can, however, use
the product rule and Proposition \ref{prop:good T lemma} to effectively transfer
$\Mtp$ away from $\Htp(\ep,\xi,w)$. For clarity, since $\Mtp$ will be applied to three
different terms, we will use $\Mtp^{u,v}$ to denote 
$e^{i\tau T(v,u)}\frac{\p}{\p\tau} e^{-i\tau T(v,u)}$. The proof is based on a repeated
application of the following process.
\begin{align}
f(z,\xi,\tau) &\Mtp^{\xi,w}\Htp(\ep,\xi,w)
= \frac{\p}{\p\tau}\Big(f(z,\xi,\tau) \Htp(\ep,\xi,w)\Big)
- \frac{\p}{\p\tau}f(z,\xi,\tau)\, \Htp(\ep,\xi,w) \nn\\
&- iT(w,z)f(z,\xi,\tau)\Htp(\ep,\xi,w) + iT(\xi,z)f(z,\xi,\tau)\Htp(\ep,\xi,w)
-r(w,\xi,z)f(z,\xi,\tau)\Htp(\ep,\xi,w) \nn\\
&= \Mtp^{z,w}\Big(f(z,\xi,\tau) \Htp(\ep,\xi,w)\Big)
- \Mtp^{z,\xi}f(z,\xi,\tau)\, \Htp(\ep,\xi,w)-r(w,\xi,z)f(z,\xi,\tau)\Htp(\ep,\xi,w).
\label{eqn:fake integration by parts in tau}
\end{align}
We now integrate by parts and
use (\ref{eqn:fake integration by parts in tau}) repeatedly.
\begin{align}
&\int_\C \Htp(s,z,\xi) U^\alpha_w  X^\beta_\xi(\Mtp^{\xi,w})^n\Htp(\ep,\xi,w)\, dA(\xi)
= (-1)^{|\beta|}U^\alpha_w \gamma_0^{0,0}
\int_\C U^\beta_\xi\Htp(s,z,\xi) (\Mtp^{\xi,w})^n\Htp(\ep,\xi,w)\, dA(\xi)\nn\\
&= (-1)^{|\beta|}U^\alpha_w \gamma_0^{0,0} \bigg(
\Mtp^{z,w}\int_\C U^\beta_\xi\Htp(s,z,\xi) (\Mtp^{\xi,w})^{n-1}\Htp(\ep,\xi,w)\, dA(\xi) \nn\\
&+ \int_\C\Big(-\Mtp^{z,\xi}U^\beta_\xi\Htp(s,z,\xi) 
- r(w,\xi,z)U^\beta_\xi\Htp(s,z,\xi)\Big) (\Mtp^{\xi,w})^{n-1}
\Htp(\ep,\xi,w)\, dA(\xi) \bigg)\nn\\
&=(-1)^{|\beta|}\sum_{k=0}^1\sum_{j=0}^{1-k} \gamma_{j}^{1,k} U^\alpha_w
(\Mtp^{z,w})^{1-j-k}\int_\C r(w,\xi,z)^k (\Mtp^{z,\xi})^{j}U^\beta_\xi\Htp(s,z,\xi) 
(\Mtp^{\xi,w})^{n-1}\Htp(\ep,\xi,w)\, dA(\xi) \nn\\
&= \cdots = (-1)^{|\beta|}\sum_{k=0}^n\sum_{j=0}^{n-k} \gamma_{j}^{n,k} U^\alpha_w
(\Mtp^{z,w})^{n-j-k}\int_\C r(w,\xi,z)^k(\Mtp^{z,\xi})^{j}U^\beta_\xi\Htp(s,z,\xi) 
\Htp(\ep,\xi,w)\, dA(\xi)\nn\\
&= (-1)^{|\beta|}\sum_{k=0}^n\sum_{j=0}^{n-k} \gamma_{j}^{n,k} \overline{ X^\alpha_w
(\Mtp^{w,z})^{n-j-k}e^{-\ep\Boxtp}\big[ \overline{r(w,\cdot,z)}{}^k
(\Mtp^{\cdot,z})^{j}X^\beta_\xi\Htp(s,\cdot,z)\big](w)}.\label{eqn:M int M problem}
\end{align}
The problem with (\ref{eqn:M int M problem}) is that we cannot commute $\Mtp$ across
$e^{-\ep\Boxtp}$. However, we can control the error term caused by the commutation. For a function
$f =f_\tau(z,w)$, we have
\begin{align*}
&\Mtpzw e^{-\ep\Boxtp} [f_\tau(\cdot,w)](z) \\
&= e^{-\ep\Boxtp}[\Mtp^{\cdot,w} f_\tau(\cdot,w)](z)
+ \ep \Mtpzw\Big[\Big( \frac{e^{-\ep\Boxtp}-I}\ep\Big) [f_\tau(\cdot,w)](z)\Big] 
+ \ep \Big(\frac{I-e^{-\ep\Boxtp}} {\ep} \Big) [\Mtp^{\cdot,w}f_\tau(\cdot,w)](z) \\
&= e^{-\ep\Boxtp}[\Mtp^{\cdot,w} f_\tau(\cdot,w)](z) 
+ \ep \Big[\Mtpzw,\Big( \frac{e^{-\ep\Boxtp}-I}\ep\Big)\Big][f(\cdot,w)](z)
\end{align*}
By the spectral theorem, $\lim_{\ep\to 0}\frac{e^{-\ep \Boxtp}-I}\ep = \Boxtp$,
and $[\Boxtp,\Mtp^{z,\xi}]$ applied to a derivative of $\Htp(s,z,\xi)$ is well-controlled. Thus,
\begin{equation}
\label{eqn:exp M flip flop}
\lim_{\ep\to0} \Mtpzw\big(e^{-\ep\Boxtp}[\Htp(s,\cdot,w)](z)\big)
= \lim_{\ep\to0} \big(e^{-s\Boxtp}[\Mtp^{\cdot,w} \Htp(s,\cdot,w)](z)\big) = \Mtpzw \Htp(s,z,w).
\end{equation}
By a repeated use of (\ref{eqn:exp M flip flop}), taking the
limit as $\ep\to 0$ in (\ref{eqn:M int M problem}), we have
\begin{align*}
&\lim_{\ep\to 0}\int_\C \Htp(s,z,\xi) U^\alpha_w  X^\beta_\xi
(\Mtp^{\xi,w})^n\Htp(\ep,\xi,w)\, dA(\xi)\\
&= \lim_{\ep\to 0} (-1)^{|\beta|}\sum_{k=0}^n\sum_{j=0}^{n-k} \gamma_{j}^{n,k} 
\overline{ X^\alpha_w
(\Mtp^{w,z})^{n-j-k}e^{-\ep\Boxtp}\big[ \overline{r(w,\cdot,z)}{}^k
(\Mtp^{\cdot,z})^{j}X^\beta_\xi\Htp(s,\cdot,z)\big](w)}\\
&=\lim_{\ep\to 0} \Bigg[
(-1)^{|\beta|}\sum_{k=0}^n\sum_{j=0}^{n-k} \gamma_{j}^{n,k} \overline{ X^\alpha_w
e^{-\ep\Boxtp}\big[ \overline{r(w,\cdot,z)}{}^k
(\Mtp^{\cdot,z})^{n-k}X^\beta_\xi\Htp(s,\cdot,z)\big](w)} + \ep(OK)\Bigg]\\
&=\lim_{\ep\to 0} 
(-1)^{|\beta|}\sum_{k=0}^n\sum_{j=0}^{n-k} \gamma_{j}^{n,k} \overline{ X^\alpha_w
e^{-\ep\Boxtp}\big[ \overline{r(w,\cdot,z)}{}^k
(\Mtp^{\cdot,z})^{n-k}X^\beta_\xi\Htp(s,\cdot,z)\big](w)}.
\end{align*}
Since $\lim_{\ep\to 0}e^{-\ep\Boxtp} = I$,
\begin{multline*}
\lim_{\ep\to 0} 
(-1)^{|\beta|}\sum_{k=0}^n\sum_{j=0}^{n-k} \gamma_{j}^{n,k} \overline{ X^\alpha_w
e^{-\ep\Boxtp}\big[ \overline{r(w,\cdot,z)}{}^k
(\Mtp^{\cdot,z})^{n-k}X^\beta_\xi\Htp(s,\cdot,z)\big](w)} \\
= (-1)^{|\beta|}\sum_{k=0}^n\sum_{j=0}^{n-k} \gamma_{j}^{n,k} U^\alpha_w\big[
r(w,\xi,z)^k(\Mtp^{z,w})^{n-k}X^\beta_\xi\Htp(s,z,\xi)\big]\Big|_{\xi=w}.
\end{multline*}
By Proposition \ref{prop:combinatorics with the a's} and Lemma \ref{lem:combinatorics},
\[
\sum_{j=0}^{n-k} \gamma_{j}^{n,k} = (-1)^k\binom nk\sum_{j=0}^{n-k} (-1)^{j} \binom {n-k}j
= \delta_0(n-k).
\]
This means $k=n$, so
\[
(-1)^{|\beta|}\sum_{k=0}^n\sum_{j=0}^{n-k} \gamma_{j}^{n,k} U^\alpha_w\big[
r(w,\xi,z)^k(\Mtp^{z,w})^{n-k}X^\beta_\xi\Htp(s,z,\xi)\big]\Big|_{\xi=w}
\hspace*{-9pt}= (-1)^{|\beta|} U^\alpha_w\big[
r(w,\xi,z)^n X^\beta_\xi\Htp(s,z,\xi)\big]\Big|_{\xi=w}.
\]
\end{proof}

\begin{remark}\label{rem:first int ok}
We have showed that the first term in (\ref{eqn:H^J int}) satisfies the bounds
of Theorem \ref{thm:Boxtp heat estimates} and Theorem \ref{thm:Boxtp cancellation conditions},
so we can now concentrate solely on the double integral term. 
\end{remark}
It turns out
that for the remainder of this estimate, 
we will not need (or use) the fact that we can restrict
ourselves to the case $Y^J = U^\alpha_w X^\beta_z (\Mtpzw)^n$.

%subsection: cancellation conditions
\subsection{Cancellation Conditions for $\Htp^s$} \label{subsec:cancel cond}
Unless explicitly stated, we now assume that $\tau>0$ for the remainder of the paper.
The cancellation conditions for $\Htp^s$ are proven in stages.

%Lemma: (n,0) cancellation conditions
\begin{lem}\label{lem:(n,0) cancel cond} Let $n\geq 1$. 
If $\Htp^s$ satisfies the 
$(m,\infty)$-size and cancellation conditions for $0\leq m\leq n-1$, 
then $\Htp^s$ satisfies the $(n,0)$-cancellation condition. 
\end{lem}

\begin{proof} Fix $(s,z)\in (0,\infty)\times\C$. 
Let $ \delta\leq\max\{\mu_p(z,1/\tau), s^{\frac 12}\}$ and
$\vp\in \cic{D(z,\delta)}$.
Let $Y^J = (\Mtpzw)^n$.  From Proposition
\ref{prop:Y^J H_tp formula} and Remark \ref{rem:first int ok}, 
\[
(\Mtpzw)^n\Htp(s,z,w) = \int_{0}^s \int_\C \Htp(s-r,z,\xi) (\Mtpxiw)^k 
\big[ \Boxtpxi, \Mtpxiw \big] (\Mtpxiw)^{n-k-1} \Htp(r,\xi,w) \, dA(\xi)\, dr.
\]
However, $\big[\Mtpxiw, [\Boxtpxi,\Mtpxiw]\big] = -2|e(w,\xi)|^2$, so
\begin{align*}
(\Mtpxiw)^k \big[ \Boxtpxi, \Mtpxiw \big] (\Mtpxiw)^{n-k-1} 
&= (\Mtpxiw)^{k-1} \big[ \Boxtpxi, \Mtpxiw \big] (\Mtpxiw)^{n-k} -2|e(w,\xi)|^2 (\Mtpxiw)^{n-2} \\
&= \cdots = [\Boxtpxi,\Mtpxiw](\Mtpxiw)^{n-1} - 2j |e(w,\xi)|^2(\Mtpxiw)^{n-2}
\end{align*}
Thus,
\begin{align} \label{eqn:Mtpn computation}
(\Mtpzw)^n\Htp(s,z,w) &= n\int_0^s\int_\C \Htp(s-r,z,\xi)[\Boxtpxi,\Mtpxiw] (\Mtpxiw)^{n-1}\Htp(r,\xi,w)\, dA(\xi)\, dr \\
&-n(n-1) \int_0^s\int_\C \Htp(s-r,z,\xi) |e(w,\xi)|^2 \Htp(r,\xi,w)\, dA(\xi)\, dr. \nn
\end{align}
We now test against a test function. To estimate $\Mtp^{z,\cdot}H^s[\vp](z)$, we start with the second integral
from \ref{eqn:Mtpn computation}. We rewrite
\begin{multline} \label{eqn:(n,0)commutator cancel}
\int_0^s\int_\C \int_\C \Htp(s-r,z,\xi) |e(w,\xi)|^2 \Htp(r,\xi,w)\, \vp(w) \, dA(\xi)\, dA(w) \, dr \\
= \int_0^s\int_\C \Htp(s-r,z,\xi)(\Mtpxiw)^{n-2}\Htp^r[|e(\cdot,\xi)|^2\vp](\xi)\, dA(\xi)\, dr.
\end{multline}
We can estimate (\ref{eqn:(n,0)commutator cancel}) with the $(n-2,0)$-cancellation condition,
Corollary \ref{cor:deriv of e bound} and (in the case $\Delta = \mu_p(z,1/\tau)$)
Lemma \ref{lem: Lambda(z,|xi-z|) replaced by 1/tau}. We have
(with a possible decrease in $c$),
\begin{align*}
&\frac{1}{\delta}
\int_{0}^s\int_\C \frac{1}{s-r} e^{-c\frac{|z-\xi|^2}{s-r}}
e^{-c (\frac{s-r}{\mu_p(z,1/\tau)^2})^\ep} e^{-c (\frac{s-r}{\mu_p(\xi,1/\tau)^2})^\ep}
\Lambda(\xi,\Delta)^{n-2}
\big(\||e(\cdot,\xi)|^2\vp\|_{L^2} + \delta^2\|\Boxtp|e(\cdot,\xi)|^2\vp\|_{L^2}\big)\, dA(\xi) dr\\
&\leq \frac{\Lambda(z,\Delta)^{n-2}}{\delta}\int_{0}^s\int_\C \frac{1}{s-r} 
e^{-c\frac{|z-\xi|^2}{s-r}}
e^{-c (\frac{s-r}{\mu_p(z,1/\tau)^2})^\ep} e^{-c (\frac{s-r}{\mu_p(\xi,1/\tau)^2})^\ep}
\Big(\sup_{w\in D(z,\delta)}|e(w,\xi)|^2\big[
\|\vp\|_{L^2} + \delta^2\|\Boxtp\vp\|_{L^2}\big] \\
&+\delta^2\sup_{w\in D(z,\delta)}|\nabla e(w,\xi)|^2 \|\tnabla\vp\|_{L^2} 
+ \delta^2\sup_{w\in D(z,\delta)}\big|\nabla^2 |e(w,\xi)|^2\big|
\|\vp\|_{L^2}\Big)\, dA(\xi) dr.
\end{align*}
In the second line, we changed $\Lambda(\xi,\Delta)$ to $\Lambda(z,\Delta)$ and brought it outside of the integral.  
This is possible by reexpanding $\A{jk}{\xi}$ in terms of $\A{jk}{z}$ and using Lemma \ref{lem: Lambda(z,|xi-z|) replaced by 1/tau}.
If $\Delta = s^{\frac 12}$, we can apply Corollary \ref{cor:deriv of e bound} to 
attain the desired result.

The case $\Delta = \mu_p(z,1/\tau)$ requires a more delicate estimate. Note that
$\Lambda(z,\Delta) \sim 1/\tau$.
We bound $D^k |e(w,\xi)|^2$ for $k\leq 2$.
If $\frac 12 s^{\frac 12} \leq \mu_p(z,1/\tau)$ 
and $|\xi-z|\leq 2\mu_p(z,1/\tau)$, then for $w\in D(z,\delta)$,
$|\xi-w| \les \mu_p(z,1/\tau)$, so
\[
|D^k e(w,\xi)|^2 \les \frac{1}{\tau^2 s\delta^k}.
\]
If $|z-\xi| \geq 2\mu_p(z,1/\tau)$, then $|w-\xi|\sim |z-\xi|$. Also,
$\mu_p(z,1/\tau)\sim \mu_p(w,1/\tau)$ since $|z-w|\les \mu_p(z,1/\tau)$. By
Lemma \ref{lem: Lambda(z,|xi-z|) replaced by 1/tau} (use the argument of
Lemma \ref{lem: Lambda(z,|xi-z|) replaced by 1/tau} with $\mu_p(w,1/\tau)$ and
use the fact that $s^{\frac 12} \les \mu_p(w,1/\tau)$),
\[
e^{-c\frac{|w-\xi|^2}{s-r}}
e^{-c (\frac{s-r}{\mu_p(z,1/\tau)^2})^\ep} e^{-c (\frac{s-r}{\mu_p(\xi,1/\tau)^2})^\ep}
D^k |e(w,\xi)|^2 
\les 
e^{-c\frac{|w-\xi|^2}{s-r}}
e^{-c (\frac{s-r}{\mu_p(z,1/\tau)^2})^\ep} e^{-c (\frac{s-r}{\mu_p(\xi,1/\tau)^2})^\ep}
\frac{1}{\tau^2 s \delta^k}.
\]
Then we bound (\ref{eqn:(n,0)commutator cancel}) by 
\begin{align*} 
\frac{C_n}{\tau^n s \delta} (\|\vp\|_{L^2} + \delta^2\|\Boxtp\vp\|_{L^2})
\int_0^s\int_\C \frac{1}{s-r} e^{-c\frac{|z-\xi|^2}{s-r}}\, dA(\xi) dr
\leq \frac{C_n}{\tau^n \delta} (\|\vp\|_{L^2} + \delta^2\|\Boxtp\vp\|_{L^2}),
\end{align*}
the desired estimate.
If $\mu_p(z,1/\tau) \leq \delta \leq s^{\frac 12}$, the integral in
(\ref{eqn:(n,0)commutator cancel}) can be bounded as follows:
\begin{align*}
&\int_{\frac s2}^s \int_\C\int_\C \Htp(s-r,z,\xi) (\Mtpxiw)^{n-2} \Htp(r,\xi,w) 
|e(w,\xi)|^2 \vp(w)\, dA(\xi) dA(w) dr \\
&\les \int_{\frac s2}^s \int_\C\int_\C \frac{1}{s-r} e^{-c \frac{|z-\xi|^2}{s-r}}
\frac{1}{s \tau^{n-2}} e^{-c \frac{|\xi-w|^2}{s}} e^{-c \big(\frac{s}{\mu_p(\xi,1/\tau)^2}\big)^\ep}
e^{-c \big(\frac{s}{\mu_p(w,1/\tau)^2}\big)^\ep} |e(w,\xi)|^2 |\vp(w)|\, dA(\xi) dA(w) dr \\
&\les \int_{\frac s2}^s \frac{1}{\tau^n s\delta} \int_{\C}|\vp(w)|\, dA(w) 
\les \frac{1}{\tau^n} \|\vp\|_{L^2}. 
\end{align*}
The integral from $0$ to $\frac s2$ is estimated similarly. 

To estimate the first integral in \eqref{eqn:Mtpn computation} tested against a test function,
we concentrate on each term of $[\Boxtpxi,\Mtpxiw]$ from Proposition \ref{prop:box, deriv commutator} separately. 
They are handled analogously, and we will only discuss the $\Zstpxi(\Mtpxiw)^{n-1}$ term. Although the operator
$\Zstpxi(\Mtpxiw)^{n-1}$
is an operator of order $(n-1,1)$ and hence under control, 
if we naively applied the $(n-1,1)$-cancellation condition, 
the result would not satisfy the $(n,0)$-cancellation condition estimate. Derivatives
of too high an order would appear. Instead,
we bring the $ \Zstpxi$ onto the $\Htp(s-r,z,\xi)$ term.
\begin{multline}
\int_{0}^s\int_\C \int_\C \Htp(s-r,z,\xi)  e(w,\xi)
\Zstpxi (\Mtpxiw)^{n-1} \Htp(r,\xi,w) \, \vp(w)\, dA(w) dA(\xi) dr \\
= -\int_{0}^s\int_\C 
\hspace*{-1pt}\Wstpxi\Htp(s-r,z,\xi)\,
(\Mtpxiw)^{n-1} \Htp^r[e(\cdot,\xi)\vp](\xi) \, dA(\xi) dr \label{eqn:(n,0)Z cancel} \\
- \int_{0}^s\int_\C 
\Htp(s-r,z,\xi)\, (\Mtpxiw)^{n-1} 
\Htp^r\Big[\frac{\p e(\cdot,\xi)}{\p\xi}\vp\Big](\xi) \, dA(\xi) dr.
\end{multline}
The argument for
(\ref{eqn:(n,0)Z cancel}) is essentially a repeat of the argument for (\ref{eqn:(n,0)commutator cancel}).
Also, the $n=1$ case is handled easily using the arguments
for the $n\geq 2$ case (in fact, the proof of the $n=1$ case 
is contained in the proof of the $n=2$ case -- no 
commutators are needed).
\end{proof}

%%%%%%% Lemma: Christ lemma:
From \cite{Christ91}, we have
\begin{lem}\label{lem:christ} If $|z-\zeta| > \mu_p(\zeta,1/\tau)$, 
then there exist constants $C$, $M$, and $\delta = \frac2{\deg p}$ so that
\[
\frac {\mu_p(\zeta,1/\tau)}{\mu_p(z,1/\tau)}
\leq C\Big(\frac{|z-\zeta|}{\mu_p(\zeta,1/\tau)}\Big)^M
\qquad\text{and}\qquad 
\frac{\mu_p(z,1/\tau)}{\mu_p(\zeta,1/\tau)}
\leq C\Big(\frac{|z-\zeta|}{\mu_p(\zeta,1/\tau)}\Big)^{1-\delta}.
\]
\end{lem}

In \cite{Christ91}, Christ only finds the exists of $\delta>0$ so that the second inequality holds.
However, using reverse H\"older classes and the techniques of \cite{Shen96}, we can explicitly find
$\delta$. We omit the computation because we will only use that $\delta>0$; a quantitative
estimate of $\delta$ is not necessary for our work.

%Lemma: (n,\ell) cancellation condition
\begin{lem}\label{lem:(n,ell) cancel cond} Let $n\geq 1$ and $\ell\geq 1$ be integers.
If $\Htp^s$ satisfies the $(j,\infty)$-size and cancellation conditions 
for $0\leq j\leq n-1$ and the
$(n,k)$-size and cancellation conditions for $0\leq k\leq \ell-1$, 
then $\Htp^s$ satisfies the $(n,\ell)$-cancellation conditions.
\end{lem}

\begin{proof}
Let $Y^J\in(n,\ell)$ and $\vp\in\cic{D(z,\delta)}$ where 
$\delta \leq\max\{\mu_p(z,1/\tau),s^{1/2}\}$.
We start by reducing $Y^J\Htp^s[\vp]$ into integrals for which our inductive hypothesis
is valid.

By Proposition
\ref{prop:Y^J H_tp formula}, we must estimate
\[
\bigg| \int_{0}^s\int_\C \int_\C \Htp(s-r,z,\xi) \Big(\prod_{\imath=0}^k Y_{|J|-\imath}\Big)
\big[ \Boxtpxi, Y_{|J|-k-1}\big] Y^{J-k-2} \Htp(r,\xi,w) \, \vp(w)\, dA(w) dA(\xi) dr\bigg|.
\]
Let $\Big(\prod_{\imath=0}^k Y_{|J|-\imath}\Big) = Y^K$ (so $|K|=k+1$). The commutator
$\big[ \Boxtpxi, Y_{|J|-k-1}\big]$ is nonzero only if $Y_{|J|-k-1} = \Mtpzw$ or $X_\xi$.
If $Y_{|J|-k-1} = \Mtpzw$, by Proposition \ref{prop:box, deriv commutator},
the integral to estimate is
\begin{equation}\label{eqn:Mtp cancel integral}
\bigg| \int_{0}^s\int_\C \int_\C \Htp(s-r,z,\xi) Y^K
\Big(-\frac{\p^2 p(\xi)}{\p\xi\p\bar\xi} - e(w,\xi)\Zstpxi + \overline{e(w,\xi)}\Zbstp\Big)
Y^{J-k-2} \Htp(r,\xi,w) \, \vp(w)\, dA(w) dA(\xi) dr \bigg|.
\end{equation}
When $Y_{|J|-k-1} = X_\xi$, the integral to bound can be written as 
\begin{equation}\label{eqn:Xxi cancel integral}
\bigg| \tau\int_{0}^s\int_\C \int_\C \Htp(s-r,z,\xi) Y^K
\Big(c_1\frac{\p^2 p(\xi)}{\p\xi\p\bar\xi}\Zstpxi + c_2\frac{\p^2 p(\xi)}{\p\xi\p\bar\xi}\Zbstpxi
+ c_3 \frac{\p^3 p(\xi)}{\p\xi\p\bar\xi^2}\Big)
Y^{J-k-2} \Htp(r,\xi,w) \, \vp(w)\, dA(w) dA(\xi) dr \bigg|
\end{equation}
where $c_i$, $i=1,2,3$ are constants that depend on whether $X_\xi = \Zbstpxi$, $\Zstpxi$, etc. 
In the case $Y_{|J|-k-1} = \Mtpzw$, the operator 
$Y^K Y^{J-k-2}\in (n-1,\ell)$ so even after factoring 
$\big(-\frac{\p^2 p(\xi)}{\p\xi\p\bar\xi} + e(w,\xi)\Zstpxi + \overline{e(w,\xi)}\Zbstp\big)$
into the derivative, the derivative is at worst an $(n-1,\ell+1)$-derivative and covered
by the induction hypothesis. In the second case, $Y^K Y^{J-k-2}\in (n,\ell-1)$, so
after taking $\big(c_1\frac{\p^2 p(\xi)}{\p\xi\p\bar\xi}\Zstpxi 
+ c_2\frac{\p^2 p(\xi)}{\p\xi\p\bar\xi}\Zbstpxi
+ c_3 \frac{\p^3 p(\xi)}{\p\xi\p\bar\xi^2}\big)$
into account, the derivative can be an $(n,\ell)$-derivative. Thus, we cannot immediately
use the induction hypothesis. We can, however, integrate by parts to bring a $\xi$-derivative
onto the $\Htp(s-r,z,\xi)$-term and use our size and cancellation conditions to estimate the
integral. 

Fortunately, the estimations of
(\ref{eqn:Mtp cancel integral}) and (\ref{eqn:Xxi cancel integral}) 
can be done in a similar fashion, so we
only present the case $Y_{|J|-k-1} = \Mtpxiw$. In 
(\ref{eqn:Mtp cancel integral}), the commutator $[\Boxtpxi,\Mtpxiw]$ creates a
sum of three terms. Each of these terms can be estimated with the same techniques, so
we only demonstrate the estimate of
\begin{align*}
&\int_0^s \int_\C\int_\C \Htp(s-r,z,\xi) Y^K \big(e(w,\xi)\Zstpxi Y^{J-k-2} 
\Htp(r,\xi,w)\big)\vp(w)\, dA(w) dA(\xi) dr \\
&= \sum_{|K_1|+|K_2|=k+1} c_{K_1,K_2}\int_0^s \int_\C\int_\C \Htp(s-r,z,\xi) D^{K_1}_\xi e(w,\xi)
Y^{K_2}\Zstpxi Y^{J-k-2}\Htp(r,\xi,w) \vp(w)\, dA(w) dA(\xi) dr.
\end{align*}
To integrate by parts, observe
that we can write
$Y^{K_2}\Zstpxi = Y^{\alpha_1} X_{\xi} Y^{\alpha_2}$ where $Y^{\alpha_1}$ is composed only of
$\Mtpxiw$, $\Wbstpw$, and $\Wstpw$. This means $X_\xi$ is the first $\xi$-derivative.
Of course, $X_\xi$
commutes with $\Wbstpw$ and $\Wstpw$. Also, $[\Mtpxiw,X_\xi] = e(w,\xi)$ or $\overline{e(w,\xi)}$,
we can commute $X_\xi$ by $\Mtpxiw$ with an error of $e(w,\xi)$. Thus, with the convention
that $Y_{|J|+1} = 1$,
\begin{align*}
Y^{\alpha_1} X_{\xi}& Y^{\alpha_2} = Y_{|J|}\cdots Y_{|J|-|\alpha_1|+1} X_\xi Y^{\alpha_2}\\
&= X_{\xi}Y_{|J|}\cdots Y_{|J|-|\alpha_1|+1}Y^{\alpha_2} + \sum_{n=0}^{|\alpha_1|-1}
Y_{|J|}\cdots Y_{|J|-n+1}\big[Y_{|J|-n},X_{\xi}\big] Y_{|J|-n-1}\cdots
 Y_{|J|-|\alpha_1|+1}Y^{\alpha_2}.
\end{align*}
Consequently,
\begin{align*}
&\int_0^s \int_\C\int_\C \Htp(s-r,z,\xi) D^{K_1}_\xi e(w,\xi)
Y^{K_2}\Zstpxi Y^{J-k-2}\Htp(r,\xi,w) \vp(w)\, dA(w) dA(\xi) dr \\
&= \int_0^s \int_\C\int_\C \Htp(s-r,z,\xi) D^{K_1}_\xi e(w,\xi)
X_{\xi}Y_{|J|}\cdots Y_{|J|-|\alpha_1|+1}Y^{\alpha_2}Y^{J-k-2}
\Htp(r,\xi,w)\, \vp(w)\, dA(w) dA(\xi) dr \\
&+ \sum_{\jmath=0}^{|\alpha_1|-1}\int_0^s \int_\C\int_\C \Htp(s-r,z,\xi) D^{K_1}_\xi e(w,\xi)
Y_{|J|}\cdots Y_{|J|-\jmath+1}\big[Y_{|J|-\jmath},X_{\xi}\big]\times \\
&\hspace{2in} Y_{|J|-\jmath-1}\cdots Y_{|J|-|\alpha_1|+1}Y^{\alpha_2}Y^{J-k-2}
\Htp(r,\xi,w)\, \vp(w)\, dA(w) dA(\xi) dr.
\end{align*}
The integrals in sum can be handled using the size and cancellation conditions from
the induction hypotheses in the same manner as the first integral (after we integrate
by parts in the first integral). We only show the computation for the first (and most
difficult) integral.
\begin{align*}
\bigg|& \int_0^s \int_\C\int_\C \Htp(s-r,z,\xi) D^{K_1}_\xi e(w,\xi)
X_{\xi}Y_{|J|}\cdots Y_{|J|-|\alpha_1|+1}Y^{\alpha_2}Y^{J-k-2}
\Htp(r,\xi,w)\, \vp(w)\, dA(w) dA(\xi) dr \bigg| \\
=\bigg|& \int_0^s \int_\C\int_\C \big( X_\xi\sh \Htp(s-r,z,\xi)D^{K_1}_\xi e(w,\xi) 
+ \Htp(s-r,z,\xi)D_\xi D^{K_1}_\xi e(w,\xi)\big)\times\\
&\hspace{2.5in} Y_{|J|}\cdots Y_{|J|-|\alpha_1|+1}Y^{\alpha_2}Y^{J-k-2}
\Htp(r,\xi,w)\, \vp(w)\, dA(w) dA(\xi) dr \bigg|
\end{align*}
The two integrals can be estimated with the same size and cancellation condition
argument, and we will show only the estimate of 
\begin{equation}\label{eqn:n,ell int to est}
\bigg| \int_0^s \int_\C\int_\C X_\xi\sh \Htp(s-r,z,\xi)D^{K_1}_\xi e(w,\xi) 
Y_{|J|}\cdots Y_{|J|-|\alpha_1|+1}Y^{\alpha_2}Y^{J-k-2}\Htp(r,\xi,w)
\, \vp(w)\, dA(w) dA(\xi) dr \bigg|.
\end{equation}
Note that $Y_{|J|}\cdots Y_{|J|-|\alpha_1|+1}Y^{\alpha_2}Y^{J-k-2}\in (n-1, \ell-|K_1|)$.
(If we had chosen $Y_{|J|-k-1} = X_\xi$, we would have a similar integral
with a derivative of $\triangle p$ replacing $e(w,\xi)$ and 
$Y_{|J|}\cdots Y_{|J|-|\alpha_1|+1}Y^{\alpha_2}Y^{J-k-2}\in (n, \ell-1-|K_1|)$).
Our induction hypothesis applies. We break the
$s$-integral in (\ref{eqn:n,ell int to est}) into two pieces and estimate each piece
separately. 

We show the argument for $\ell$ and $|K_1|$ even, but the cases when
at least of $\ell$ and
$\ell-|K_1|$ are odd is done similarly.
For $0\leq r \leq s/2$, $(s-r) \sim s$, so
\begin{align*}
&\bigg| \int_0^{s/2} \int_\C\int_\C X_\xi\sh \Htp(s-r,z,\xi)D^{K_1}_\xi e(w,\xi) 
Y_{|J|}\cdots Y_{|J|-|\alpha_1|+1}Y^{\alpha_2}
Y^{J-k-2}\Htp(r,\xi,w)\, \vp(w)\, dA(w) dA(\xi) dr \bigg| \\
&=\bigg|\int_0^{s/2}\int_\C X_\xi\sh\Htp(s-r,z,\xi)
Y_{|J|}\cdots Y_{|J|-|\alpha_1|+1}Y^{\alpha_2}Y^{J-k-2} 
\Htp^r\big[D_\xi^{K_1}e(\cdot,\xi)\vp\big](\xi)\, dA(\xi) dr\\
&\les \int_0^{s/2} \int_\C \frac{1}{s^{3/2}} e^{-c\frac{|z-\xi|^2}s}
e^{-c\frac{s}{\mu_p(z,1/\tau)^2}}e^{-c\frac{s}{\mu_p(\xi,1/\tau)^2}} 
\frac{\Lambda(\xi,\Delta)^{n-1}}{\delta} 
\big( \|\Boxtp^{\frac 12(\ell-|K_1|)}\big(D^{K_1}_\xi e(\cdot,\xi)\vp\big)\|_{L^2}\\
&\hspace{3truein} + \delta^2\|\Boxtp^{\frac 12(\ell-|K_1|)+1}\big(D^{K_1}_\xi 
e(\cdot,\xi)\vp\big)\|_{L^2} \big)\, dA(\xi) dr.\\
\end{align*}
The two terms are handled similarly, and we show the estimate of 
\[
\int_0^{s/2} \int_\C \frac{1}{s^{3/2}} e^{-c\frac{|z-\xi|^2}s}
e^{-c\frac{s}{\mu_p(z,1/\tau)^2}}e^{-c\frac{s}{\mu_p(\xi,1/\tau)^2}} 
\frac{\Lambda(\xi,\Delta)^{n-1}}{\delta} 
\|\Boxtpw^{\frac 12(\ell-|K_1|)}\big(D^{K_1}_\xi e(\cdot,\xi)\vp\big)\|_{L^2}\, dA(\xi) dr.
\]
Since
\[
\Boxtpw^{\frac 12(\ell-|K_1|)}\big(D^{K_1}_\xi e(\cdot,\xi)\vp\big)
= \sum_{|\gamma_1|+\gamma_2| =\ell-|K_1|} c_{\gamma_1,\gamma_2}
D^{\gamma_1}D^{K_1}e(w,\xi) X^{\gamma_2}\vp(w),
\]
it follows that
\begin{align}
&\int_0^{s/2} \int_\C \frac{1}{s^{3/2}} e^{-c\frac{|z-\xi|^2}s}
e^{-c\frac{s}{\mu_p(z,1/\tau)^2}}e^{-c\frac{s}{\mu_p(\xi,1/\tau)^2}} 
\frac{\Lambda(\xi,\Delta)^{n-1}}{\delta} 
\|\Boxtp^{\frac 12(\ell-|K_1|)}\big(D^{K_1}_\xi e(\cdot,\xi)\vp\big)\|_{L^2}\, dA(\xi) dr\nn\\
&\les
\sum_{|\gamma_1|+|\gamma_2|=\ell-|K_1|} 
\int_0^{s/2} \int_\C \frac{1}{s^{3/2}} e^{-c\frac{|z-\xi|^2}s}
e^{-c\frac{s}{\mu_p(z,1/\tau)^2}}e^{-c\frac{s}{\mu_p(\xi,1/\tau)^2}} 
\frac{\Lambda(\xi,\Delta)^{n-1}}{\delta} 
\|D^{\gamma_1}D^{K_1}e(\cdot,\xi) X^{\gamma_2}\vp(\cdot)\|_{L^2}\, dA(\xi) dr \label{eqn:gam sum}
\end{align}
To estimate $\Lambda(x,\Delta)$ and $D^{\gamma_1}D^{K_1}e(w,\xi)$, we turn 
Lemma \ref{lem: Lambda(z,|xi-z|) replaced by 1/tau}, Corollary \ref{cor:deriv of e bound}, 
the proofs of these two results, and (\ref{eqn:a expansion}).
With a decrease in $c$, we can bound
\[
e^{-c\frac{|z-\xi|^2}s}
e^{-c\frac{s}{\mu_p(z,1/\tau)^2}}e^{-c\frac{s}{\mu_p(\xi,1/\tau)^2}} 
\Lambda(\xi,\Delta)^{n-1}
\les e^{-c\frac{|z-\xi|^2}s}
e^{-c\frac{s}{\mu_p(z,1/\tau)^2}}e^{-c\frac{s}{\mu_p(\xi,1/\tau)^2}} 
\Lambda(z,\Delta)^{n-1}.
\]
Also, again with a decrease in $c$, since $|z-w|\leq \delta$,
\begin{align*}
&e^{-c\frac{|z-\xi|^2}s}
e^{-c\frac{s}{\mu_p(z,1/\tau)^2}}e^{-c\frac{s}{\mu_p(\xi,1/\tau)^2}} 
|D^{\gamma_1}D^{K_1}e(\cdot,\xi)| \\
&\les e^{-c\frac{|z-\xi|^2}s}
e^{-c\frac{s}{\mu_p(z,1/\tau)^2}}e^{-c\frac{s}{\mu_p(\xi,1/\tau)^2}}
 \min\{ \Lambda(z,s^{1/2})s^{-\frac12(1+|K_1|+|\gamma_1|)}, 
\tau^{-1}\mu_p(\xi,1/\tau)^{-(1+|K_1|+|\gamma_1|)} \\
&\les e^{-c\frac{|z-\xi|^2}s}
e^{-c\frac{s}{\mu_p(z,1/\tau)^2}}e^{-c\frac{s}{\mu_p(\xi,1/\tau)^2}}
\Lambda(z,\Delta) \min\{ s^{-\frac12(1+|K_1|+|\gamma_1|)}, 
\mu_p(z,1/\tau)^{-(1+|K_1|+|\gamma_1|)} \}.
\end{align*}
In this previous string of inequalities, the first estimate uses
Lemma \ref{lem: Lambda(z,|xi-z|) replaced by 1/tau} while the second
inequality is justified with Lemma \ref{lem:christ} and a reduction of $c$
in the exponent to control terms of the form $\big(|z-\xi|/\mu_p(z,1/\tau)\big)^M$.
Thus, choosing an arbitrary term from (\ref{eqn:gam sum}), we estimate
\begin{align}
&\int_0^{s/2} \int_\C \frac{1}{s^{3/2}} e^{-c\frac{|z-\xi|^2}s}
e^{-c\frac{s}{\mu_p(z,1/\tau)^2}}e^{-c\frac{s}{\mu_p(\xi,1/\tau)^2}} 
\frac{\Lambda(\xi,\Delta)^{n-1}}{\delta} 
\|D^{\gamma_1}D^{K_1}e(\cdot,\xi) X^{\gamma_2}\vp(\cdot)\|_{L^2}\, dA(\xi) dr\nn\\
&\les \| X^{\gamma_2}\vp(\cdot)\|_{L^2}
\int_0^{s/2} \int_\C \frac{1}{s^{3/2}} e^{-c\frac{|z-\xi|^2}s}
e^{-c\frac{s}{\mu_p(z,1/\tau)^2}}e^{-c\frac{s}{\mu_p(\xi,1/\tau)^2}} 
\frac{\Lambda(\xi,\Delta)^{n-1}}{\delta} 
\sup_{w\in\supp\vp} |D^{\gamma_1}D^{K_1}e(w,\xi)|
\, dA(\xi) dr \label{eqn:good s int cancel condition}\\
&\les \|X^{\gamma_2}\vp(\cdot)\|_{L^2} \frac{\Lambda(z,\Delta)^{n}}
{\delta \max\{ s^{\frac12}, \mu_p(z,1/\tau) \}^{1+|K_1|+|\gamma_1|}} 
\int_0^{s/2} \int_\C \frac{1}{s^{3/2}} e^{-c\frac{|z-\xi|^2}s}
e^{-c\frac{s}{\mu_p(z,1/\tau)^2}}e^{-c\frac{s}{\mu_p(\xi,1/\tau)^2}} 
 \, dA(\xi) dr \nn\\
&\les \|X^{\gamma_2}\vp(\cdot)\|_{L^2} \frac{\Lambda(z,\Delta)^{n}}
{\delta \max\{ s^{\frac12}, \mu_p(z,1/\tau) \}^{1+|K_1|+|\gamma_1|}} s^{1/2}.
\end{align}
From \cite{Rai07}, we know $\|\vp\|_{L^2} \leq \sqrt2\delta \|X_j \vp\|_{L^2}$ for $j=1,2$. 
Also, from \cite{Rai06f}, if $X^\alpha\in(0,\ell)$, then
$\|X^\alpha \vp\|_{L^2} \sim \|\Boxtp^{\ell/2}\vp\|_{L^2}$.
Thus, since $|\gamma_1|+|\gamma_2| =\ell-|K_1|$,
\[
\|X^{\gamma_2}\vp(\cdot)\|_{L^2} \frac{\Lambda(z,\Delta)^{n}}
{\delta \max\{ s^{\frac12}, \mu_p(z,1/\tau) \}^{1+|K_1|+|\gamma_1|}} s^{1/2}
\les \frac{\Lambda(z,\Delta)^{n}} \delta \|\Boxtp^{\ell/2}\vp\|_{L^2},
\]
the desired estimate.

We have one integral remaining to estimate. If $\Delta = \sqrt s$, then we can use
the integral estimate as the $0\leq r \leq s/2$ case. We can follow the estimate line by
line, except for two differences. First, we cannot replace $s-r$ with $s$. However,
this is not an issue $e^{-c\frac{|z-\xi|^2}{s-r}} \leq e^{-c\frac{|z-\xi|^2}s}$, so the
the use of Corollary \ref{cor:deriv of e bound} to bound
$e^{-c\frac{|z-\xi|^2}{s-r}} |D^{\gamma_1}D^{K_1}e(w,\xi)|$ remains unchanged.
To bound
$e^{-c\frac{|z-\xi|^2}{s-r}} \Lambda(\xi,\Delta)^{n-1} \les
e^{-c\frac{|z-\xi|^2}{s-r}} \Lambda(z,\Delta)^{n-1}$, we have (using the arguments
of Lemma \ref{lem: Lambda(z,|xi-z|) replaced by 1/tau} and
(\ref{eqn:ajkz vs. ajkw}))
\begin{align*}
e^{-c\frac{|z-\xi|^2}{s-r}}\Lambda(\xi,\sqrt s) 
&= e^{-c\frac{|z-\xi|^2}{s-r}}\sjk |\A{jk}\xi| s^{(j+k)/2}
\les e^{-c\frac{|z-\xi|^2}{s-r}} \sjk \sum_{\atopp{\alpha\geq j}{\beta\geq k}}
|\A{\alpha\beta}z| |\xi-z|^{\alpha+\beta-j-k} s^{(j+k)/2} \\
&\leq e^{-c\frac{|z-\xi|^2}{s-r}} \sjk \sum_{\atopp{\alpha\geq j}{\beta\geq k}}
|\A{\alpha\beta}z| |\xi-z|^{\alpha+\beta-j-k} s^{(j+k)/2} \frac{s^{(\alpha+\beta-j-k)2}}
{|\xi-z|^{\alpha+\beta-j-k}} = e^{-c\frac{|z-\xi|^2}{s-r}}\Lambda(z,\sqrt s). 
\end{align*}
As usual, the bound comes with a price of a decrease in $c$. 
Last, the line of argument in (\ref{eqn:good s int cancel condition}) proceeds
line by line, replacing $s-r$ with $s$. Since $\Delta = \{\mu_p(z,1/\tau),\sqrt s\}$,
$\max\{ s^{\frac12}, \mu_p(z,1/\tau) \} = \mu_p(z,1/\tau)$.

Thus, the final integral to estimate is from $s/2$ to $s$ in the case
that $\Delta = \mu_p(z,1/\tau)$. In this case, $\delta \leq \sqrt s$
and $\Lambda(z,\Delta) \sim 1/\tau$.
We use size estimates to bound the integral.
% If $|z-w|\leq \mu_p(z,1\tau)$, then
% $\mu_p(w,1/\tau)\sim \mu_p(z,1/\tau)$ while if
% $\mu_p(w,1/\tau) \leq |z-w|\leq \sqrt s$, then by 
% Lemma \ref{lem:christ},
% \[
% \frac{1}{\mu_p(w,1/\tau)} = \frac{1}{\mu_p(z,1/\tau)} 
% \frac{\mu_p(z,1/\tau)}{\mu_p(w,1/\tau)} 
% \les \frac{1}{\mu_p(z,1/\tau)} \Big(\frac{s^{1/2}}{\mu_p(w,1/\tau)}\Big)^M
% \]
% for some $M$ ($M$<1 actually). Thus, for this integral
% $\frac{1}{\mu_p(w,1/\tau)} e^{-c (\frac{s}{\mu_p(w,1/\tau)^2})^\ep}
% \les \frac{1}{\mu_p(z,1/\tau)} e^{-c (\frac{s}{\mu_p(w,1/\tau)^2})^\ep}$.
Using Lemma \ref{lem: Lambda(z,|xi-z|) replaced by 1/tau}, we estimate
\begin{align*}
&\bigg| \int_{s/2}^s \int_\C\int_\C X_\xi\sh \Htp(s-r,z,\xi)D^{K_1}_\xi e(w,\xi) 
Y_{|J|}\cdots Y_{|J|-|\alpha_1|+1}Y^{\alpha_2}
Y^{J-k-2}\Htp(r,\xi,w)\, \vp(w)\, dA(w) dA(\xi) dr \bigg| \\
&\les
\int_\C |\vp(w)| \int_{s/2}^s \int_\C
\frac{1}{(s-r)^{3/2}} e^{-c\frac{|\xi-z|^2}{s-r}} |D^{K_1}e(w,\xi)|
\frac{ e^{-c\frac{|\xi-w|^2}{s}}}{\tau^{n-1} s^{\frac 12(1+\ell-|K_1|)}}
e^{-c (\frac{s}{\mu_p(\xi,1/\tau)^2})^\ep}e^{-c (\frac{s}{\mu_p(w,1/\tau)^2})^\ep}
\, dA(\xi)\hspace{.5pt} dr\hspace{1pt} dA(w) \\
&\les \int_\C |\vp(w)| \int_{s/2}^s \int_\C
\frac{1}{(s-r)^{3/2}} e^{-c\frac{|\xi-z|^2}{s-r}} \frac{s^{|K_1|/2}}{\mu_p(w,1/\tau)^{|K_1|+1}}
\frac{ e^{-c\frac{|\xi-w|^2}{s}}}{\tau^{n} s^{\frac 12(1+\ell)}}
e^{-c (\frac{s}{\mu_p(\xi,1/\tau)^2})^\ep}e^{-c (\frac{s}{\mu_p(w,1/\tau)^2})^\ep}
\, dA(\xi)\hspace{.5pt} dr\hspace{1pt} dA(w) \\
&\les \frac{1}{\tau^n \delta^{1+\ell}}
\int_\C |\vp(w)|\, dA(w) 
\leq \frac{1}{\tau^n \delta^{1+\ell}} \delta^{1/2} \|\vp_w\|_{L^2}
\les \frac{1}{\tau^n \delta} \|\Boxtp^{\ell/2}\vp\|_{L^2}.
\end{align*}
\end{proof}

% subsection: size conditions
\subsection{$(n,\ell)$-size estimates for $\Htp(s,z,w)$}

% Proposition: (n,\ell) size estimates
\begin{prop} \label{prop:(n,ell)-size}
Fix $(n,\ell)$, $0<n<\infty$ and $0\leq \ell < \infty$.
If $\Htp(s,z,w)$ satisfies the $(n',\infty)$-size and cancellation conditions for $0\leq n'<n$ 
and $(n,\ell')$-size and cancellation
conditions for $0\leq \ell'<\ell$, then $\Htp(s,z,w)$ satisfies
$(n,\ell)$-size conditions.
\end{prop}

\begin{proof} As above, to estimate $Y^J \Htp(s,z,w)$ for $J\in(n,\ell)$, it
suffices to estimate
\[
\int_0^s \int_\C \Htp(s-r,z,\xi) \sum_{k=0}^{|J|-2} \Big(\prod_{\imath =0}^k Y_{|J|-\imath}\Big) 
\big[ \Boxtpxi,Y_{|J|-k-2} \big] Y^{J-k-2} \Htp(r,\xi,w)\, dA(\xi) dr.
\]
Also, by the conjugate symmetry of $\Htp(s,z,w)$, i.e., $\Htp(s,z,w) = \overline{\Htp(s,w,z)}$,
it is enough to obtain the bound
\[
\frac{\Lambda(z,\Delta)^n}{s^{1+\frac12|\alpha|}}e^{-c \frac{|z-w|^2}s} 
e^{-c\big(\frac{s}{\mu_p(w,1/\tau)^2}\big)^\ep}
\]
for some $\ep>0$.

We will estimate the integral for a fixed $k$. Let $Y^K = \prod_{\imath =0}^k Y_{|J|-\imath}$ (so
$|K| = k+1$). The are three cases to consider: $Y_{J-k-2} = \Mtpxiw$, $\Zbstpxi$, or $\Zstpxi$. 
First, assume $Y_{J-k-2} = \Mtpxiw$. In this case, we must estimate
\[
\int_0^s \int_\C \Htp(s-r,z,\xi) Y^{K}
\Big( -\frac{\p^2 p(\xi)}{\p\xi \p\bar\xi} -e(w,\xi)\Zstpxi + \overline{e(w,\xi)}\Zbstpxi\Big)
Y^{J-k-2} \Htp(r,\xi,w)\, dA(\xi) dr.
\]
All three terms are estimated similarly, so we only show the estimate for the
$e(w,\xi)\Zstpxi$ term. 
Note that
$Y^K \Zstpxi Y^{J-k-2}\in(n-1,\ell+1)$, so we can apply 
size and cancellation conditions. We can write
\[
Y^K\big[ e(w,\xi)\Zstpxi Y^{J-k-2}\Htp(r,\xi,w)\big]
= \sum_{|K_1|+|K_2|=K} c_{K_1,K_2} D^{K_1}_\xi e(w,\xi) Y^{K_2}\Zstpxi Y^{J-k-2}\Htp(r,\xi,w).
\]
It is enough to bound
\begin{equation}\label{eqn:double int to bound for Htp}
\left|\int_0^s \int_\C \Htp(s-r,z,\xi) D^{K_1}_\xi e(w,\xi) Y^{K_2}\Zstpxi Y^{J-k-2}
\Htp(r,\xi,w)\, dA(\xi) dr\right|.
\end{equation}
If $\frac s2 \leq r\les s$, then $r \sim s$, so by size estimates,  
Lemma \ref{lem: Lambda(z,|xi-z|) replaced by 1/tau}, and
Corollary \ref{cor:deriv of e bound},
\begin{align*}
\Big| \int_{\frac s2}^s &\int_\C \Htp(s-r,z,\xi) D^{K_1}_\xi e(w,\xi) 
Y^{K_2}\Zstpxi Y^{J-k-2}\Htp(r,\xi,w)\, dA(\xi) dr \Big| \\
&\les \int_{\frac s2}^s \int_\C \frac{1}{s-r} e^{-c \frac{|z-\xi|^2}{s-r}}
|D^{K_1}_\xi e(w,\xi)| \frac{e^{-c\frac{|w-\xi|^2}s} e^{-c(\frac{s}{\mu_p(w,1/\tau)^2})^\ep}
e^{-c(\frac{s}{\mu_p(\xi,1/\tau)^2})^\ep} \Lambda(\xi,s^{1/2})^{n-1}}
{s^{1+\frac 12(|K_2| + 1 + |J|-k-2 - (n-1))}} \, dA(\xi) dr \\
&\les \frac{\Lambda(z,\Delta)^n}{s^{1+\frac 12(|K_1|+1+\ell + |K_2|-k)}} 
e^{-c\big(\frac{s}{\mu_p(w,1/\tau)^2}\big)^\ep}
\int_{\frac s2}^s \int_\C \frac{1}{s-r} e^{-c \frac{|z-\xi|^2}{s-r}}
e^{-c\frac{|w-\xi|^2}s}\, dA(\xi) dr.
\end{align*}
Note that if $|z-\xi|\leq|w-\xi|$, then $|w-\xi| \geq \frac 12|z-w|$, and if
$|z-\xi|\geq |w-\xi|$, then $|z-\xi|\geq \frac 12|z-w|$. Thus, with  a slight decrease in $c$, 
\begin{align*}
&\frac{\Lambda(z,\Delta)^n}{s^{2+\frac\ell2}}
e^{-c\big(\frac{s}{\mu_p(w,1/\tau)^2}\big)^\ep}
\int_{\frac s2}^s \int_\C \frac{1}{s-r} e^{-c \frac{|z-\xi|^2}{s-r}}
e^{-c\frac{|w-\xi|^2}s}\, dA(\xi) dr\\
&\les \frac{\Lambda(z,\Delta)^n}{s^{2+\frac\ell2}}e^{-c\frac{|z-w|^2}s}
e^{-c\big(\frac{s}{\mu_p(w,1/\tau)^2}\big)^\ep}
\int_{\frac s2}^s \int_\C \frac{1}{s-r} e^{-c \frac{|z-\xi|^2}{s-r}}\, dA(\xi) dr
\leq \frac{\Lambda(z,\Delta)^n}{s^{1+\frac\ell2}}e^{-c\frac{|z-w|^2}s}
e^{-c\big(\frac{s}{\mu_p(w,1/\tau)^2}\big)^\ep},
\end{align*}
the desired estimate.
The estimate for $0\leq r \leq s/2$ is more delicate. 
Let $\delta = \frac 12 \min\{\mu_p(w, 1/\tau), s^{\frac 12}\}$ and
$\vp_w\in C^\infty_c(\C)$ where $\supp \vp_w \subset D(w,2\delta)$. Let $\vp_w \equiv 1$ on 
$D(w,\frac 12\delta)$,
$0\leq \vp_w \leq 1$, and $|\nabla^\beta \vp_w| \leq \frac{c_k}{\delta^|\beta|}$. 
The first integral we estimate is
\begin{align*}
\int_0^{\frac s2} \int_\C& \Htp(s-r,z,\xi) D^{K_1}e(w,\xi) Y^{K_2} \Zstpxi Y^{J-k-2}
\Htp(r,\xi,w)
\vp_w(\xi)\, dA(\xi) dr\\
&= \int_0^{\frac s2} Y^{K_2} \Zstpxi Y^{J-k-2} \overline{\Htp^r}
\big[\Htp(s-r,z,\cdot) D^{K_1}e(w,\cdot) \vp_w\big](w)\, dr.
\end{align*}
$Y^{K_2} \Zstpxi Y^{J-k-2} \in (n-1, |K_2|+1+|J|-k-2-(n-1))$ and 
$|K_2|+1+|J|-k-2-(n-1) = \ell -|K_1|+1$. We can assume that $\ell-|K_1|+1$ is even
since the odd
case is handled analogously.
By the $(n-1,\ell-|K_1|+1)$-cancellation condition and using the fact that $s-r \sim s$, 
\begin{align*}
&\Big| \int_0^{\frac s2} Y^{K_2} \Zstpxi Y^{J-k-2} \overline{\Htp^r}
\big[\Htp(s-r,z,\cdot) D^{K_1}e(w,\cdot) \vp_w\big](w)\, dr \Big| \\
&\leq c_n\frac{\Lambda(z,\Delta)^{n-1}}{\delta} \int_0^{\frac s2}
\Big( \big\|(\Boxtpxi\sh)^{\frac12(\ell - |K_1|+1)}
\big( \Htp(s-r,z,\cdot) D^{K_1}_\xi e(w,\cdot) \vp \big) \big\|_{L^2(\C)} \\
&\hspace{1in}+ \delta^2 \big\|(\Boxtpxi\sh)^{\frac12(\ell - |K_1|+1)+1}\big( \Htp(s-r,z,\cdot) 
D^{K_1}_\xi e(w,\cdot) \vp \big) \big\|_{L^2(\C)} \Big) \, dr
\end{align*}
Since the two terms can be estimated similarly, we estimate the first term. Since
$|\xi-w|< \frac 12 \mu_p(w,1/\tau)$, $\mu_p(w,1/\tau)\sim \mu_p(\xi,1/\tau)$
and $|e^{-c \frac{s}{\mu_p(w,1/\tau)^2}}D^\beta e(w,\xi)| \les 
e^{-c \frac{s}{\mu_p(w,1/\tau)^2}}\frac{1}{\tau s^{\frac 12(1+|\beta|)}}$. 
Also,
since $\delta < \frac12 s^{1/2}$,
$\frac{|z-\xi|}{s-r}\sim \frac{|z-w|}{s}$. 
Thus, by Lemma \ref{lem: Lambda(z,|xi-z|) replaced by 1/tau} and
Corollary \ref{cor:deriv of e bound},
\begin{align*}
&\frac{\Lambda(z,\Delta)^{n-1}}{\delta}  
\int_0^{\frac s2} \big\|(\Boxtpxi\sh)^{\frac12(\ell - |K_1|+1)}
\big( \Htp(s-r,z,\cdot) D^{K_1}_\xi e(w,\cdot) \vp \big) \big\|_{L^2(\C)} \, dr \\
&\frac{\Lambda(z,\Delta)^{n-1}}{\delta} 
 \int_0^{\frac s2} \sum_{|\alpha_1|+|\alpha_2|+|\alpha_3|=\ell+1-|K_1|}
\big\| U^{\alpha_1}_\xi \Htp(s-r,z,\cdot) D^{\alpha_2}_\xi D^{K_1}_\xi e(w,\cdot)
D^{\alpha_3}_\xi\vp_w \big\|_{L^2(\C)}\, dr \\
&\les\frac{\Lambda(z,\Delta)^{n-1}}{\delta} \int_0^{\frac s2} \Big\|
\frac{1}{s^{1+\frac 12|\alpha_1|}} e^{-c \frac{|z-\xi|^2}s}
e^{-c \frac{s}{\mu_p(z,1/\tau)^2}} e^{-c \frac{s}{\mu_p(\xi,1/\tau)^2}}
|D^{|\alpha_2|+|K_1|}e(w,\xi)| \Big\|_{L^2(\supp\vp_w)} \frac{1}{\delta^{|\alpha_3|}}\, dr \\
&\les \frac{\Lambda(z,\Delta)^n}{s^{1+ \frac 12(|\alpha_1|+|\alpha_2|+|K_1|+1)}}
\int_0^{\frac s2} e^{-c\frac{|z-w|^2}{s}}e^{-c \frac{s}{\mu_p(w,1/\tau)^2}} 
\frac{1}{\delta^{|\alpha_3|}}\, dr \\
&\les e^{-c\frac{|z-w|^2}{s}} e^{-c \frac{s}{\mu_p(w,1/\tau)^2}}
\frac{\Lambda(z,\Delta)^n}{s^{\frac 12(|\alpha_1|+|\alpha_2| 
+ |K_1|+1+|\alpha_3|)}} 
= \frac{\Lambda(z,\Delta)^n}{s^{1+ \frac 12\ell}} 
e^{-c\frac{|z-w|^2}{s}} e^{-c \frac{s}{\mu_p(w,1/\tau)^2}}.
\end{align*}
The lemma will be proved once we estimate 
\[
\int_0^{\frac s2} \int_\C \Htp(s-r,z,\xi) D^{K_1}e(w,\xi) Y^{K_2} \Zstpxi Y^{J-k-2} \Htp(r,\xi,w)
(1-\vp_w(\xi))\, dA(\xi) dr
\]
This estimate will rely on size estimates. Since $0 < r < \frac s2$, 
$e^{-c \frac{|\xi-w|^2}{r}}
< e^{-c \frac{|\xi-w|^2}{s}}$. Also,
\begin{equation}\label{eqn:2nd region est}
e^{-c \frac{|\xi-w|}{s}} \leq e^{-c \frac{s}{\mu_p(w,1/\tau)^2}}, \qquad
\text{if } s \leq |\xi-w|\mu_p(w, 1/\tau). 
\end{equation}
Thus, we have 
(at most) two regions to consider:
$\mu_p(w,1/\tau) \leq |\xi-w| \les \frac{s}{\mu_p(w,1/\tau)}$ and 
$|\xi-w| \ges \frac{s}{\mu_p(w,1/\tau)}$. The second region is not included in the first
region when $s$ is large (relative to $\mu_p(w,1/\tau)$).
On the second region, by Lemma \ref{lem: Lambda(z,|xi-z|) replaced by 1/tau},
Corollary \ref{cor:deriv of e bound}, \eqref{eqn:2nd region est}, and 
with a (possible) decrease
of $c$, we have
\begin{align*}
&\left|\int_0^{\frac s2} \int_{|\xi-w|\ges \frac s{\mu_p(w,1/\tau)}}
 \Htp(s-r,z,\xi) D^{K_1}e(w,\xi) Y^{K_2} \Zstpxi Y^{J-k-2} \Htp(r,\xi,w)
(1-\vp_w(\xi))\, dA(\xi) dr\right|\\
&\les 
\int_0^{\frac s2} \int_{\atopp{|\xi-w|\geq \mu_p(w,1/\tau)}
{|\xi-w|\ges \frac s{\mu_p(w,1/\tau)}}}
\frac{e^{-c\frac{|\xi-z|^2}{s-r}}}{s-r} e^{-c \frac{s}{\mu_p(\xi,1/\tau)^2}}
\frac {|D^{K_1}e(w,\xi)|\Lambda(\xi,\Delta)^{n-1}}{r^{\frac12(1+\ell-|K_1|)}}
\frac{e^{-c\frac{|\xi-w|^2}{r}}}{r} e^{-c \frac{s}{\mu_p(w,1/\tau)^2}} \, dA(\xi) dr \\
&\les \frac{\Lambda(z,\Delta)^n}{s^{1+\frac \ell2}} e^{-c\frac{|z-w|^2}s} 
e^{-c \frac{s}{\mu_p(w,1/\tau)^2}}
\int_0^{\frac s2}\int_{\C} \frac{e^{-c\frac{|\xi-z|^2}{s-r}}}{s-r} \frac{e^{-c\frac{|\xi-w|^2}{r}}}{r}\, dA(\xi) dr
\les \frac{\Lambda(z,\Delta)^n}{s^{1+\frac \ell2}} e^{-c\frac{|z-w|^2}s} 
e^{-c \frac{s}{\mu_p(w,1/\tau)^2}}.
\end{align*}
The key size estimates for the region 
$\mu_p(w,1/\tau) \leq |\xi-w| \les \frac{s}{\mu_p(w,1/\tau)}$ follow
from the second inequality in 
Lemma \ref{lem:christ} and $\frac{|\xi-w|}{\mu_p(w,1/\tau)} \les \frac{s}{\mu_p(w,1/\tau)^2}$. 
Since $|\xi-w| \geq \mu_p(w,1/\tau)$, $\frac{\mu_p(w,1/\tau)}
{\mu_p(\xi,1/\tau)} \ges \big( \frac{\mu_p(w,1/\tau)}{|w-\xi|}\big)^{1-\delta}$, so
with a decrease in $c$,
\begin{align}
e^{-c\big(\frac{s}{\mu_p(\xi,1/\tau)^2}\big)^{\ep}}
= e^{-c\big(\frac{s}{\mu_p(w,1/\tau)^2}
(\frac{\mu_p(w,1/\tau)}{\mu_p(\xi,1/\tau)})^2\big)^{\ep}}
&\ges e^{-c\big(\frac{s}{\mu_p(w,1/\tau)^2}
(\frac{\mu_p(w,1/\tau)}{|\xi-w|})^{2-2\delta}\big)^{\ep}} \nn\\
&\ges e^{-c\big(\frac{s}{\mu_p(w,1/\tau)^2}
(\frac{s}{\mu_p(w,1/\tau)})^{2\delta-2}\big)^{\ep}} 
= e^{-c\big(\frac{s}{\mu_p(w,1/\tau)^2}\big)^{2\delta \ep}} 
\label{eqn:ep to delta ep estimate}
\end{align}
With this estimate in hand, the size estimate follows from similar arguments as before. Indeed,
with a decrease in $c$,
\begin{align*}
&\left|\int_0^{\frac s2} \int_{\mu_p(w, 1/\tau)\leq |\xi-w|\les \frac s{\mu_p(w,1/\tau)}}
 \Htp(s-r,z,\xi) D^{K_1}e(w,\xi) Y^{K_2} \Zstpxi Y^{J-k-2} \Htp(r,\xi,w)
(1-\vp_w(\xi))\, dA(\xi) dr\right|\\
&\les \int_0^{\frac s2} \int_{\mu_p(w, 1/\tau)\leq |\xi-w|\les \frac s{\mu_p(w,1/\tau)}} 
\hspace*{-28pt}
\frac{e^{-c\frac{|\xi-z|^2}{s-r}}}{s-r} e^{-c \frac{s}{\mu_p(\xi,1/\tau)^2}}
e^{-c\big(\frac{s}{\mu_p(w,1/\tau)^2}\big)^{2\delta \ep}} 
  \frac {|D^{K_1}e(w,\xi)|\Lambda(\xi,\Delta)^{n-1}}{r^{1+\frac12(1+\ell-|K_1|)}}
e^{-c\frac{|\xi-w|^2}{r}} \, dA(\xi) dr \\
&\les \frac{\Lambda(z,\Delta)^n}{s^{1+\frac \ell2}} 
e^{-c\big(\frac{s}{\mu_p(w,1/\tau)^2}\big)^{2\delta \ep}}
\int_0^{\frac s2}\int_{\C} \frac {e^{-c\frac{|\xi-w|^2}{r}} }r  \frac{e^{-c\frac{|\xi-z|^2}{s-r}}}{s-r}\, dA(\xi) dr
\les \frac{\Lambda(z,\Delta)^n}{s^{1+\frac \ell2}} e^{-c\frac{|z-w|^2}s}
e^{-c\big(\frac{s}{\mu_p(w,1/\tau)^2}\big)^{2\delta \ep}}.
\end{align*}
Thus, in the case $Y_{J-k-2} = \Mtpxiw$, we have obtained the desired estimate.
 The remaining cases
use no new ideas. If $Y_{J-k-2} = \Zbstpxi$ or $\Zstpxi$, we must integrate by parts 
(for the term with the $\Zbstpxi$ or $\Zstpxi$) as in
Lemma \ref{lem:(n,ell) cancel cond} to put a $\Wbstpxi$ 
or $\Wstpxi$ on $\Htp(s-r,z,\xi)$. 
At which point, we can use the $(n,\gamma)$-size and cancellation conditions 
because the integration by parts guarantees
that $\gamma<\ell$. The integral estimation is then similar
to the one just performed.
\end{proof}
\begin{rem} The curious $\ep$ in the definition of the $(n,\ell)$-size condition
is explained by  \eqref{eqn:ep to delta ep estimate}.
\end{rem}

%subsection: Proof of cancellation condition
\subsection{Proof of Theorem \ref{thm:Boxtp cancellation conditions}}

\begin{proof}[Proof (Theorem \ref{thm:Boxtp cancellation conditions})] 
Proof by induction. The $(0,\infty)$-size and cancellation conditions are proved in
\cite{Rai06h,Rai07}. From Lemma \ref{lem:(n,0) cancel cond}, Lemma \ref{lem:(n,ell) cancel cond},
and Proposition \ref{prop:(n,ell)-size}, it is clear that 
$\Htp(s,z,w)$ satisfies the $(n,\ell)$-size and cancellation conditions for all $n$ and $\ell$.
Thus, Theorem \ref{thm:Boxtp cancellation conditions} is proved.
\end{proof}

%subsection: difficulty with the boxtp heat estimate proof
\subsection{Proof of Theorem \ref{thm:Boxtp heat estimates}}
From the proof of Theorem \ref{thm:Boxtp cancellation conditions}, we know
that $\Htp(s,z,w)$ satisfies the $(n,\ell)$-size and cancellation conditions
for all $n$ and $\ell$. Thus, the remaining estimate to show to finish the 
proof of Theorem \ref{thm:Boxtp heat estimates} is the improved long time
decay.
We will use
an integration by parts argument. 

Recall a Sobolev embedding lemma from \cite{Rai07}.
\begin{thm}\label{thm:Poincare} Let $\Delta = (a_1,b_1)\times(a_2,b_2)\subset \R^2$ be a square
of sidelength $\delta$. If $(x_1,x_2)\in \Delta$ and if $f\in\mathcal{C}^2(\Delta)$, 
then
\[
|f(x_1,x_2)|^2 \leq 4 \left( \frac{1}{\delta^2} \int_\Delta |f|^2 + \int_\Delta |\tnabla f|^2 
+ \delta^2 \int_\Delta |X_2X_1 f|^2\right).
\]
\end{thm}

Using the method of argument of the proof of Theorem \ref{thm:Poincare}, we show
\begin{cor}\label{cor:sobo/poinc}
Let $\Delta_1, \Delta_2 \subset \R^2$ be squares
of sidelength $\delta$, and let $I = (\tau_0-\gamma,\tau_0+\gamma)\subset\R$.
If $(z,w,\tau)\in \Delta_1\times\Delta_2\times I$ and if 
$f\in\mathcal{C}^4(\Delta_1\times\Delta_2\times I)$, then 
\[
|f(z,w,\tau)|^2 \leq
\frac{64}{\delta^4\gamma} \sum_{K=(k,j)\leq(1,4)} \delta^{2j} \gamma^{2k} 
\|Y^K f\|_{L^2(\Delta_1\times\Delta_2\times I)}^2.
\]
\end{cor}

\begin{proof} For $h\in \mathcal{C}^1(\R)$, $h(\tau) = h(x) + \int_x^\tau h'(\sigma)\, d\sigma$. 
Integrating
in $x$ over $I$, we have
\[
2\gamma |h(\tau)| \leq \int_{I} |h(\sigma)|\, d\sigma + 2\gamma \int_I |h'(\sigma)|\, d\sigma.
\]
Applying Cauchy-Schwarz and squaring, we have
\[
|h(\tau)|^2 \leq 4\bigg( \frac{1}{\gamma} \int_{I}|h(\sigma)|^2\, d\sigma
+ \gamma \int_{I} |h'(\sigma)|^2\, d\sigma\bigg).
\]
$T(w,z)$ is $\R$-valued, so if  $g(\sigma)
= e^{-i\sigma T\big(w,(x_1,x_2)\big)} h(\sigma)$, then
\begin{equation}\label{eqn:t,gamma est}
|g(\tau)|^2 \leq 4\bigg( \frac{1}{\gamma} \int_{I}|g(\sigma)|^2\, d\sigma
+ \gamma \int_{I} |M_{\sigma p}g(\sigma)|^2\, d\sigma\bigg).
\end{equation}
To finish the proof, we apply Theorem \ref{thm:Poincare} to $f$ twice: once in $z$ and
once  in $w$ (with $U_1$ and $U_2$ replacing $X_1$ and $X_2$) for each term from the
$z$ estimate. This gives
\[
|f(z,w,\tau)|^2 \leq \frac{16}{\delta^4} \sum_{K=(0,j), j\leq 4} 
\delta^{2|K|}\|Y^{K}f(\cdot,\cdot,\tau) \|_{L^2(\Delta_1\times\Delta_2)}^2.
\]
Applying \eqref{eqn:t,gamma est} to each term in the previous inequality finishes the proof.
\end{proof}
% Lemma: integration by parts
\begin{lem}\label{lem: integration by parts}
Let $z,w\in\C$ and $\tau\in\R$ and $J\in(n,\ell)$.
If $\delta,\gamma>0$, $B = D(z,\delta)\times D(w,\delta) \times (\tau-\gamma,\tau+\gamma)$,  and
$F\in \mathcal{C}^\infty(B)$, then there exits $C>0$ so that
\[
|Y^J F(z,w,\tau)|^2 \leq C \max_{K\in(k,j)\leq(2n+2,2\ell+8)} \delta^{j-2\ell} \gamma^{k-2n}
\| Y^K F\|_{L^{\infty}(B)} \|F\|_{L^\infty(B)}.
\]
\end{lem}

\begin{proof} Let $\psi \in \cic{\C\times\C\times\R}$ so that 
$\supp\psi\subset B$, $\psi(z,w,\tau)=1$, and 
$|\frac{\p^k}{\p\tau^k}\nabla^{j}\psi| \leq C_{j,k} \delta^{-j}\gamma^{-k}$. By
Corollary \ref{cor:sobo/poinc},
\[
|Y^J F(z,w,\tau)|^2 = |Y^J F(z,w,\tau) \psi(z,w,\tau)|^2
\les \frac{1}{\delta^4\gamma} \sum_{L=(k,j)\leq(1,4)} \delta^{2j} \gamma^{2k}
\big\| Y^L\big(Y^J F \psi\big)\big\|_{L^2(B)}.
\]
Since
\[
\big\| Y^L\big(Y^J F \psi\big)\big\|_{L^2(B)} \les \sum_{\atopp{J_1\in (n+k_1,\ell+j_1)}
{\atopp{0\leq j_1\leq j}{0\leq k_1\leq k}}} \big\| (Y^{J_1} F)\,  \nabla^{j-j_1} 
\frac{\p^{k-k_1}}{\p\tau^{k-k_1}} \psi \big\|_{L^2(B)}.
\]
Let $\chi_B$ be the characteristic function of $B$.
Taking an arbitrary term from the sum, we have
\begin{align*}
& \big\| (Y^{J_1} F)\,  \nabla^{j-j_1} 
\frac{\p^{k-k_1}}{\p\tau^{k-k_1}} \psi \big\|_{L^2(B)}^2 = 
\Big((Y^{J_1} F)\,  \nabla^{j-j_1}  \frac{\p^{k-k_1}}{\p\tau^{k-k_1}} \psi, (Y^{J_1} F)\,  
\nabla^{j-j_1} 
\frac{\p^{k-k_1}}{\p\tau^{k-k_1}} \psi\Big) \\
&= \bigg|\bigg( Y^{J_1}\Big[ (Y^{J_1} F)\,   \Big(\nabla^{j-j_1}  
\frac{\p^{k-k_1}}{\p\tau^{k-k_1}} \psi\Big)
\Big(\nabla^{j-j_1}  \frac{\p^{k-k_1}}{\p\tau^{k-k_1}} \psi \Big)\Big], 
F \chi_{B} \bigg)\bigg| \\
&\les \sum_{\atopp{K\in (n+k_1+k_1^1, \ell+j_1+j_1^1)}{\atopp{j_1^1+j_1^2+j_1^3=\ell+j_1}
{k_1^1+k_1^2+k_1^3=n+k_1}}} \bigg|\bigg( (Y^K F) 
 \Big(\nabla^{j-j_1+j_1^2}  \frac{\p^{k-k_1}}{\p\tau^{k-k_1+k_1^2}} \psi\Big)
\Big( \nabla^{j-j_1+j_1^3}  \frac{\p^{k-k_1}}{\p\tau^{k-k_1+k_1^3}} \psi\Big), 
F\chi_B \bigg)\bigg|.
\end{align*}
Thus, it is enough to estimate
\begin{align*} 
\frac{1}{\delta^4 \gamma} \delta^{2j}\gamma^{2k}  \bigg|\bigg( (Y^K F) 
&\Big(\nabla^{j-j_1+j_1^2}  \frac{\p^{k-k_1}}{\p\tau^{k-k_1+k_1^2}} \psi\Big)
\Big( \nabla^{j-j_1+j_1^3}  \frac{\p^{k-k_1}}{\p\tau^{k-k_1+k_1^3}} \psi\Big), 
F\chi_B \bigg)\bigg|\\
&\les \delta^{2j}\tau^{2k} \frac{1}{\delta^{j-j_1+j_1^2}\gamma^{k-k_1+k_1^2}}
\frac{1}{\delta^{j-j_1+j_1^3}\gamma^{k-k_1+k_1^3}}
\|Y^K F\|_{L^\infty(B)} \|F \|_{L^\infty(B)}.
\end{align*}
However, $\delta^{-2j + j-j_1+j_1^2+j-j_1+j_1^3} = \delta^{-2j_1+j_1^2+j_1^3} = 
\delta^{\ell-j_1-j_1^1}$ since $j_1^1+j_1^2+j_1^3 = \ell+j_1$. Similarly, since
$k_1^1+k_1^2+k_1^3 = n+k_1$,
$\gamma^{-2n +k-k_1+k_1^2+k-k_1+k_1^3}= \gamma^{n-k_1-k_1^1}$.  Thus,
\[
|Y^J F(z,w,\tau)|^2 \les \max_{\atopp{K\in(n+k_1+k_1^1,\ell+j_1+j_1^1)}
{\atopp{0\leq j_1\leq 4}{0\leq j_1^1 \leq \ell+4}}}
\frac{1}{\delta^{\ell-j_1-j_1^1}\gamma^{n-k_1-k_1^1}}
\|Y^K F\|_{L^\infty(B)} \|F \|_{L^\infty(B)}.
\]
\end{proof}

We will use Lemma \ref{lem: integration by parts} to prove
Theorem \ref{thm:Boxtp heat estimates}.
\begin{proof}[Proof (Theorem \ref{thm:Boxtp heat estimates})]
We now recover the superior estimates of Theorem \ref{thm:Boxtp heat estimates}. 
Fix $(s,z,w)\in (0,\infty)\times\C\times\C$. Since $\Htp(s,z,w)$ satisfies 
(\ref{eq:he i1}), it is enough to estimate $Y^J\Htp(s,z,w)$ where
$J\in (n,\ell)$ since $s$-derivatives can be written in terms of $\Boxtp$.

We already have the estimate
\begin{equation}\label{eqn:good H est 1}
|Y^\alpha \Htp(s,z,w)| \les 
\frac{\Lambda(z,s^{1/2})^n}{s^{1+\frac\ell2}} e^{-c\frac{|z-w|^2}s}.
\end{equation}
Since $\frac 1\tau > \Lambda(z,s^{1/2})$ means
$\mu_p(z,1/\tau) > s^{1/2}$, (\ref{eqn:good H est 1}) is the estimate
in the $\tau$-small case. Thus, we have left to show
\begin{equation*}\label{eqn:good H est 2}
|Y^\alpha \Htp(s,z,w)| \les 
\frac{1}{\tau ^n s^{1+\frac\ell2}} e^{-c\frac{|z-w|^2}s}
e^{-c\frac{s}{\mu_p(z,1/\tau)^2}}e^{-c\frac{s}{\mu_p(w,1/\tau)^2}}.
\end{equation*}

Let  $\delta = \frac 12\Delta$ and
$\gamma = \frac 14\tau$. Since $\mu_p(z,r) \sim \mu_p(z,2r)$ for all $z$ with a constant
depending only on $\deg p$, it follows that $\mu_p(\xi,\sigma)\sim \mu_p(z,1/\tau)$ for
all $(\xi,\sigma)\in D(z,\delta)\times (\tau-\gamma,\tau+\gamma)$.
Thus, using the notation of Lemma \ref{lem: integration by parts},
we have $F(z,w,\tau) = \Htp(s,z,w)$, and 
$\| Y^K \Htp\|_{L^\infty(B)} \sim |Y^K\Htp(s,z,w)|$. We bound
$|\Htp(s,z,w)|$ from the known $(0,\ell)$-estimate of Theorem \ref{thm:Boxtp heat estimates}
from \cite{Rai06h}  and
$|Y^K\Htp(s,z,w)|$ from Proposition \ref{prop:(n,ell)-size}.
Since $\delta\leq s^{1/2}$, with a decrease in $c$,
\begin{align*}
|Y^J\Htp(s,z,w)| &\les \max_{K\in(k,j)\leq (2n+2,2\ell+8)} 
\delta^{\frac j2-\ell}\tau^{\frac k2-n}
|Y^K\Htp(s,z,w)|^{\frac 12} |\Htp(s,z,w)|^{\frac 12} \\
&\les \max_{K\in(k,j)\leq (2n+2,2\ell+8)} \delta^{\frac j2-\ell}\tau^{\frac k2-n}
\frac{1}{\tau^{k/2}s^{1+j/2}} e^{-c\frac{|z-w|^2}s} e^{-c\frac{s}{\mu_p(w,1/\tau)^2}}
e^{-c\frac{s}{\mu_p(z,1/\tau)^2}} \\
&\les \frac{1}{\tau^n s^{1+\frac 12\ell}}e^{-c\frac{|z-w|^2}s} e^{-c\frac{s}{\mu_p(w,1/\tau)^2}}
e^{-c\frac{s}{\mu_p(z,1/\tau)^2}}.
\end{align*}
\end{proof}

%%%%%%%%%%%%%%%%%%%%%%%%%%%%%%%
%
%
%	SECTION: SIZE ESTIMATES FOR G_{\tau p}(s,z,w)
%
%
%%%%%%%%%%%%%%%%%%%%%%%%%%%%%%%%5
\section{Size estimates for $Y^\alpha \Gwtp(s,z,w)$}
\label{sec:Gwtp size est}

We now turn to the proof of Theorem \ref{thm:G wig and deriv}.
\begin{proof}[Proof. (Theorem \ref{thm:G wig and deriv})]
Since $\Gwtp(s,z,w) = \overline{\Gwtp(s,w,z)}$,
it is enough to show
\[
|Y^\alpha \Gwtp(s,z,w)| \leq \frac{C_{|\alpha|}}{\tau^n} e^{-c\frac s{\mu_p(w,1/\tau)^2}}
\max\Big\{ \frac{e^{-c\frac{|z-w|^2}s}}{s^{1+\frac 12\ell}},
\frac{e^{-c \frac{|z-w|}{\mu_p(z,1/\tau)}}}
{\mu_p(z,1/\tau)^{2+\ell}} \Big\}.
\]
From Proposition \ref{prop:integral relating Gwtp and Htp}, 
\[
Y^\alpha \Gwtp(s,z,w) 
= - \int_\C Y^\alpha \Big[ \Wbstpw\Htp(s,v,w) \Rtp(z,v)\Big]\, dA(v).
\]
We expand $Y^\alpha \Big[ \Wbstpw\Htp(s,v,w) \Rtp(z,v)\Big]$ and use Proposition
\ref{prop:good T lemma} to see
\[
Y^\alpha \Big[ \Wbstpw\Htp(s,v,w) \Rtp(z,v)\Big] = \hspace{-30pt}
\sum_{\atopp{n_1+n_2+n_3=n}{\atopp{\ell_1+\ell_2
+\sum_{j=1}^{n_3}\hspace{-2pt}|\beta_j| \hspace{2pt}=\ell}
{\alpha_1\in(n_1,\ell_1), \alpha_2\in (n_2,\ell_2)}}} \hspace{-30pt}
c_{\alpha_1,\alpha_2,\beta_j} Y^{\alpha_1} \Wbstpw \Htp(s,v,w) Y^{\alpha_2}\Rtp(z,v)
\Big(\prod_{j=1}^{n_3} D^{\beta_j}_{w,z} r(w,v,z)\Big).
\]
It is enough to take one term from the sum and estimate its integral over $\C$ (in $v$). 

First, assume that $|z-w| \geq 2\mu_p(w,1/\tau)$ and
$|z-w| \geq 2\mu_p(z,1/\tau)$. We decompose the
integral into four pieces: near $w$, near
$z$, and away from $z$ and $w$ (will be two integrals). For $x\in\C$, 
let $\vp_x \in C^\infty(\C)$ so that $\vp_x \equiv 1$ on $D\big(x,\mu_p(x,1/\tau)/2\big)$,
$\supp\vp_x \subset D(x,\mu_p(x,1/\tau))$, $0\leq \vp_x\leq 1$, and 
$|D^\alpha \vp_x| \les |\mu_p(x,1/\tau)|^{-|\alpha|}$
with constants independent of $x$. Note that 
$\supp \vp_z \cap \supp \vp_w = \emptyset$.

Near $z$, we use the $\Rtp$-cancellation conditions from Corollary \ref{cor:R cancel}. 
Our first estimate is:
\begin{align}
& \left| \int_\C Y^{\alpha_1} \Wbstpw \Htp(s,v,w) Y^{\alpha_2}\Rtp(z,v)
\Big(\prod_{j=1}^{n_3} D^{\beta_j}_{w,z} r(w,v,z)\Big) \vp_z(v) 
\big(1-\vp_w(v)\big)  \, dA(v) \right| \nn\\
&= \left| \int_\C Y^{\alpha_1} \Wbstpw \Htp(s,v,w) Y^{\alpha_2}\Rtp(z,v)
\Big(\prod_{j=1}^{n_3} D^{\beta_j}_{w,z} r(w,v,z)\Big) \vp_z(v) \, dA(v) \right|\nn\\
&= \Big| Y^{\alpha_2}\Rtp\Big[ Y^{\alpha_1} \Wbstpw \Htp(s,\cdot,w)
\Big(\prod_{j=1}^{n_3} D^{\beta_j}_{w,z} r(w,\cdot,z)\Big) \vp_z\Big](z) \Big|
\label{eqn:R est in I}
\end{align} 
The cases $\ell_2$ even and $\ell_2$ odd are estimated similarly, and we will only show
the $\ell_2 = 2k-1$ odd case. Recall that $\Rtp = \Zstp G_{\tau p}$, so by 
Corollary \ref{cor:R cancel}, we can estimate (\ref{eqn:R est in I}) by
\begin{multline}\label{eqn:R cancel applied}
\frac{C_{|\alpha_2|}}{\tau^{n_2}} \Big( \delta \Big\| \Box_{\tau p,v}^k\Big(
Y^{\alpha_1} \Wbstpw \Htp(s,\cdot,w)
\big(\prod_{j=1}^{n_3} D^{\beta_j}_{w,z} r(w,\cdot,z)\big) \vp_z\Big) \Big\|_{L^2(\C)}
\\ + \delta^3 \Big\| \Box_{\tau p,v}^{k+1}\Big(
Y^{\alpha_1} \Wbstpw \Htp(s,\cdot,w)
\big(\prod_{j=1}^{n_3} D^{\beta_j}_{w,z} r(w,\cdot,z)\big) \vp_z\Big) \Big\|_{L^2(\C)}\Big). 
\end{multline}
The two terms in (\ref{eqn:R cancel applied}) are estimated in the same fashion, 
and we only present the estimate of the first term.
\begin{align}
&\Box_{\tau p,v}^k\Big(
Y^{\alpha_1} \Wbstpw \Htp(s,v,w)
\big(\prod_{j=1}^{n_3} D^{\beta_j}_{w,z} r(w,v,z)\big) \vp_z(v)\Big)\nn\\
&= \sum_{|\gamma_0|+\cdots + |\gamma_{n_3+1}|=2k} c_{\gamma_j}\,
X^{\gamma_0}Y^{\alpha_1}\Wbstpw\Htp(s,v,w) \big(\prod_{j=1}^{n_3} D^{\gamma_j}_v
 D^{\beta_j}_{w,z} r(w,v,z)\big)
D^{\gamma_{n_3+1}}\vp_z(v) \label{eqn:Box expansion in R cancel}
\end{align}
It is enough to estimate the $L^2$ norm of an arbitrary term in the expansion of 
(\ref{eqn:Box expansion in R cancel}). On $\supp\vp_z$, note that
$|w-v|\sim |z-w|$ and $\mu_p(v,1/\tau)\sim \mu_p(z,1/\tau)$. Also, 
since $|v-w|\sim|z-w|\geq \mu_p(z,1/\tau)$,
we can interchange $s$ with $\mu_p(z,1/\tau)$ at will because of our exponential factors (though
we may have to decrease $c$ with each substitution).
Then 
\begin{align*}
&\frac{\mu_p(z,1/\tau)}{\tau^{n_2}} \Big\|
X^{\gamma_0}Y^{\alpha_1}\Wbstpw\Htp(s,\cdot,w) 
\Big(\prod_{j=1}^{n_3} D^{\gamma_j}_v D^{\beta_j}_{w,z} r(w,\cdot,z)\Big)
D^{\gamma_{n_3+1}}\vp_z \Big\|_{L^2(\C)} \\
&\les \frac{\mu_p(z,1/\tau)}{\tau^{n_2}} \bigg( \int_{\supp\vp_z} \hspace{-20pt}\tau^{-2n_1}
\frac{e^{-c\frac{|v-w|^2}s} e^{-c \frac{s}{\mu_p(v,1/\tau)^2}} 
e^{-c \frac{s}{\mu_p(w,1/\tau)^2}}} {s^{2+|\gamma_0|+\ell_1+1}}  
\prod_{j=1}^{n_3} \big|D^{\gamma_j}_v D^{\beta_j}_{w,z} r(w,v,z)\big|^2
\frac{1}{\mu_p(z,1/\tau)^{2\gamma_{n_3+1}}}\, dv \bigg)^{1/2} \\
&\les \frac{s}{\tau^{n}}\frac{1}{s^{1+ \frac12(|\gamma_0|+\cdots+|\gamma_{n_3+1}| + \ell_1
+ |\beta_1|+\cdots + |\beta_{n_3}|+ 1)}} e^{-c \frac{|z-w|^2}s}
e^{-c \frac{s}{\mu_p(z,1/\tau)^2}} e^{-c \frac{s}{\mu_p(w,1/\tau)^2}}\\ 
&= \frac{1}{\tau^n s^{1+ \frac 12\ell}} e^{-c \frac{|z-w|^2}s}
e^{-c \frac{s}{\mu_p(z,1/\tau)^2}} e^{-c \frac{s}{\mu_p(w,1/\tau)^2}}, 
\end{align*} 
since $\ell_1+\sum|\beta_j| = \ell-\ell_2$ and 
$|\gamma_0|+\cdots +|\gamma_{n_3+1}| = 2k = \ell_2+1$.
This is (better than) the desired estimate.

We begin the estimate for region near $w$. We first find the estimate for the
case $s^{1/2} \leq \mu_p(w,1/\tau)$.
We can assume that $\ell_1$ is odd because the $\ell_1$ even case is handled analogously.
If we set $\delta = \mu_p(w,1/\tau)$,
by Theorem \ref{thm:Boxtp cancellation conditions} we estimate:
\begin{align*}
\bigg| \int_\C Y^{\alpha_1} \Wbstpw \Htp(s,v,w) &Y^{\alpha_2} \Rtp(z,v)
\Big( \prod_{j=1}^{n_3} D^{\beta_j}_{w,z} r(w,v,z)\Big) \vp_w(v)\, dA(v)\bigg| \\
&=\Big| Y^{\alpha_1}\Wbstpw (\Htp^s)\sh \Big[ Y^{\alpha_2} \Rtp(z,\cdot)
\Big( \prod_{j=1}^{n_3} D^{\beta_j}_{w,z} r(w,\cdot,z)\Big) \vp \Big](w) \Big| \\
&\les \frac{1}{\tau^{n_1}\delta} \Big(
\Big\| (\Boxtp\sh)^{\frac{\ell_1+1}2} \Big(Y^{\alpha_2} \Rtp(z,\cdot) 
\big(\prod_{j=1}^{n_3} D^{\beta_j}_{w,z} r(w,\cdot,z)\big) \vp_w\Big) \Big\|_{L^2} \\
&+ \delta^2 \Big\| (\Boxtp\sh)^{\frac{\ell_1+3}2} \Big(Y^{\alpha_2} \Rtp(z,\cdot) 
\big(\prod_{j=1}^{n_3} D^{\beta_j}_{w,z} r(w,\cdot,z)\big) \vp_w\Big) \Big\|_{L^2}\Big).
\end{align*}
The two terms are handled similarly. We estimate the first term.
\begin{multline*}
\frac{1}{\tau^{n_1}\delta}
\Big\| (\Boxtp\sh)^{\frac{\ell_1+1}2} \Big(Y^{\alpha_2} \Rtp(z,\cdot) 
\big(\prod_{j=1}^{n_3} D^{\beta_j}_{w,z} r(w,\cdot,z)\big) \vp_w\Big) \Big\|_{L^2(\C)}\\
\les \frac 1{\tau^{n_1}\delta} \sum_{|\gamma_0|+\cdots+|\gamma_{n_3}+1|=\ell_1+1}
\big\| X^{\gamma_0} Y^{\alpha_2} \Rtp(z,\cdot) 
\big(\prod_{j=1}^{n_3} D^{\beta_j}_{w,z} r(w,\cdot,z)\big) 
D^{\gamma_{n_3}+1}\vp_w \big\|_{L^2(\C)}.
\end{multline*}
We pick an arbitrary term from the sum to bound. On $\supp \vp_w$, note that
$|v-z| \sim |w-z|$ and $\mu_p(v,1/\tau)\sim \mu_p(w,1/\tau)$. 
\begin{align*}
&\big\| X^{\gamma_0} Y^{\alpha_2} \Rtp(z,\cdot) 
\big(\prod_{j=1}^{n_3} D^{\beta_j}_{w,z} r(w,\cdot,z)\big) 
D^{\gamma_{n_3}+1}\vp_w \big\|_{L^2(\C)} \\
& \frac{1}{\tau^{n_1}\delta} \bigg(
\int_{\supp\vp_w} \frac1{\mu_p(z,1/\tau)^{2+2|\gamma_0|+2\ell_2} \tau^{2n_2}}
e^{-c \frac{|z-v|}{\mu_p(v,1/\tau)}}e^{-c \frac{|z-v|}{\mu_p(z,1/\tau)}}
\prod_{j=1}^{n_3} |D^{\beta_j}_{w,z}r(w,v,z)|^2 
\frac{1}{\delta^{2 \gamma_{n_3+1}}}\, dv \bigg)^{1/2}\\
&\les \frac{1}{\tau^n} \frac{1}{\mu_p(z,1/\tau)^{1+\ell_2+|\gamma_0|+\cdots+|\gamma_{n_3+1}| + 
|\beta_1|+\cdots+|\beta_{n_3}|}} e^{-c \frac{|z-w|}{\mu_p(w,1/\tau)}}
e^{-c \frac{|z-w|}{\mu_p(z,1/\tau)}}\\
&= \frac{1}{\tau^n \mu_p(z,1/\tau)^{2+\ell}} e^{-c \frac{|z-w|}{\mu_p(w,1/\tau)}}
e^{-c \frac{|z-w|}{\mu_p(z,1/\tau)}}
\end{align*}
since $|\gamma_0|+\cdots+|\gamma_{n_3+1}| = \ell_1+1$ and $\ell_1+\ell_2+\sum|\beta_j| = \ell$.
This is the desired estimate in the case $s^{\frac 12} \leq \mu_p(w,1/\tau)$. If 
$s^{\frac 12} \geq \mu_p(w,1/\tau)$, our estimate follows from size estimates. Indeed,
\begin{align}
&\bigg| \int_\C Y^{\alpha_1} \Wbstpw \Htp(s,v,w) Y^{\alpha_2}\Rtp(z,v)
\Big(\prod_{j=1}^{n_3} D^{\beta_j}r(w,v,z)\Big)\vp_w(v)\, dA(v)\bigg| \nn\\
&\les\frac{\mu_p(w,1/\tau)^2}{\tau^{n_1}s^{1+\frac 12\ell_1 + \frac 12}}
e^{-c\frac s{\mu_p(w,1/\tau)^2}}
e^{-c\frac{|z-w|}{\mu_p(z,1/\tau)}}e^{-c\frac{|z-w|}{\mu_p(w,1/\tau)} }
\frac{1}{\tau^{n_2}\mu_p(w,1/\tau)^{1+\ell_2}} 
\frac{1}{\tau^{n_3}\mu_p(w,1/\tau)^{\sum|\beta_j|}}\nn\\
&= \frac{1}{\tau^n s^{1+\frac 12\ell}} e^{-c\frac s{\mu_p(w,1/\tau)^2}}
e^{-c\frac{|z-w|}{\mu_p(z,1/\tau)}}e^{-c\frac{|z-w|}{\mu_p(w,1/\tau)} }. \label{eqn: Gw II est}
\end{align}

The remaining two estimates simply use the size conditions from 
Theorem \ref{thm:Boxtp heat estimates}
and Corollary \ref{cor:Rtp estimates}. The third integral we estimate is on the region
$|v-w|\geq |v-z|$. On this region, $|v-w| \geq \frac 12|z-w|$, so
\begin{align*}
&\bigg| \int_{|v-w|\geq|v-z|} Y^{\alpha_1}\Wbstpw \Htp(s,v,w) Y^{\alpha_2}\Rtp(z,v)
\Big(\prod_{j=1}^{n_3} D^{\beta_j}_{w,z} r(w,v,z)\Big) \big(1-\vp_z(v)\big)\big(1-\vp_w(v)\big)\, dA(v)
\bigg|\\
&\les \int_{|v-w|\geq|v-z|} \frac{1}{\tau^{n_1}s^{1+\frac 12(\ell_1+1)}}
e^{-c\frac{|v-w|^2}s} e^{-c\frac{s}{\mu_p(v,1/\tau)^2}}e^{-c\frac{s}{\mu_p(w,1/\tau)^2}}
\frac{1}{\mu_p(v,1/\tau)^{1+\ell_2}}
\Big|\prod_{j=1}^{n_3} D^{\beta_j}_{w,z} r(w,v,z)\Big|\, dA(v) \\
&\les\frac{1}{\tau^n}\frac{1}{s^{1+\frac 12(\ell_1+\ell_2)}} 
\frac{1}{s^{\frac 12(|\beta_1|+\cdots|\beta_{n_3}|)}}
e^{-c\frac{|z-w|^2}s} e^{-c\frac{s}{\mu_p(w,1/\tau)^2}} \int_{\C} \frac 1s 
e^{-c \frac{|v-w|^2}{s}}\, dA(v) \\
&\les \frac{1}{\tau^n s^{1+\frac 12\ell}}e^{-c\frac{|z-w|^2}s} 
e^{-c\frac{s}{\mu_p(w,1/\tau)^2}}. 
\end{align*}

The final integral is over the region $|v-w|\leq |v-z|$. In this case,
$|v-z| \geq \frac 12|w-z|$. We estimate
\begin{align*}
&\bigg| \int_{|v-w|\leq|v-z|} Y^{\alpha_1}\Wbstpw \Htp(s,v,w) Y^{\alpha_2}\Rtp(z,v)
\Big(\prod_{j=1}^{n_3} D^{\beta_j}_{w,z} r(w,v,z)\Big) 
\big(1-\vp_z(v)\big)\big(1-\vp_w(v)\big)\, dA(v)
\bigg|\\
&\les \int_{|v-w|\leq|v-z|} \frac{e^{-c\frac{|v-w|^2}s}}{\tau^{n_1}s^{1+\frac 12(\ell_1+1)}}
 e^{-c\frac{s}{\mu_p(v,1/\tau)^2}}e^{-c\frac{s}{\mu_p(w,1/\tau)^2}}
\frac{e^{-c \frac{|z-v|}{\mu_p(v,1/\tau)}}
e^{-c \frac{|z-v|}{\mu_p(z,1/\tau)}}}{\mu_p(v,1/\tau)^{1+\ell_2}} 
\Big|\prod_{j=1}^{n_3} D^{\beta_j}_{w,z} r(w,v,z)\Big|\, dA(v) \\
&\les \frac{1}{\tau^n}\frac{1}{\mu_p(z,1/\tau)^{\ell_1+\ell_2+1+\sum |\beta_j|}}
e^{-c \frac{s}{\mu_p(w,1/\tau)^2}} e^{-c \frac{|z-w|}{\mu_p(z,1/\tau)}}
\int_{\C} \frac 1s e^{-c \frac{|v-w|^2}s}\, dA(v) \\
&\les \frac{1}{\tau^n \mu_p(z,1/\tau)^{2+\ell}} e^{-c \frac{s}{\mu_p(w,1/\tau)^2}} 
e^{-c \frac{|z-w|}{\mu_p(z,1/\tau)}}.
\end{align*}
We have completed the estimates for the case $|z-w| \geq 2\mu_p(z,1/\tau)$ and
$2\mu_p(w,1/\tau)$. 

The cases $|z-w|\leq 2\mu_p(z,1/\tau)$ and $|z-w|\leq 2\mu_p(w,1/\tau)$ 
are similar to the estimates
already performed. An important feature of the near-diagonal estimate is that 
$\mu_p(z,1/\tau)\sim \mu_p(w,1/\tau)$.
Using size estimates and mimicking the techniques used earlier in this proof,
with a decrease in $c$ (to help turn $s$ into $\mu_p(w,1/\tau)$), we can show
\begin{align*}
\bigg| \int_\C Y^{\alpha_1}\Wbstpw \Htp(s,v,w) Y^{\alpha_2}\Rtp(z,v)
&\Big(\prod_{j=1}^{n_3} D^{\beta_j}_{w,z} r(w,v,z)\Big) 
\big(1-\vp_z(v)\big)\big(1-\vp_w(v)\big)\, dA(v)\bigg| \\
&\les \frac{1}{\mu_p(w,1/\tau)^{2+\ell}} e^{-c 
\frac{s}{\mu_p(w,1/\tau)^2}}e^{-c\frac{|z-w|}{\mu_p(z,1/\tau)}}.
\end{align*}
The estimation of the near $w$ integral
\begin{multline*}
 \int_\C Y^{\alpha_1}\Wbstpw \Htp(s,v,w) Y^{\alpha_2}\Rtp(z,v)
\Big(\prod_{j=1}^{n_3} D^{\beta_j}_{w,z} r(w,v,z)\Big) \big(1-\vp_z(v)\big)\vp_w(v)\, dA(v)\\
= Y^{\alpha_1}\Wbstpw (\Htp^s)\sh \Big[ \Rtp(z,\cdot)
\Big(\prod_{j=1}^{n_3} D^{\beta_j}_{w,z} r(w,\cdot,z)\Big) 
\big(1-\vp_z(\cdot)\big)\vp_w(\cdot)\Big](w),
\end{multline*}
proceeds as before with the $\Htp^s$-cancellation conditions and Theorem
\ref{thm:Boxtp cancellation conditions}.
Also, in the case that $s^{\frac 12} \geq \mu_p(w,1/\tau)$, 
the integral estimate  \eqref{eqn: Gw II est} 
suffices in the near-diagonal case. Finally, the
near $z$ integral
\begin{align*}
 \int_\C Y^{\alpha_1}\Wbstpw \Htp(s,v,w) Y^{\alpha_2}&\Rtp(z,v)
\Big(\prod_{j=1}^{n_3} D^{\beta_j}_{w,z} r(w,v,z)\Big) \big(1-\vp_z(v)\big)\vp_w(v)\, dA(v)\\
&= Y^{\alpha_2} \Rtp\Big[ Y^{\alpha_1}\Wbstpw \Htp(s,\cdot,w)
\Big(\prod_{j=1}^{n_3} D^{\beta_j}_{w,z} r(w,\cdot,z)\Big)
\big(1-\vp_w(\cdot)\big)\vp_z(\cdot)\Big](z),
\end{align*}
and the estimate follows from the $\Rtp$-cancellation condition, 
Corollary \ref{cor:R cancel}. The proof of theorem is complete with the observation that
\begin{equation}\label{eqn:min of Gaussian and time decays}
\min_{s\geq 0}  \frac{|z-w|^2}s + \frac{s}{\mu_p(z,1/\tau)^2} = \frac{|z-w|}{\mu_p(z,1/\tau)}. 
\end{equation}
which allows to pull the $e^{-c\frac{|z-w}{\mu_p(z,1/\tau)}} e^{-c\frac{|z-w}{\mu_p(w,1/\tau)}}$ out of the max.
\end{proof}

%%%%%%%%%%%%%%%%%%%%%%%
%
%	SECTION: SIZE ESTIMATES OF Y^J\Hwtp(s,z,w)
%
%
%%%%%%%%%%%%%%%%%%%%%%%%%
\section{Size Estimates of $Y^J\Hwtp(s,z,w)$ -- Proof of Theorem \ref{thm:H wig and deriv}}
\label{sec:H wig ests}
The estimation of $Y^J \Hwtp(s,z,w)$ follows from ideas we have already used and
Theorem \ref{thm:G wig and deriv}. 
%From \cite{Rai07}, we have:
%% Lemma: multiplication of e^{} maxes
%\begin{lem}\label{lem:mult maxes}
%Let $a, b >0$, $|z-w| \geq \mu_p(z,1/\tau)$, and $|z-w| \geq\mu_p(w,1/\tau)$.
%There exists a constant $C_{a,b}$ so that
%\begin{multline*}
%e^{-c \frac{|z-w|^2}{s}} \max \left\{ \frac{1}{s^{a}}, \frac{1}{\mu_p(z, 1/\tau)^{2a}} \right\}
%\cdot \max\left\{\frac{e^{-c\frac{s}{\mu_p(z,1/\tau)^2}}}{s^{b}}, 
%\frac{e^{-c\frac{|z-w|}{\mu_p(z,1/\tau)}}}{\mu_p(z, 1/\tau)^{2b}} \right\} \\
%\le C_{a,b} 
%e^{-c \frac{|z-w|^2}{s}}\max\left\{\frac{e^{-c\frac{s}{\mu_p(z,1/\tau)^2}}}{s^{a+b}}, 
%\frac{e^{-c\frac{|z-w|}{\mu_p(z,1/\tau)}}}{\mu_p(z, 1/\tau)^{2(a+b)}} \right\}.
%\end{multline*}
%with a possible decrease in $c$.
%\end{lem}

\begin{proof}[Proof. (Theorem \ref{thm:H wig and deriv})]
We know $Y^J\Hwtp(s,z,w) = Y^J \Gwtp(s,z,w) + Y^J \Stp(z,w)$.  From 
Theorem \ref{thm:G wig and deriv} and Corollary \ref{cor:R cancel} and (\ref{eqn:min of Gaussian and time decays}),
we have the bound
\begin{align}
|Y^J \Hwtp(s,z,w)| &\les  \frac{1}{\tau^n} \max\left\{ 
\frac{e^{-c \frac{|z-w|^2}s} e^{-c\frac{s}{\mu_p(w,1/\tau)^2}}e^{-c\frac{s}{\mu_p(z,1/\tau)^2}}}
{s^{1+ \frac 12\ell}},
\frac{ 
e^{-c\frac{|z-w|}{\mu_p(w,1/\tau)}}e^{-c\frac{|z-w|}{\mu_p(z,1/\tau)}}}{\mu_p(w,1/\tau)^{2+\ell}}
\right\} \\
&\leq  \frac{1}{\tau^n}e^{-c\frac{|z-w|}{\mu_p(w,1/\tau)}}e^{-c\frac{|z-w|}{\mu_p(z,1/\tau)}}\max\left\{ 
\frac{1} {s^{1+ \frac 12\ell}},
\frac{ 1}{\mu_p(w,1/\tau)^{2+\ell}}
\right\}.
\label{eqn:bad H wig bound}
\end{align}
We will move the Gaussian decay term outside of the brackets 
using Lemma \ref{lem: integration by parts}. 

Fix $(s,z,w)$ and $\tau>0$. 
As in the proof of Theorem \ref{thm:Boxtp heat estimates}, we let
$\delta = \frac 12 \min\{s^{\frac 12},\mu_p(w,1/\tau),\mu_p(z,1/\tau)\}$ and
$\gamma = \frac 14\tau$. Set
$F(z,w,\tau) = \Hwtp(s,z,w)$. Then $\|Y^K\Hwtp\|_{L^\infty(B)} \sim |Y^K\Hwtp(s,z,w)|$. We
use the bound for
$|Y^K\Hwtp(s,z,w)|$ from   (\ref{eqn:bad H wig bound}) and the bound for  $|\Hwtp(s,z,w)|$ from
the  known $(0,\ell)$-case of Theorem \ref{thm:H wig and deriv}. 
If $\delta = \frac 12 \mu_p(w,1/\tau)$, then
\begin{align*}
&|Y^J\Hwtp(s,z,w)|\les \max_{K\in(k,j)\leq(2n+2,2\ell+8)}
\delta^{\frac j2-\ell}\tau^{\frac k2-n} |Y^K \Hwtp(s,z,w)|^{\frac 12} |\Hwtp(s,z,w)|^{\frac 12}\\
&\les \max_{K\in(k,j)\leq(2n+2,2\ell+8)} \frac{1}{\tau^{k/2}}
 \delta^{\frac j2-\ell}\tau^{\frac k2-n} e^{-c\frac{|z-w|}{\mu_p(w,1/\tau)}}e^{-c\frac{|z-w|}{\mu_p(z,1/\tau)}}
\max\left\{\frac{1} {s^{\frac 12+\frac j4}},
\frac{1} {\mu_p(w,1/\tau)^{1+j/2}}\right\}\\
&\times e^{-c\frac{|z-w|^2}s}
\max\left\{\frac{1}{s^{\frac 12}},
\frac{1} {\mu_p(w,1/\tau)}\right\}\\
&\les \frac 1{\tau^n} e^{-c\frac{|z-w|^2}s}e^{-c\frac{|z-w|}{\mu_p(w,1/\tau)}}e^{-c\frac{|z-w|}{\mu_p(z,1/\tau)}}
\max\left\{\frac{1}{s^{1 + \frac \ell2}},
\frac{1}{\mu_p(w,1/\tau)^{2+\ell}}\right\}.
\end{align*}
This is the desired estimate in the case $\mu_p(z,1/\tau)\geq s^{1/2}$.

The final case is when $s^{1/2}\leq \mu_p(w,1/\tau)$.
Consider the following:
we know the result holds if $Y^J\in(0,\infty)$. Assume the result holds if
$Y^J\in(n-1,\infty)$. Let $\vp\in\cic{B(z,\delta)}$ so that $\vp\equiv 1$ on $B(z,\delta/2)$ and
$|\nabla^k\vp|\leq c_k/\delta^k$ for $k\leq 3$. Choose $\delta<\Delta$ small enough
so that $|Y^K\Htp(s,\xi,w)|\sim |Y^K\Htp(s,z,w)|$ if $Y^K = X^\alpha Y^J$ and $|\alpha|\leq 2$.
By argument leading up to \eqref{eqn:exp M flip flop}, it follows that
\[
\Mtpzw Y^J \Hwtp(s,z,w) = \Mtp^{z,w}\big(Y^J \Hwtp(s,z,w)\, \vp(z)\big)
= \lim_{\ep\to 0} e^{-\ep\Boxtp}\big[\Mtp^{z,\cdot}\big(Y^J \Hwtp(s,\cdot,w)\, \vp(\cdot) \big)\big](z).
\]
Let $\tilde\nabla = \big( \Zstp+\Zbstp, i(\Zstp-\Zbstp)\big)$.
By Theorem \ref{thm:Boxtp cancellation conditions} and the inductive hypothesis, we have
\begin{align*}
|\Mtpzw Y^J&\Hwtp(s,z,w)|
\les \lim_{\ep\to 0} \frac{\Lambda(z,\Delta)}{\delta}
\Big( \|\Hwtp(s,\cdot,w)\vp\|_{L^2} + \delta^2\|\Boxtp\big(\Hwtp(s,\cdot,w)\big)\vp\|_{L^2}
+ \delta^2 \|\Hwtp(s,\cdot,w)|\nabla^2\vp|\|_{L^2}\Big) \\
&\les \lim_{\ep\to 0}\Lambda(z,s^{1/2})\big(|Y^J\Hwtp(s,z,w)| 
+ \delta |\tilde\nabla Y^J\Hwtp(s,z,w)| + \delta^2|\Boxtpz Y^J\Hwtp(s,z,w)|\big)\\
&\les \Lambda(z,s^{1/2}) e^{-c\frac{|z-w|^2}s} \Lambda(z,s^{1/2})^{n-1} 
e^{-c\frac{|z-w|}{\mu_p(w,1/\tau)}} e^{-c\frac{|z-w|}{\mu_p(z,1/\tau)}}
\max\left\{ \frac{1}{s^{1 + \frac \ell2}}, \frac{1}{\mu_p(w,1/\tau)^{2+\ell}}\right\}.
\end{align*}
By the argument leading to Remark \ref{rem:first int ok}, it is enough  to only 
check $(n,\ell)$-derivatives of the form $\Mtp Y^J$. 
\end{proof}

\bibliographystyle{plain}
\bibliography{mybib}

\end{document}